\documentclass[12pt]{amsart}
 
\usepackage{blindtext}
\usepackage[T1]{fontenc}
\usepackage[utf8]{inputenc}
\usepackage{graphicx}
\usepackage{amssymb, mathrsfs, url, amsfonts, amsthm, amsmath}
\usepackage{esint}
\usepackage{enumerate}
\usepackage{xcolor}
\usepackage{etoolbox}
\usepackage{float}
\usepackage{hyperref}
\usepackage{standalone}
\usepackage{tikz}
\usepackage{subcaption}
\usepackage{varioref}
\usepackage[margin=2cm]{geometry}
\usepackage[english]{babel}
\usepackage{comment}
\usepackage{natbib}
\usepackage{mathrsfs}

\newtheorem{theorem}{Theorem}[section]
\newtheorem{lemma}[theorem]{Lemma}
\newtheorem{corollary}[theorem]{Corollary}
\newtheorem{proposition}[theorem]{Proposition}

\theoremstyle{definition}

\theoremstyle{remark}
\newtheorem{remark}[theorem]{Remark}

\let\inf\relax \DeclareMathOperator*\inf{\vphantom{p}inf}

\newcommand{\N}{\textbf N}
\newcommand{\Euler}{\text{Euler}}
\newcommand{\id}{\text{Id}}
\newcommand{\cL}{\mathcal L}
\newcommand{\sL}{\mathscr L}
\newcommand{\sG}{\mathscr G}
\newcommand{\sF}{\mathscr F}
\newcommand{\cF}{\mathcal F}

\newcommand{\sD}{\mathscr D}
\newcommand{\tr}{\text{tr}}

\newcommand{\cC}{\mathcal{C}}

\newcommand{\cQ}{\mathcal{Q}}

\newcommand{\cP}{\mathscr{P}}
\newcommand{\cB}{\mathcal{B}}
\newcommand{\dom}{\text{dom}}

\newcommand{\bP}{\mathbb P}

\newcommand{\BBT}{\mathbb{T}}
\newcommand{\be}{\begin{equation}}
\newcommand{\ee}{\end{equation}}
\newcommand{\p}{\partial}
\newcommand{\R}{\mathbb{R}}
\newcommand{\Ve}{\vec{\mathbf{e}}}

\newcommand{\ep}{\varepsilon}

\numberwithin{equation}{section}

\title[Unstable Manifolds of Stratified Euler Equations]{Unstable Manifolds of Stratified Euler Equations}
\author[Lin]{Zhiwu Lin}
\address[Z. Lin]{School of Mathematical Sciences, Fudan University, Shanghai 200433, China}
\email{zwlin@fudan.edu.cn}

\author[Wang]{Yanbo Wang}
\address[Y. Wang]{School of Mathematics, Georgia Institute of Technology, Atlanta, GA 30332}
\email{ywang4134@gatech.edu}

\author[Zeng]{Chongchun Zeng}
\address[C. Zeng]{School of Mathematics, Georgia Institute of Technology, Atlanta, GA 30332}
\email{zengch@math.gatech.edu}

\begin{document}

\begin{abstract}

We consider a spectrally unstable steady state $(\rho_0,v_0)$ of the
incompressible stratified Euler equations on a class of $d$-dimensional
domains. Assuming that the linearized equation admits an exponential
dichotomy 
with a reasonably large spectral gap
relative to the
maximal Lyapunov exponent of the background steady flow $v_0$, we construct the local stable and 
unstable manifolds of $(\rho_0,v_0)$. The proof is based on the Lyapunov--Perron
method after reformulating the Euler equation as an ODE on the
infinite-dimensional manifold of volume-preserving Lagrangian maps, with the density treated as a frozen Lagrangian parameter as well as the weight in the $L^2$ metric. We also discuss some applications
to two-dimensional steady flows.

\end{abstract}

\maketitle

\tableofcontents

\section{Introduction}

We consider the incompressible stratified Euler equations 
\begin{equation}
\label{stratified-euler}
\begin{cases}
\rho_t + v \cdot \nabla \rho = 0 & \text{ in } \Omega,\\
     \rho (v_t + v \cdot \nabla v) + \nabla p + \rho g \Ve_d=0 & \text{ in }\Omega,\\    
    \nabla \cdot v = 0 & \text{ in } \Omega,\\
    v \cdot \textbf N = 0 & \text{ on } \partial \Omega,
\end{cases}
\end{equation}
with the slip boundary condition on $\p \Omega$ if $\p \Omega\ne \emptyset$. Here $v= (v^1, \ldots, v^d)^T$ is the incompressible velocity field, $\rho$ the density, and $p$ the pressure. In particular, $-\Ve_d$ represents the ``vertical'' direction of the gravity and $g\ge 0$ the gravitational acceleration. The smooth domain $\Omega \subset \R^d$, with outward unit normal vector $\N$ along $\p \Omega$,  is assumed to be one of the two following cases:
\be \label{E:Omega}
\Omega = \BBT^{d_1} \times \R^{d_2}, \; d_2 >0, \ \text{ or } \ \Omega =  \mathbb{T}^{d_1} \times \Omega_2, \ \text{ where } \ \Omega_2 \subset\subset \R^{d_2}; \quad d_1 + d_2 =d\ge 2. 
\ee
Naturally, the direction of the gravity should not be a periodic direction, namely, 
\be \label{E:g} 
d_2 >0 \; \text{ if } \; g>0. 
\ee

The goal of this paper is to construct local stable and unstable manifolds of unstable equilibria of \eqref{stratified-euler} under a linear exponential dichotomy and a spectral gap condition.  

The study of the instability of a given steady flow involves two tasks.
The first is linear: often one first studies the spectral distribution and linear evolution. This depends on the particular steady state and belongs to the linear stability theory of inviscid stratified flows. If an exponential dichotomy and a sufficiently large exponential growth  gap hold for the linearized  equation, then the second task is nonlinear and geometric.  
One may prove the nonlinear instability, or more strongly, construct the local invariant manifolds in the infinite-dimensional phase space, which is the focus of this paper. Here the main challenges are the derivative loss of the Euler equation and the involvement of the density as well as its coupling to the pressure. 

The motivation of this work comes from two closely related classical topics in inviscid fluid dynamics.  The first is the classical linear and nonlinear instability theory of steady Euler flows, beginning with Rayleigh and Kelvin and continuing through the modern spectral theory of ideal fluids.  Even for homogeneous incompressible Euler equations, passing from an unstable eigenvalue to a nonlinear dynamical conclusion is delicate, because the equation is quasilinear and the transport term $v\cdot \nabla v$ produces a loss of derivatives.  Sufficient criteria for unstable eigenvalues 
have been developed in several important settings, especially for shear flows and rotating flows; see, for example, \cite{DH66, DR04, lin03, Lin04, Lin05} and the references therein.  Nonlinear instability results for homogeneous Euler equilibria based on unstable eigenvalues were obtained, for instance, in \cite{FSV97, Gr00, BGS02, VF03, Lin04}.  The local unstable manifolds of equilibria of the homogeneous Euler equation were constructed in \cite{lin-zeng-manifold} under an exponential dichotomy assumption and a gap condition involving the stretching rate of the steady Lagrangian flow.

The second motivation is the physical and mathematical instability theory of stratified flows.  Density stratification introduces buoyancy effects, internal waves, and density-dependent pressure coupling, all of which change the spectral and nonlinear dynamics.  Classical examples include Rayleigh--Taylor instability, Kelvin--Helmholtz instability, Holmboe instability, and more general  instabilities from the Taylor--Goldstein equation of stratified shear layers.  
In the inviscid parallel-flow setting, the Richardson number and the
Miles--Howard criterion are central stability mechanisms
\cite{LinC55, Miles61,Howard61,MilesHoward64,DR04}, while Howard’s semicircle theorem and
related bounds are discussed in \cite{Howard61,DH66,DR04}.
Numerical and physical studies of stratified shear layers, such as
\cite{Hazel72,CaulfieldPeltier00}, together with the broader account of
buoyancy effects in \cite{Turner73}, illustrate the connection between
linear instability, billow formation, secondary instability, and mixing
in laboratory and geophysical flows.
The Rayleigh--Taylor instability goes back to Rayleigh
\cite{Ray1883} and Taylor \cite{Tay50}; see also the classical
monograph \cite{Chand61} and the reviews \cite{Sharp84,Abarzhi10}.
For smooth Rayleigh--Taylor unstable density profiles, rigorous
mathematical results on growing modes and nonlinear instability of
non-homogeneous incompressible Euler flows were obtained in
\cite{CLR01,La01,HL03,HG03,La08}.
When a background shear flow is present, the spectral picture is more
subtle: unstable essential spectrum may arise in stratified shear flows;
see \cite{essential-spectrum,Shvydkoy-Latushkin} for general advective
mechanisms and \cite{roman} for the Rayleigh--Taylor problem in the
presence of background shear. Related nonlinear instability and
stability issues for stratified shear flow were studied in
\cite{Fr01}.

Beyond parallel shear flows, stratified fluids also exhibit important
instabilities of coherent vortical structures. Examples include the
zigzag instability of vertical vortex pairs in strongly stratified fluids
\cite{BillantChomaz00a,BillantChomaz00b,BillantChomaz00c,
OtheguyChomazBillant06,OtheguyBillantChomaz07,
DeloncleBillantChomaz08,WaiteSmolarkiewicz08}. Elliptic instability has
a classical origin in the instability of elliptic vortices and elliptic
flows \cite{Kida81,Pierrehumbert86,Bayly86}; in rotating or stratified
settings, related mechanisms involve resonant interactions with
inertia--gravity waves
\cite{MiyazakiFukumoto97,AspdenVanneste09,
GuimbardLeDizesLeBarsLeGalLeblanc10}. Most of these works are formulated
in Boussinesq, rotating, viscous, or asymptotic models rather than exactly
in \eqref{stratified-euler}, and hence they are not direct applications
of the present theorem. Nevertheless, they illustrate the physical
relevance of constructing invariant manifolds near linearly unstable
stratified steady or relative steady flows, especially when the
instability is carried by isolated coherent modes rather than by
short-wave essential spectrum.


The present paper addresses the second task -- obtaining nonlinear local unstable and stable manifolds based on the linear properties -- in a general form and then verifies the linear assumptions -- the first task -- in two-dimensional examples.
In this sense we shall prove a nonlinear phase-space theorem: once a sufficiently separated unstable spectral component is known for a stratified Euler steady state, the theorem identifies the set of nearby solutions converging to this equilibrium exponentially as $t\to \pm \infty$ and proves that this set is a local invariant manifold. Under certain conditions, such as the finite dimensionality of the unstable subspace, it yields nonlinear instability in a stronger sense that solutions starting in arbitrarily small neighborhoods of the equilibrium under a stronger Sobolev norm may grow in a weaker Sobolev norm. Thus this result should be interpreted as a rigorous description of the early inviscid
stage of the nonlinear departure from equilibria.

\noindent
$\bullet$ {\bf Equilibrium.} 
Suppose that $(v_0,\rho_0, p_0)$ is a smooth steady state of \eqref{stratified-euler} such that the following are satisfied: 
\begin{subequations}\label{E:steady}
\begin{equation}
    \label{boundedness-of-density}
    0< \rho_- \triangleq \inf_{\bar{\Omega}} \rho_0 \le  \rho_+ \triangleq \sup_{\bar{\Omega}} \rho_0 < +\infty; \quad \nabla \rho_0 \in H^{k+r-1}, \ \text{ where } k> d/2+1, \; r\ge 4, 
\end{equation}
\begin{equation}
    \label{v-zero-rho-zero-requirements}
    D v_0, \nabla p_0 \in L^\infty, \;\; D^2 v_0 \in H^{k+r-2}, \;\; D^2 p_0  \in H^{k+r-1}, 
    \ \text{ where } k> d/2+1, \; r \geq 4.
\end{equation}
\end{subequations}   
We shall consider $H^k$ perturbations of the above steady state, where naturally the velocity perturbations belong to  
\be \label{E:H_Euler}
H^{k}_{\Euler}\triangleq
    \left\{v \in H^k\left(\Omega,\mathbb{R}^d\right) : \nabla \cdot v = 0 \text{ in } \Omega, \ v \cdot \N = 0 \text{ on } \partial \Omega \text{ (if $\p \Omega\ne \emptyset$)} \right\}.
\ee 

\noindent
$\bullet$ {\bf Linearization and exponential dichotomy assumption.} As usual the study of the local dynamics near the equilibrium $(v_0,\rho_0)$ starts with the linearization of \eqref{stratified-euler} at $(v_0,\rho_0)$: 
\begin{equation}
    \label{linearization}
    \begin{pmatrix}
  \delta \rho \\ \delta v 
    \end{pmatrix}_t 
    = 
    - v_0 \cdot \nabla\begin{pmatrix}
          \delta \rho \\ \delta v \end{pmatrix} + \begin{pmatrix} 
          - \delta v \cdot \nabla \rho_0 \\
          - \delta v \cdot \nabla v_0 -\frac{1}{\rho_0}\nabla \delta p + \frac{\delta\rho}{\rho_0^2} \nabla p_0        
    \end{pmatrix}  \triangleq L_0 \begin{pmatrix}
        \delta \rho \\ \delta v 
    \end{pmatrix} +  L_1 \begin{pmatrix}
        \delta \rho \\ \delta v
    \end{pmatrix} \triangleq L \begin{pmatrix}
                \delta \rho \\ \delta v
    \end{pmatrix},
\end{equation}
where the linearized pressure $ \delta p$ is determined by the linearized velocity $\delta v$ and the  density $\delta \rho$ nonlocally as 
\be \label{E:LPressure} \begin{cases} 
- \nabla \cdot \big( \frac 1{\rho_0} \nabla \delta p \big) = \nabla \cdot \big( v_0 \cdot \nabla \delta v +  \delta v \cdot \nabla v_0 -  \frac {\delta \rho}{\rho_0^2} \nabla p_0 \big) = 2 \text{tr} \big( (D v_0) (D \delta v) \big)   - \nabla \cdot \big( \frac {\delta \rho}{\rho_0^2} \nabla p_0 \big), \\
\N \cdot \frac 1{\rho_0} \nabla \delta p = -2 \delta v \cdot  \Pi v_0 + \N \cdot \frac {\delta \rho}{\rho_0^2} \nabla p_0, \qquad \text{ on } \ \p \Omega \ \text{ if } \ \p \Omega \ne \emptyset, 
\end{cases} \ee
to ensure $\delta v\in H_\Euler^k$. Here $\Pi=\Pi(x)$, $x \in \p \Omega$, is the second fundamental form of $\partial \Omega$, see \eqref{E:2ndFF}.

Since $v_0$ is a smooth divergence free vector field with a global Lipschitz constant, it defines a Lagrangian map $u_0 (t, \cdot) \in H^{k+r} (\Omega, \R^d)$ for all $t\in \R$ given by 
\begin{equation}
\label{lagrangian-coordinate-map-0}
  u_{0t}(t,y) = v_0(t,u_0(t,y)), \quad    u_0(0,y) = y,
\end{equation}
which satisfies that $u_0(t, \cdot): \Omega \to \Omega$ is diffeomorphic, $u_0(t, \p \Omega) = \p \Omega$, and $\det Du_0(t, x) \equiv 1$ for all $t \in \R$ (see Subsection \ref{SS:L-coord}).  Hence $L_0$ generates a strongly $C^0$ unitary group $e^{tL_0}$ on $L^2$ which is the transport operator $e^{tL_0} f = f \circ u_0(t, \cdot)^{-1}$. It is also a strongly $C^0$ group on $H^n$, $0 \le n \le k+r$. Using \eqref{E:LPressure}, one can see that $L_1$ is a bounded operator on $H^n$, $0 \le n \le k+r-1$, so $L$ also generates strongly $C^0$ groups on $H^n$, $0 \le n \le k+r-1$. 

We shall construct the local unstable manifolds in the space $(\rho, v) \in (\rho_0, v_0) +  H^k(\Omega,\mathbb{R}) \times H^k_\Euler$, where one notices that the equilibrium $(\rho_0, v_0)$ is assumed to have stronger regularity than the phase space. The following standard exponential dichotomy of the linearized operator $L$ is assumed.  
\begin{enumerate}[(ED)]
    \item    \label{A1}
    There exist closed subspaces $X_{u,cs} \subset  H^k (\Omega,\mathbb{R}) \times H^{k}_\Euler$ such that the following hold. 
    \begin{enumerate} 
    \item $H^k (\Omega,\mathbb{R}) \times H^{k}_\Euler= X_u \oplus X_{cs}$, $X_u \ne \{0\}$. 
    \item $L (X_{u,cs} \cap \dom(L)) \subset X_{u,cs}$, where we denote $L_{u,cs} = L|_{X_{u,cs}}$. 
    \item There exists $M>0,\lambda_u > \lambda_{cs} \geq 0$ such that the groups $e^{t L_{u,cs}}$  on $X_{u, cs}$ satisfy 
    $$
    \left\|e^{t  L_u}\right\|_{\cL(X_u)} \leq M e^{\lambda_u t},\; \forall t \leq 0; \quad \left\|e^{t L_{cs}}\right\|_{\cL(X_{cs})}\leq Me^{\lambda_{cs} t},\; \forall t \geq 0.
    $$
    \end{enumerate}
\end{enumerate}

\noindent
$\bullet$ {\bf Main results.} 
In dynamical systems, the unstable manifold of the steady state $(\rho_0, v_0)$ is the set of solutions which converge to $(\rho_0, v_0)$ at certain exponential rate as $t \to -\infty$. More precisely, we consider solutions satisfying 
\be \label{E:UnstableS-1}
\sup_{t \le 0} e^{-\lambda t} \big(\|v(t)- v_0\|_{H^k} + \|\rho(t) -\rho_0\|_{H^k} \big) \le \ep 
\ee
for some $\lambda \in (\lambda_{cs}, \lambda_u)$ and $\ep>0$. However, the standard approaches of the Lyapunov--Perron integral equation method or the graph transform method by Hadamard do not apply directly to \eqref{stratified-euler} due to the quasilinear nature of the Euler equation. Instead we work mainly in the Lagrangian formulation. Let 
\be \label{E:Lyap-E} \begin{split} 
\mu & > \max \big\{ \lim_{t\to + \infty} |t|^{-1} \log\|Du_0(t)\|_{L^\infty}, \, \lim_{t\to - \infty} |t|^{-1} \log\|Du_0(t)\|_{L^\infty}\big\} \\
& = \max \big\{ \inf_{t> 0} |t|^{-1} \log\|Du_0(t)\|_{L^\infty}, \, \inf_{t< 0} |t|^{-1} \log\|Du_0(t)\|_{L^\infty}\big\}
\triangleq \mu_0 \ge 0, 
\end{split} \ee
where the Lagrangian map $u_0(t)$ of the steady flow $v_0$ was defined in \eqref{lagrangian-coordinate-map-0}. The above equality of the limits, which are both bounded above by $\|D v_0\|_{L^\infty}$, is due to the sub-additivity of $\log \|Du_0(t)\|_{L^\infty}$ in $t$ and Fekete's lemma. Hence we have 
\be \label{u0-L-infinity} 
C_* \triangleq \sup_{t\in \R} e^{-\mu |t|}  \|Du_0(t,\cdot)\|_{L^{\infty}} 
< \infty. 
\ee  
This quantification of the exponential growth rate $\mu$ is needed when we transform back and forth between the Lagrangian and Eulerian formulations.

\begin{theorem}
    \label{main-theorem}
There exists $K_0 \in \mathbb N$ determined by $k+r$ such that, for any equilibrium $(\rho_0,v_0, p_0)$ of the  Euler equation \eqref{stratified-euler} satisfying \eqref{E:steady}, \eqref{u0-L-infinity}, \eqref{E:Lyap-E}, and the exponential dichotomy (ED) with 
\be \label{E:lambda-1} \begin{split} 
& \mathcal{I} \triangleq \big[\lambda_{cs} + 2K_0 \mu, \lambda_u- K_0\mu\big] \ne \emptyset, \\
& \exists r_1, \; 1\le r_1 \le k,  \; X_u \ \text{ is closed in } \ H^{k-r_1}(\Omega,\R)\times H_\Euler^{k-r_1}, 
\end{split} \ee
there exists a topological submanifold (with boundary) $W^u \subset (\rho_0, v_0) + H^k (\Omega, \R)  \times H_\Euler^k$ homeomorphic to a closed ball in $X_u$ satisfying the following properties. 
\begin{enumerate}
        \item For any $1\le r_0 \le \min\{r-2, r_1\}$, $W^u$ is a $C^{r_0}$ submanifold of $ (\rho_0, v_0) + H^{k-r_0} (\Omega, \R)  \times H_\Euler^{k-r_0}$ containing $(\rho_0, v_0)$ in its interior and the tangent space $T_{(\rho_0, v_0)}W^u = X_u$. 
        \item For any initial value $(\rho(0), v(0)) \in W^u$, there exists a unique solution $(\rho(t), v(t), p(t))$, $t \le 0$, to \eqref{stratified-euler} such that $\nabla (p(t) - p_0) \in H^k (\Omega, \R^d)$. Moreover, the following hold. 
        \begin{enumerate} 
        \item The solution can exit $W^u$ only through its boundary, namely, 
        \[
        (\rho(t_0), v(t_0)) \in \p W^u, \ \text{ if } \ t_0 \triangleq \inf\{t\le 0: (\rho(t'), v(t')) \in W^u, \, \forall t'\in [t, 0]\}> -\infty.
        \] 
        \item there exists $C \ge 1$ determined by $k+r, \rho_\pm, \Omega$, $C_*, \mu$, $\mu_0$, 
        $|\rho_0|_{Z^{k+r}}$, $|D v_0|_{Z^{k+r-1}}$, and $|\nabla p_0|_{Z^{k+r-1}}$, where the $|\cdot|_{Z^l}$ norm is defined in \eqref{E:Z^l},  such that   
        \[
        \left\|(\rho(t), v(t)) - (\rho_0, v_0)\right\|_{H^k} \leq Ce^{(\lambda_u-K_0 \mu)t}\|(\rho(0), v(0)) - (\rho_0, v_0)\|_{H^{k-r_1}},\; \forall \, t \leq 0.
        \]
        \end{enumerate}        
        \item If a solution $(\rho(t), v(t), p(t))$, $t \le 0$, to \eqref{stratified-euler} satisfies 
        \be \label{E:ExpDecay}
        \exists\lambda\in \mathcal{I}, \quad \sup_{t \le 0} e^{-\lambda t} \big(\|v(t)- v_0\|_{H^k} + \|\rho(t) -\rho_0\|_{H^k} \big) < \infty, 
        \ee
        then there exists $t_0 \le 0$, such that $(\rho(t), v(t)) \in W^u$ for all $t \le t_0$.     
\end{enumerate}
\end{theorem}

\begin{remark} \label{R:SM}
Instead, assume the linear exponential dichotomy (ED) with the linearly invariant decomposition $H^k (\Omega,\mathbb{R}) \times H^{k}_\Euler= X_{cu} \oplus X_{s}$ along with the bounds $\lambda_s < \lambda_{cu} \le 0$ on the exponential growth of $e^{tL_{cu, s}}$ and 
\[
\lambda_{cu} - \lambda_s > 3K_0\mu, \ \text{ and } \ \exists r_1, \; 1\le r_1 \le k,  \; X_s \ \text{ is closed in } \ H^{k-r_1}(\Omega,\R)\times H_\Euler^{k-r_1}. 
\]
Due to the reversibility $(\rho(t, x), v(t, x)) \to (\rho(-t, x), - v(-t, x))$ of solutions to the stratified Euler equation \eqref{stratified-euler}, $(\rho_0, -v_0)$ is also an equilibrium which satisfies the assumptions of Theorem \ref{main-theorem}. The unstable manifold of $(\rho_0, -v_0)$ immediately yields the {\it stable manifold} $W^s$ of $(\rho_0, v_0)$ by the same transformation. The stable manifold satisfies the analogous properties except for $t \ge 0$, where the  solutions converge to $(\rho_0, v_0)$ exponentially as $t\to +\infty$.   
\end{remark}

By reversing the time direction in the exponential decay estimate in Theorem \ref{main-theorem}(2b), a direct corollary is the stronger nonlinear instability property of the steady state $(v_0, \rho_0, p_0)$ with growth in the rougher $H^{k-r_1}$ norm of solutions with initial perturbations in $H^k$ space, along with a logarithmic exit-time estimate. 

\begin{corollary}
    \label{Hk-L2} 
    Under the assumptions of Theorem \ref{main-theorem}, there exists $C, \delta>0$ determined by $k+r, \rho_\pm, \Omega$, $C_*, \mu$,  $\mu_0$, $\lambda_u  -3K_0 \mu -\lambda_{cs}$, $|\rho_0|_{Z^{k+r}}$, $|D v_0|_{Z^{k+r-1}}$, and $|\nabla p_0|_{Z^{k+r-1}}$ such that for any initial value $(\rho(0),v(0)) \ne (\rho_0, v_0)$ on $W^u$, its solution $(\rho(t),v(t))$ satisfies 
\[
\sup \big\{ \|(\rho(t), v(t)) - (\rho_0, v_0)\|_{H^{k-r_1}} : 0 \le t \le \max \big\{0, \tfrac 1{\lambda_u- K_0 \mu} \log \tfrac {C\delta}{\|(\rho(0), v(0)) - (\rho_0, v_0)\|_{H^{k}}} \big\} \big\} \ge \delta. 
\]
\end{corollary}

Besides its implication on the growth in the weaker $H^{k-r_1}$ norm in Corollary \ref{Hk-L2}, the assumption that $X_u$ is closed in the weaker Sobolev space is mainly used to ensure the smoothness of $W^u$ when we transform it from the Lagrangian coordinates to the Eulerian coordinates. Without this assumption, the unstable manifold is still a $C^{r-2}$ manifold in the Lagrangian coordinates; see Theorem \ref{T:UM-L}. 

One observes that in \eqref{E:steady}, the steady velocity $v_0$ is not assumed to decay at spatial infinity if $\Omega$ is unbounded. Thus a larger class of steady flows 
with possible linear growth of $v_0$ at spatial infinity satisfy the above assumptions. 
The local well-posedness of the stratified Euler equation \eqref{stratified-euler} with initial data given as $H^k$ perturbations of such $(\rho_0, v_0)$ satisfying \eqref{E:steady} is provided in Section \ref{S:Lagrangian}.

Exponential dichotomy (ED) is both crucial and classical among the assumptions on the linearization. Its verification can be delicate and involves the spectral mapping property. When $\Omega$ is two dimensional, the exponential dichotomy (ED) of certain equilibria is discussed in Section \ref{S:2D}.

In addition to the usual condition $\lambda_u>\lambda_{cs}$ in the exponential dichotomy (ED), we also impose the additional assumption $\lambda_u - \lambda_{cs} > 3 K_0 \mu$. In particular, any $\mu>0$ satisfies \eqref{u0-L-infinity} and \eqref{E:Lyap-E} for shear flows and rotating flows and thus $\lambda_u>\lambda_{cs}$ is indeed sufficient in these cases.  Technically this additional assumption is needed to control the norms of various quantities when transforming the problem between Eulerian and Lagrangian formulations. 
More essentially, the linearized Euler operator may have nonempty unstable essential spectrum due to nontrivial Lyapunov exponents \eqref{E:Lyap-E} of the steady flow $v_0$; see, e.g., \cite{FV91, Lif91, Vishik1996SpectrumOS, SL03, Shvydkoy-Latushkin}.  Essential-spectrum instability often corresponds to short-wave or high-frequency amplification and may be sensitive to small viscosity or density diffusion.  By contrast, isolated unstable eigenvalues, especially those associated with large-scale or shear instabilities, describe coherent modes which often dominate the early inviscid departure from equilibrium.  Under the assumption $\lambda_u - \lambda_{cs} > 3K_0 \mu$, the unstable manifold constructed here captures the nonlinear phase-space structure associated with such modes.   

We now explain the proof strategy and its relation to existing invariant-manifold theory. Under the exponential dichotomy assumption, local invariant manifolds are classical for smooth maps, ODEs, and semilinear PDEs, and also for some quasilinear PDEs with smoothing properties; see, for example, \cite{CHOW1988285}.  More recently, related constructions have been developed for classes of quasilinear PDEs including irrotational water waves in \cite{SZ26, SY26}.  The Euler equation does not fall directly into these categories because the linearized flow does not provide enough regularity control to compensate for the derivative loss in the nonlinearity.  In the construction of the local unstable manifolds of \eqref{stratified-euler}, we follow the Lagrangian approach of \cite{lin-zeng-manifold}, but the stratified case introduces several new features.

In Section \ref{S:Lagrangian}, we develop the Lagrangian formulation and obtain the local well-posedness of solutions with initial data given as  $H^k$ perturbations of a general class of background stratified flow.  
When the density is non-uniform, it is transported by the fluid flow and therefore becomes time independent in Lagrangian coordinates.  It may thus be treated as an external parameter.  The inhomogeneous density also changes the pressure projection from the standard Leray projection to a density-weighted Hodge projection.  
With a careful treatment of the density, particularly when  $D v_0$ does not decay at spatial infinity in the case of unbounded $\Omega$, the Euler equation becomes an ODE on the infinite-dimensional manifold of volume-preserving diffeomorphisms on $\Omega$.  The equilibrium $(\rho_0,v_0)$ becomes a special time-dependent solution in this Lagrangian formulation. In local coordinates the nearby dynamics is transformed into a non-autonomous ODE on $(H_\Euler^k)^2$. In the current paper we adopt a different approach from \cite{lin-zeng-manifold} in the analysis of the dependence of the pressure on the Lagrangian map.  

In Subsection \ref{SS:ED-L}, we derive the exponential dichotomy of the non-autonomous linearized system in the Lagrangian formulation from the Eulerian dichotomy (ED).  This step is much more subtle than in \cite{lin-zeng-manifold} since (ED) is stated for the Eulerian variables $(\delta\rho,\delta v)$, while the Lagrangian unknowns are the flow map and the Lagrangian velocity. Moreover the relabeling symmetry transformation, which plays an essential role in the connection between the Eulerian and Lagrangian formulations, does not leave the Lagrangian density invariant and causes additional complications. 
With the linear analysis accomplished, we then construct the local unstable manifold in the Lagrangian formulation by the Lyapunov--Perron approach in Subsection \ref{SS:InMa-L}. The local unstable manifold $W^u$ in the Eulerian variables is recovered in Subsection \ref{SS:InMa-E}. It is worth pointing out that the transformation between the Lagrangian and Eulerian formulation has certain regularity issues, see Remark \ref{R:non-smooth}, which is why $W^u$ is a smooth submanifold only in $ (\rho_0, v_0) + H^{k-r_0} (\Omega, \R)  \times H_\Euler^{k-r_0}$. 

In Section \ref{S:2D}, we discuss two classes of two-dimensional steady flows in perturbation settings, focusing on the most delicate exponential dichotomy (ED) property.  The first consists of steady flows with small density stratification, viewed as perturbations of homogeneous Euler steady flows. Even at the linear level, this is not a small regular perturbation problem when the unbounded transport operator is perturbed and delicate linear analysis is required. 
The second consists of shear flows with $\rho_0'(x_2)>0$ exhibiting Rayleigh--Taylor instability, analyzed using the spectral results in \cite{essential-spectrum,Shvydkoy-Latushkin, roman}.  In both cases, sufficient conditions for the linear exponential dichotomy are obtained, and Theorem \ref{main-theorem} then gives the corresponding local unstable and stable manifolds if the remaining assumptions are verified.

\vspace {.05in} 
\noindent {\bf Notations.} We use $B_r(X)$ to denote the open ball centered at the origin with radius $r$ in a Banach space $X$. $X^*$ represents the dual space of $X$, consisting of all real-valued bounded linear functionals. For any $k \geq 1$, $\dot{H}^k (\Omega)$ denotes the homogeneous Sobolev spaces, which are defined as the completions of Schwartz functions (smooth functions in the case of bounded $\Omega$) under the $H^{k-1}$ norm of gradient. In particular, for $d=2$ or bounded $\Omega$, $\dot H^k$ is a quotient space where constants are identified with $0$. Moreover we shall need the Banach space $(Z^l, |\cdot|_{Z^l})$ 
\be \label{E:Z^l} 
Z^l \triangleq     \left\{f \in L^\infty \mid \nabla f \in H^{l-1}\right\}, \quad |f|_{Z^l} = |f|_{L^\infty} + |\nabla f|_{H^{l-1}}.
\ee
Clearly for $k>\frac d2+1$, $Z^k= H^k(\Omega)$ if $\Omega$ is bounded. Usually $\sD$ denotes Fréchet derivatives, and  $D$ or $\nabla$ are spatial derivatives.   
We use $\cL(X,Y)$ to denote the set of bounded linear maps from $X$ to $Y$. When $\p \Omega \ne \emptyset$, we use $X^\top$ and $X^\perp$ to denote the tangential and normal component of a vector $X$ at some $x \in \p \Omega$. 
Finally $C>0$ and $K \in \mathbb{N}$ often denote generic upper bounds which might change from line to line.

\section{Analysis in Lagrangian coordinates of the stratified Euler equations with a background flow} \label{S:Lagrangian} 

Writing the Euler equations in Lagrangian coordinates is a well-established approach. 
However, one notices that the assumptions in \eqref{E:steady} on the steady states allow it to have possibly nontrivial asymptotic behavior at spatial infinity. In this section we carefully set up the functional analytic framework to study the stratified  incompressible  Euler equation with a background flow $(v_*, \rho_*)$, 
in which we shall construct local unstable manifolds of steady states in Section \ref{S:InMa}.

Throughout this section, we fix a background flow $(v_*, \rho_*)$,   along with some $p_*$, such that for some $\rho_\pm >0$,  
\begin{equation}
    \label{E:BG-density}  
    \rho_* \in Z_{\rho_+, \rho_-}^{k+r} \triangleq \{ f\in Z^{k+r} \mid \rho_- \le f \le \rho_+\},  \text{ where } \, 
    \; k> d/2+1,\; r \geq 3,
\end{equation}
\begin{equation}
    \label{E:BG-velocity}
    \nabla \cdot v_*=0, \;\; \N \cdot v_*|_{\p \Omega}=0, \;\; Dv_* \in Z^{k+r-1}, \; \nabla p_* \in Z^{k+r-1}, \;\; 
    \;\; k> d/2+1, \;
    r \geq 3,
\end{equation}
\be \label{E:p*}
v_* \cdot \nabla \rho_* \in H^{k+r-1}, \quad w_* \triangleq v_* \cdot \nabla v_* + g \Ve_d + \frac 1{\rho_*} \nabla p_* \in H_\Euler^{k+r-1}. 
\ee
In addition to the extra regularity of $v_*$, the assumptions \eqref{E:p*} require that $(v_*, \rho_*, p_*)$ is an equilibrium of the Euler equation up to error terms in $H^{k+r-1}$. 
These assumptions hold for equilibria satisfying \eqref{E:steady} with $w_*=0$. Near such a background flow, we will rewrite the Euler equation as an infinite dimensional ODE \eqref{E:Euler-L-3} in Lagrangian coordinates and obtain basic estimates. As a byproduct of the analysis in this section, we shall prove the following local well-posedness of the Euler equation with a background flow. 

\begin{theorem} \label{T:LWP} 
Assume $(\rho_*, v_*, p_*)$ satisfy assumptions \eqref{E:BG-density}, \eqref{E:BG-velocity}, and \eqref{E:p*}, then there exist $\delta, T >0$ determined by $\Omega, k+r, \rho_\pm, |\rho_*|_{Z^{k+r}}$, $|D v_*|_{Z^{k+r-1}}$, $|\nabla \rho_* \cdot v_*|_{H^{k+r-1}}$, 
$|w_*|_{H^{k+r-1}}$, and $|\nabla p_*|_{Z^{k+r-1}}$ such that for any 
\[
\rho(0) \in \rho_* + B_\delta (H^k(\Omega, \R)), \quad v(0) \in v_* + B_\delta (H_\Euler^k), 
\]
there exists a unique solution 
\[
v \in C^0( [-T, T], v_* + H_\Euler^k), \;\; \rho \in C^0 ( [-T, T], \rho_* + H^k(\Omega, \R)),  \text{ with } \nabla (p(t) - p_*) \in H^k(\Omega, \R^d). 
\]
Moreover, let $u_*(t, y)$ and $u(t, y)$ denote the Lagrangian maps corresponding to $v_*$ and $v(t)$ defined in \eqref{lagrangian-coordinate-map}, then for any $t\in [-T, T]$,  $u(t, \cdot) : \Omega \to \Omega$ is a diffeomorphism, 
$\det (D u) \equiv 1$, $\rho(t) \circ u(t) = \rho(0)$, and $\psi (t) \triangleq u_*(t) ^{-1}\circ u(t)$ satisfies that 
\begin{align*}
\big( (\rho(0), v(0)) \to (\psi(t) - \id_\Omega, \psi_t(t)) \big) : 
\big( \rho_* + B_\delta( H^k(\Omega, \R))  \big) \times \big( v_* + B_\delta (H_\Euler^k) \big) \to  (H^k(\Omega, \R^d))^2 
\end{align*}
is a $C^{r-2}$ mapping which depends on $t$ continuously. 
\end{theorem}

This theorem is a direct corollary of Lemma \ref{L:Euler-L-1} proved in Subsection \ref{SS:Euler-L}.

\subsection{Lagrangian coordinates} \label{SS:L-coord}

Given any smooth velocity field 
\be \label{E:temp-1}
v(t,\cdot) \in C^0\big([-\varepsilon,\varepsilon], v_* + H_\Euler^k(\Omega, \R^d)\big), \quad k >d/2+1, 
\ee
it defines a Lagrangian map $u (t, \cdot): \Omega \to \Omega$ 
by solving  
\begin{subequations} \label{lagrangian-coordinate-map}
\begin{equation} \label{E:L-map-1}
  u_t(t,y) = v(t,u(t,y)),
\ee
sometimes along with the initial condition 
\be \label{E:L-map-2}
    u(0,y) = y.
\end{equation}
\end{subequations}
One notes that \eqref{E:L-map-2} is a matter of convenience rather than an essential condition 
due to the relabeling symmetry. While \eqref{E:L-map-2} is not always assumed, we require $\det Du(0, \cdot)\equiv 1$. 
The Lagrangian coordinate $u(t,\cdot)$ is well defined for all $t\in [-\varepsilon,\varepsilon]$ since the velocity field $v(t,\cdot)$ has a global Lipschitz constant in $x$ which is uniform in $t$ due to (\ref{E:BG-velocity}) and $k > 1+\frac{d}{2}$ (and $u(t,\cdot)$ preserves $\p \Omega$ in the case of $\Omega= \BBT^{d_1} \times \Omega_2$).
By standard ODE estimates and the incompressibility in (\ref{E:BG-velocity}), 
\be \label{E:tsG}
u(t,\cdot) \in \tilde \sG \triangleq \left\{\text{diffeomorphic } \varphi
: \Omega \to \Omega 
\mid D \varphi, D\varphi^{-1} \in Z^{k-1}, \det(D \varphi) \equiv 1, \varphi(\partial \Omega) = \partial \Omega  \right\}.
\ee
Clearly $\tilde \sG$ is a group under the operation of composition.  

\begin{remark} 
In deriving $D u(t, \cdot) \in Z^{k-1}$, we used the following inequality: for any $k > d/2+1$, $0\le l \le k$,  $f_1(x)$, and diffeomorphism $f_2(x)$
\be \label{E:composition-1} 
\|f_1 \circ f_2\|_{H^l} \le C \| 1/ \det (D f_2) \|_{L^\infty}^{\frac 12}   |Df_2|_{Z^{k-1}} (1 + |Df_2|_{Z^{k-1}}^{l-1}) \|f_1\|_{H^l}.
\ee
This inequality will be used repeatedly throughout the paper. 
\end{remark}

One observes that, for the Lagrangian map $u(t) \in \tilde \sG$ associated with a solution $(v(t), \rho(t))$ to the Euler equation \eqref{stratified-euler}, the transport of $\rho$ by $v$ ensures   
$(\rho \circ u)(t)$ 
to be independent of $t$, namely 
\begin{equation}
\label{density-constant}
\partial_t (\rho(t) \circ u(t)) =(\partial_t\rho(t) + v(t)\cdot \nabla \rho(t)) \circ u(t)=0 \implies \rho(t) = \rho(0) \circ u(0) \circ u(t)^{-1}.
\end{equation}
Hence in the Lagrangian formulation, the density is reduced to its initial value, the incompressible Euler equation \eqref{stratified-euler} with velocity fields in $v_* + H_\Euler^k$ takes the form of 
\be \label{E:Euler-L} 
u_{tt} + g \Ve_d = -\Big( \frac 1{\rho} \nabla p \Big)  \circ u= - \frac 1{\rho(0) \circ u(0)} (\nabla p) \circ u, \quad u \in \tilde \sG. 
\ee

Let $u_*(t,\cdot)$ denote the Lagrangian map associated with the ($t$-independent) background flow $v_*$. Since $\|D v_*\|_{L^\infty} < \infty$ due to (\ref{E:BG-velocity}), $u_*$    is defined for all $t \in \R$. 
The Euler equations near $(\rho_*, v_*)$ will be analyzed based on \eqref{E:Euler-L} in a neighborhood of $u_*(t)$ in a suitable $H^k$ topology. This is carried out carefully as $u(t)$ -- even $u_*(t)$ itself -- may have nontrivial asymptotic behavior at spatial infinity if $\Omega$ is unbounded. 
More specifically, for a Lagrangian coordinate map $u(t)$ associated with a velocity field $v(t)$ satisfying \eqref{E:temp-1}, consider 
\be \label{E:temp-2} 
u(t) =  u_* (t) \circ \psi(t) \implies \psi_t = (Du_* \circ \psi)^{-1} (v- v_*) \circ u_* \circ \psi, \quad \psi(0)=\id_\Omega. 
\ee
Thus we obtain
\begin{equation}
\label{sG}
\psi(t) \in \sG \triangleq \tilde \sG \cap \big(\id_\Omega + H^k(\Omega,\mathbb{R}^d)\big), \quad \psi - \id_\Omega \in C^1 \left([-\varepsilon,\varepsilon],  H^k\right).
\end{equation}
Apparently $\sG$ is a subgroup of $\tilde \sG$ and $\sG = \tilde \sG$ if $\Omega$ is bounded. We shall convert the incompressible Euler equation \eqref{stratified-euler}, or equivalently \eqref{E:Euler-L}, to be a flow with the configuration space $\sG$. 

\subsection{Structure of $\sG$} \label{SS:sG}

Naturally one may view $\sG \subset \id_\Omega + H^k(\Omega,\mathbb{R}^d)$ as an  infinite dimensional Hilbert submanifold with the tangent space $T_\phi \sG = H^k_\Euler \circ \phi$ at any $\phi \in \sG$. In the case of {\it bounded} $\Omega$, this had been proved in \cite{lin-zeng-manifold} by expressing $\sG$  near $\id_\Omega$ locally as a graph in the  decomposition of $H^k (\Omega, \R^d)$ into the direct sum of  $H_\Euler^k$ and its $L^2$-orthogonal complement. That proof requires some delicate modifications when $\Omega$ is unbounded.\footnote{For simplicity we focus on $\Omega = \BBT^{d_1} \times \R^{d_2}$ with $\p \Omega= \emptyset$ in the case of unbounded $\Omega$ in this paper.} Moreover, based on the form of the kinetic energy with the density stratification, it makes more physical and mathematical sense to consider the weighted $L^2 (\rho dx)$ inner product in the Hodge decomposition of the space $H^k(\Omega, \R^d)$. 
Namely, 
let  
\be \label{E:normal-sG}
\left(H^l_\Euler\right)^{\perp_{\rho}} \triangleq \left\{\frac{1}{\rho} \nabla h: h \in \dot{H}^{l+1} (\Omega, \mathbb{R})\right\}, \quad 0\le l \le k.
\ee
The orthogonality is in the weighted $L^2(\rho dx)$ sense: for any $w \in H^l_{\Euler}$ and $h \in \dot{H}^{l+1}(\Omega, \mathbb{R})$,
$$
\int_\Omega w \cdot  \left(\frac{1}{\rho}\nabla h\right) \rho \, dx = -\int_{\Omega}h(\nabla \cdot w) \, dx + \int_{\partial \Omega}h (w \cdot \textbf{N}) \, dS=0.
$$
Clearly $H^l_\Euler$ and $(H^l_\Euler)^{\perp_{\rho}}$ are closed subspaces of $H^l \left(\Omega,\mathbb{R}^d\right)$. 

\begin{lemma}[Hodge Decomposition]
 \label{hodge}
For any $\rho \in Z_{\rho_+, \rho_-}^k$ and $0 \le l \le k$, $H^l \left(\Omega,\mathbb{R}^d\right) = H^l_\Euler \oplus (H^l_\Euler)^{\perp_{\rho}}$, and the associated projection $\cP(\rho): H^l(\Omega,\mathbb{R}^d) \to H^l_\Euler$ 
satisfies $| \cP(\rho)|_{\cL(H^l)} \le C$ where $C\ge 1$ is determined by $\Omega, k, l, \rho_\pm$, and $|\rho|_{Z^k}$.
\end{lemma}

\begin{proof} 
The uniqueness of the decomposition follows directly from the orthogonality and the positive upper and lower bounds $\rho_\pm$ in \eqref{E:BG-density} of $\rho$. We shall focus on the existence of the decomposition. 

In the case of $\Omega = \mathbb T^{d_1} \times \mathbb{R}^{d_2}$, let $X \in H^{l}(\Omega,\mathbb{R}^d)$. Since $\nabla \cdot X \in \dot H^1(\Omega, \R)^*$, by \eqref{E:BG-density} and the Lax-Milgram Theorem, there exists a unique $h \in \dot H^1(\Omega, \R)$ such that $-\nabla \cdot \left(\frac{1}{\rho}\nabla h\right) = \nabla \cdot X$. From the standard elliptic theory using 
the regularity \eqref{E:BG-density} of $\rho$, we obtain $ \nabla h \in H^l (\Omega, \R)$ and thus $\frac{1}{\rho}\nabla h \in \left(H^l_\Euler\right)^{\perp_{\rho}}$. Finally, let $w= X + \frac{1}{\rho}\nabla h \in H_\Euler^l$ and it completes the proof of the decomposition for $\Omega = \mathbb T^{d_1} \times \mathbb{R}^{d_2}$. 

In the bounded case of $\Omega = \mathbb T^{d_1} \times \Omega_2$ with $\Omega_2 \subset\subset \R^{d_2}$, since 
\[
\left( f \in \dot H^1 (\Omega, \R) \to \int_\Omega f  \nabla \cdot X dx - \int_{\p \Omega} f X \cdot \N dS = - \int_\Omega \nabla f \cdot X dx\right) \in \dot H^1(\Omega, \R)^*, 
\]
and 
\[
(f,g) \to \int_\Omega \frac 1{\rho} \nabla f \cdot \nabla g \, dx
\]
is bounded, bilinear, and uniformly positive on $\dot H^1(\Omega, \R)$, by the Lax-Milgram Theorem there is a unique solution $h \in \dot H^1(\Omega, \R)$ to 
\be \label{E:Poisson-BVP}
-\nabla \cdot \left(\frac 1{\rho}\nabla h\right) = \nabla \cdot X \; \text{ in } \ \Omega, \quad 
   \frac 1{ \rho}\nabla h \cdot \N = - X \cdot \N \; \text{ on }\ \partial \Omega.
\ee
Again the standard elliptic theory and \eqref{E:BG-velocity} yields $ \nabla h \in H^l(\Omega, \R)$. The rest of the proof follows in a similar fashion as in the case of $\Omega = \mathbb T^{d_1} \times \mathbb{R}^{d_2}$. 
\end{proof}

\begin{remark} \label{R:LPressure}
In the linearization \eqref{linearization} of the Euler equation,  equation \eqref{E:LPressure} of the linearized pressure $\delta p$ is equivalent to 
\[
- \frac 1{\rho_0} \nabla\delta p =\big(I- \cP(\rho_0)\big) \big( v_0 \cdot \nabla \delta v +  \delta v \cdot \nabla v_0 -  \frac {\delta \rho}{\rho_0^2} \nabla p_0 \big). 
\] 
\end{remark}

To see that $\sG$ is a Hilbert manifold with model space $H^k_\Euler$ locally near the identity map $\id_{\Omega} \in \mathscr G$, we prove that $\mathscr G$ is a graph of a smooth mapping from $H^k_{\Euler}$ to $ (H^k_\Euler)^{\perp_{\rho}}$. 

\begin{proposition}
    \label{local-coordinate}
    Given any ${\rho}\in Z_{\rho_+, \rho_-}^k$, there exist $\delta > 0$ determined only by $\Omega, k, \rho_\pm, |\rho|_{Z^k}$ and  
    $\Psi_{\rho} \in C^\infty \big( B_{\delta}(H^k_{\Euler}) \to \id_\Omega + H^k(\Omega, \mathbb{R}^d) \big)$
    such that the following holds:
    \begin{enumerate}
        \item $\Psi_{\rho}(0) = \id_\Omega$, $\mathscr D \Psi_{\rho} (0) = I|_{H^k_{\Euler}},$
        \item  $\Psi_{\rho} (w) - w - \id_{\Omega} \in \left(H^k_\Euler\right)^{\perp_{\rho}}$  for any $ w \in B_{\delta} \left(H^k_\Euler\right) ,$
         and
         \item $\Big(\id_\Omega + B_{{\delta}}(H_\Euler^k) \oplus B_{{\delta}}\big((H^k_\Euler)^{\perp_{\rho}}\big) \Big) \cap \mathscr G=\Psi_{\rho}\left( B_{\delta} \left(H^k_{Euler}\right)\right)$.
    \end{enumerate}
\end{proposition}

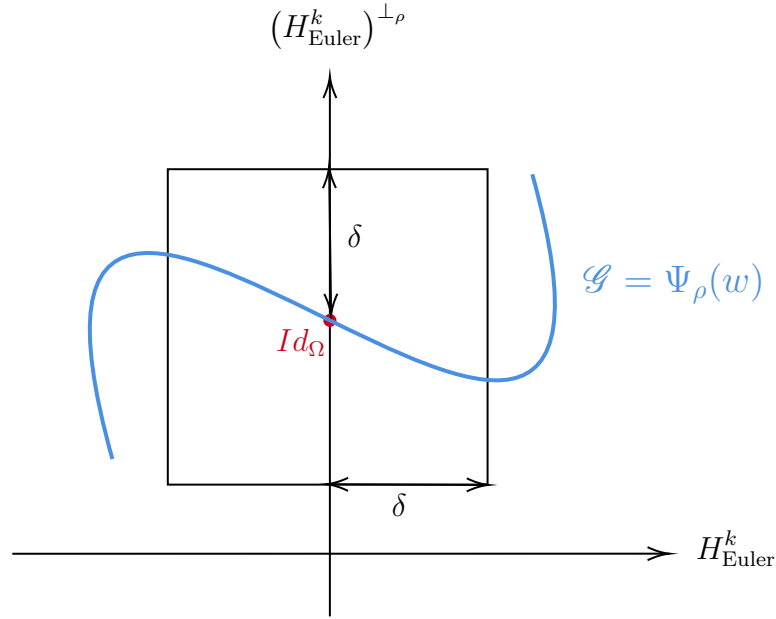
\begin{figure}[H]
  \tikzset{every picture/.style={line width=0.75pt}} 

\begin{tikzpicture}[x=0.75pt,y=0.75pt,yscale=-1,xscale=1]

\draw    (167,324.54) -- (494.75,324.5) ;
\draw [shift={(496.75,324.5)}, rotate = 179.99] [color={rgb, 255:red, 0; green, 0; blue, 0 }  ][line width=0.75]    (10.93,-3.29) .. controls (6.95,-1.4) and (3.31,-0.3) .. (0,0) .. controls (3.31,0.3) and (6.95,1.4) .. (10.93,3.29)   ;
\draw    (326.38,356) -- (326.38,86.77) ;
\draw [shift={(326.38,84.77)}, rotate = 90] [color={rgb, 255:red, 0; green, 0; blue, 0 }  ][line width=0.75]    (10.93,-3.29) .. controls (6.95,-1.4) and (3.31,-0.3) .. (0,0) .. controls (3.31,0.3) and (6.95,1.4) .. (10.93,3.29)   ;
\draw  [draw opacity=0][fill={rgb, 255:red, 208; green, 2; blue, 27 }  ,fill opacity=1 ][line width=3]  (329.68,207.39) .. controls (329.68,205.59) and (328.2,204.14) .. (326.38,204.14) .. controls (324.56,204.14) and (323.08,205.59) .. (323.08,207.39) .. controls (323.08,209.19) and (324.56,210.64) .. (326.38,210.64) .. controls (328.2,210.64) and (329.68,209.19) .. (329.68,207.39) -- cycle ;
\draw   (245.04,131.46) -- (405.52,131.46) -- (405.52,289.83) -- (245.04,289.83) -- cycle ;
\draw    (326.58,289.65) -- (404,289.99) ;
\draw [shift={(406,290)}, rotate = 180.25] [color={rgb, 255:red, 0; green, 0; blue, 0 }  ][line width=0.75]    (10.93,-3.29) .. controls (6.95,-1.4) and (3.31,-0.3) .. (0,0) .. controls (3.31,0.3) and (6.95,1.4) .. (10.93,3.29)   ;
\draw [shift={(324.58,289.64)}, rotate = 0.25] [color={rgb, 255:red, 0; green, 0; blue, 0 }  ][line width=0.75]    (10.93,-3.29) .. controls (6.95,-1.4) and (3.31,-0.3) .. (0,0) .. controls (3.31,0.3) and (6.95,1.4) .. (10.93,3.29)   ;
\draw [line width=0.75]    (326.97,203) -- (326.03,132) ;
\draw [shift={(326,130)}, rotate = 89.24] [color={rgb, 255:red, 0; green, 0; blue, 0 }  ][line width=0.75]    (10.93,-3.29) .. controls (6.95,-1.4) and (3.31,-0.3) .. (0,0) .. controls (3.31,0.3) and (6.95,1.4) .. (10.93,3.29)   ;
\draw [shift={(327,205)}, rotate = 269.24] [color={rgb, 255:red, 0; green, 0; blue, 0 }  ][line width=0.75]    (10.93,-3.29) .. controls (6.95,-1.4) and (3.31,-0.3) .. (0,0) .. controls (3.31,0.3) and (6.95,1.4) .. (10.93,3.29)   ;
\draw  [color={rgb, 255:red, 74; green, 144; blue, 226 }  ,draw opacity=1 ][line width=1.5]  (217.24,277) .. controls (137.39,-9) and (507.8,420) .. (427.95,134) ;

\draw (508.68,311.42) node [anchor=north west][inner sep=0.75pt]   [align=left] {$\displaystyle H_{\text{Euler}}^{k}$};
\draw (293.68,46.61) node [anchor=north west][inner sep=0.75pt]   [align=left] {$\displaystyle \left( H_{\text{Euler}}^{k}\right)^{\perp _{\rho }}$};
\draw (297.92,210.29) node [anchor=north west][inner sep=0.75pt]  [color={rgb, 255:red, 208; green, 2; blue, 27 }  ,opacity=1 ] [align=left] {$\displaystyle Id_{\Omega }$};
\draw (452.04,177.3) node [anchor=north west][inner sep=0.75pt]  [font=\large,color={rgb, 255:red, 74; green, 144; blue, 226 }  ,opacity=1 ] [align=left] {$\displaystyle \mathscr{G} =\Psi _{\rho }( w)$};
\draw (333.77,157.92) node [anchor=north west][inner sep=0.75pt]   [align=left] {$\displaystyle \delta $};
\draw (355.94,292.97) node [anchor=north west][inner sep=0.75pt]   [align=left] {$\displaystyle \delta $};

\end{tikzpicture}
  \caption{Local Coordinates near the Identity Map}
  \label{local-graph}
\end{figure}

\begin{proof}
We shall only give the proof in the case of unbounded domains $\Omega = \mathbb{T}^{d_1} \times \mathbb{R}^{d_2}$ without boundaries. 
For bounded domains $\Omega = \Omega_1 \times \mathbb T^{d_2}$, the proof is almost identical to the one of Proposition 2.1 in \cite{lin-zeng-manifold} except for using the density weighted  Hodge decomposition given in Lemma \ref{hodge} instead. 
Consider the map $F: H^k(\Omega,\mathbb{R}^d) \to H^{k-1}(\Omega,\mathbb{R}) \cap \big(\dot{H}^1(\Omega,\mathbb{R})\big)^\ast $ defined as 
    $$
    F(X)\triangleq 1-\det(I+DX), \; \forall X \in H^k(\Omega,\mathbb{R}^d).
    $$
    Obviously, $F(X)  \in H^{k-1}(\Omega,\mathbb{R})$. To verify that $F(X) \in \big(\dot{H}^1 (\Omega,\mathbb{{R}})\big)^\ast$, it suffices to show that $F(X)$ is in the divergence form.  Recall
    $\sum_{j = 1}^d \partial_{j}\text{cof}(D (\id+tX))_{i,j}=0$ for any $i = 1,2,\cdots,d$, where cof$(A)$ denotes the cofactor matrix of a matrix $A$. 
    By direct computations,
    $$
    \begin{aligned}
        F(X)
        &= -\int_0^1 \frac{d}{dt}\det(I + t DX) \, dt=-\int_0^1 \tr \left((\text{cof}(I+ t DX))^T D X \right)\,dt \\
        &=\int_{0}^1- \nabla \cdot \left(\text{cof}(I + t DX)^T X\right) - \sum_{i=1}^d X^i  \sum_{j = 1}^d \partial_{j}\text{cof}(D (\id+tX))_{i,j} \, dt\\
        &=\nabla \cdot \left(-\int_0^1 \text{cof}(I + t D X)^TX \, dt\right)\in \big(\dot{H}^1 (\Omega,\mathbb{{R}})\big)^\ast, 
    \end{aligned}
    $$
    so $F$ is well-defined. Since the cofactor$(A)$ is the sum of products of $d-1$ entries of $A$, the above formula also implies the $C^\infty$ smoothness of $F$ due to $k > d/2+1$. 
        
        Recall that $\sG$, defined in (\ref{sG}), is the collection of volume-preserving diffeomorphisms in $\id_\Omega + H^k(\Omega,\mathbb{R}^d)$, and one  observes that, for sufficiently small $\ep>0$, 
    $$
    F^{-1}(0) \cap B_\varepsilon(H^k(\Omega,\mathbb{R}^d)) = (\sG -\id_\Omega)\cap B_\varepsilon(H^k(\Omega,\mathbb{R}^d)).
    $$ 
    

    For any $q = \frac 1{\rho}\nabla h \in (H^k_\Euler)^{\perp_{{\rho}}}$, where $h \in \dot{H}^{k+1}(\Omega,\mathbb R)$, one could compute 
    $$
    \sD F(0)q= -\nabla \cdot q=-\nabla \cdot \left(\frac 1{\rho} \nabla h\right),
    $$
    which is an isomorphism from $(H^k_\Euler)^{\perp_{\rho}}$ to $H^{k-1} \cap (\dot{H}^1)^\ast$ due to \eqref{E:BG-density}, and the standard elliptic theory as in the proof of Lemma \ref{hodge}. By the implicit function theorem, there exist $0 < \delta \leq \varepsilon$ and a function $\Phi\in C^\infty\big(B_{\delta}(H^k_\Euler),  B_\delta((H^k_\Euler)^{\perp_{\rho}})\big)$ such that $\Phi(0) = 0$, $\sD \Phi(0) =0$, and 
\[
F(w +q ) = 0, \; (w,q) \in B_{\delta}(H^k_\Euler) \times B_\delta \big((H^k_\Euler)^{\perp_{\rho}} \big) \ \text{ iff } \ q = \Phi (w). 
\]
Let $\Psi_{\rho}: B_{\delta}(H^k_\Euler) \to \id_\Omega  + H^k(\Omega,\mathbb R^d)$ be defined as  
    $$
    \Psi_{\rho}(w)=\id_\Omega + w + \Phi(w), \; \forall w   \in B_{\delta}(H^k_\Euler),
    $$
and clearly it satisfies all the desired properties. 
    \end{proof}
    
\begin{remark} \label{R:sG}
For any $\phi \in \mathscr G$, 
$\Psi_{\rho} (\cdot) \circ \phi$ gives a smooth local coordinate system of $\sG$ near $\phi$, which makes $\mathscr G$ a rigorous Hilbert manifold with the model space $H^k_\Euler$. The tangent and normal space of $\sG$ at $\phi$ are given by 
\be \label{E:sG-rho}
T_\phi \sG = \{w \circ \phi \mid w \in H_\Euler^k\}, \quad (T_\phi \sG)^{\perp_\rho} = \Big\{ \frac 1{\rho} \big((\nabla h) \circ \phi\big) \mid h \in H^{k+1} (\Omega, \R) \Big\},
\ee
where the orthogonality is in the $L^2(\rho dx)$ inner product. 
\end{remark} 

We shall also need the Hodge decomposition of $H^k(\Omega, \R^d)$ at each $\phi \in \tilde \sG$ with a density $\sigma \in Z_{\rho_+, \rho_-}^k$. Define 
\be \label{E:Q}
Q (\sigma, \phi) X = \big(\cP (\sigma \circ \phi^{-1}) (X\circ \phi^{-1}) \big) \circ \phi, 
\ee
where 
$\cP$ is given in Lemma \ref{hodge}. 
Here $Q$ is well-defined since $ \sigma \circ \phi^{-1}\in Z_{\rho_+, \rho_-}^k$, too.  From the definition \eqref{E:Q} of $Q$, it holds 
\be \label{E:Q-1}
Q(\sigma, \phi \circ \psi) X = \big(Q(\sigma \circ \psi^{-1}, \phi) (X \circ \psi^{-1}) \big) \circ \psi, \; \forall \phi, \psi \in \tilde \sG.  
\ee
According to \eqref{E:sG-rho} and Lemma \ref{hodge} we also have 
\be \label{E:Q-1.5} 
Q(\sigma, \phi)X \in T_\phi \sG,\; \forall \phi \in \sG. 
\ee
In the following, fixing $\phi_0 \in \tilde \sG$ we consider $Q(\sigma, \phi_0 \circ \psi)$ depending on $\sigma\in Z_{\rho_+, \rho_-}^k$ and $\psi \in \sG$. Let 
\be \label{E:cQ}
Q_{\phi_0} (\sigma, \psi) = Q(\sigma, \phi_0 \circ \psi), \quad \cQ_{\phi_0} (\sigma, w)= Q_{\phi_0} (\sigma, \Psi_{\rho_*}(w)), \quad w \in B_\delta (H_\Euler^k),
\ee
where $\cQ_{\phi_0}$ is simply $Q_{\phi_0}$ represented in the local coordinate $\psi=\Psi_{\rho_*}(w)$ for $\psi$ near $\id_\Omega$.

\begin{lemma} \label{L:Hodge-2}
Suppose $\phi_0(t) \in \tilde \sG$, $t\in [0, T]$, satisfies $D\phi_0 \in C^l ([0, T], Z^{k+r-1})$, $k>\frac d2+1$, and $l, r\ge 0$, then for any $a_+>a_->0$, 
\[
Q_{\phi_0(\cdot)} \in C_{t\in [0, T]}^l C^{r} \big(Z_{a_+, a_-}^k \times \sG, \cL(H^k(\Omega, \R^d))\big), 
\]
\[
\cQ_{\phi_0(\cdot) } \in C_{t\in [0, T]}^l  C^{r} \big(Z_{a_+, a_-}^k \times B_\delta (H_\Euler^k), \cL(H^k(\Omega, \R^d))\big). 
\] 
Moreover, there exist $K \in \mathbb{N}$ determined by $k+r+l$ and $\delta>0$ determined by $\Omega, k$, and, $a_\pm$ such that  
for any $R>0$ 
\be \label{E:Q-C^{r-1}}
\|\cQ_{\phi_0} \|_{C_t^l C^{r}((Z_{a_+, a_-}^k \cap B_R (Z^k)) \times B_\delta (H_\Euler^k), \cL(H^k(\Omega, \R^d)))} \le C |D\phi_0|_{C_t^l Z^{k+r-1}}^K. 
\ee
where $C>0$ is a constant determined by $k+r, a_\pm, R$, and $\Omega$.  
\end{lemma}

\begin{proof} 
Due to Remark \ref{R:sG} and the symmetry  \eqref{E:Q-1} of $Q$, $Q_{\phi_0}$ is smooth on $[0, T]\times Z_{a_+, a_-}^k \times \sG$ iff it is smooth for $\psi \in \sG$ near $\id_\Omega$, where $Q_{\phi_0}$ and $\cQ_{\phi_0}$ are equivalent. Denote $\phi=\phi_0(t)\circ \psi$. Hence in the rest of the proof we shall work on $Q_{\phi_0(t)} (\sigma, \psi) = Q(\sigma, \phi)$. A fact which will be used repeatedly in the proof is that for $d\times d$ matrices such as $D\phi$ with determinant $1$, their inverses are their homogeneous polynomials of degree $d-1$ due to the representation using cofactors. In the following $K\in \mathbb{N}$ and $C>0$ are generic constants which may change from line to line. 

It is easier to consider $I-Q(\sigma, \phi)$ given by 
\be \label{E:Q-2}
(I-Q(\sigma, \phi))X = - \frac 1\sigma  \gamma, \ \text{ where } \  \gamma =(\gamma^1, \ldots, \gamma^d)^T \triangleq (\nabla h ) \circ \phi    \in H^k(\Omega, \R^d),   
\ee
where $ h$  solves
\be \label{E:temp-3} \begin{cases}
        -\nabla  \cdot \left(\frac{1}{\sigma \circ \phi^{-1}}\nabla h\right) = \nabla \cdot(X \circ \phi^{-1})=\tr(D(X \circ \phi^{-1})) &\text{ in } \Omega,\\
       \frac{1}{\sigma \circ \phi^{-1}}\nabla h \cdot \N = - (X \circ \phi^{-1})\cdot \N &\text{ on } \partial \Omega.
\end{cases} \ee
Here 
the boundary conditions is omitted if $\p\Omega =\emptyset$. The regularity $h \in \dot{H}^{k+1}(\Omega,\mathbb{R})$ follows from standard elliptic theory as in Lemma \ref{hodge}. Under the change of coordinates $x= \phi(y)$, we shall need the following general formulas 
\be \label{E:temp-4}
(\nabla_x f) \circ \phi = ((D\phi)^{-1})^T \nabla_y (f \circ \phi), \quad (\nabla_x \cdot W) \circ \phi=  \frac 1{\det D\phi}  \nabla_y \cdot \big( \det (D\phi)  (D\phi)^{-1} (W\circ \phi) \big). 
\ee

To analyze the smoothness of $\gamma= (\nabla h ) \circ \phi = ((D\phi)^{-1})^T\nabla \tilde{h}$ in $\phi$ (so equivalently in $\psi$), where $\tilde h = h \circ \phi$, we consider the elliptic equation satisfied by $\tilde h$:  
\be \label{E:t-h-1} \begin{cases}
    -\nabla \cdot \left(\frac{1}{\sigma}(D\phi)^{-1}((D\phi)^{-1})^T\nabla \tilde{h}\right) = \nabla \cdot \big( (D\phi)^{-1} X)
    = \tr(DX (D\phi)^{-1}) 
    & \text{ in } \Omega, \\
   \frac{1}{\sigma}(D\phi)^{-1}((D\phi)^{-1})^T\nabla \tilde{h} \cdot \N= -\frac {dS_x}{dS_y} X \cdot (\N \circ \phi) = - \N \cdot  (D\phi)^{-1} X
   & \text{ on } \partial\Omega.
\end{cases} \ee
Here we used $\det (D\phi) \equiv 1$ and 
\be \label{E:temp-4.5}
\N \circ \phi = \frac {((D\phi)^{-1})^T \N}{|((D\phi)^{-1})^T \N|}, \quad \frac {dS_x}{dS_y} = |\det D \phi| | ((D\phi)^{-1})^T \N|. 
\ee
Denote the coefficient matrix 
\begin{equation}
\label{A_sigma_u}
A(\phi)\triangleq (D\phi)^{-1}((D\phi)^{-1})^T.
\end{equation}
For any $\phi=\phi_0\circ \psi$ with $\psi \in \id_\Omega + B_{2\delta} (H^k(\Omega, \R^d)) \supset \Psi_{\rho_*} \big( B_{\delta} (H_\Euler^k )\big)$, 
we observe:
\begin{itemize}
    \label{observations}
    \item At any $x \in \bar \Omega$, as a $d\times d$ symmetric matrix, 
    \[
   2 |(D\phi_0(x))^{-1}|^2 \ge |(D\phi(x))^{-1}|^2 \ge A(\phi) (x) \ge |D\phi(x)|^{-2} \ge |D\phi_0(x)|^{-2} /2.
    \] 
    \item The mapping 
    \[
    \Big( \psi \to (D \phi)^{-1} = (\text{cof}(D \psi))^T (\text{cof}((D \phi_0) \circ \psi))^T   \Big) \in C^{r} \big(\id_\Omega + B_{2\delta} (H^k(\Omega, \R^d)), H^{k-1} (\Omega, \R^{d^2}) \big).   
    \]
    \item $A \in C^{r} \big( \id_\Omega + B_{2 \delta} (H^k(\Omega, \R^d)), H^{k-1} (\Omega, \R^{d^2})\big)$ and 
    \[
    \| A \|_{C^{r} (\id_\Omega + B_{2\delta} (H^k(\Omega, \R^d)), H^{k-1} (\Omega, \R^{d^2}))} \le C |D \phi_0|_{Z^{k+r-1}}^K.
    \]
\end{itemize}
The $H^{k-1}$ estimates of $\nabla \tilde h$ follow from the standard elliptic theory. By differentiating in $t$, $\sigma$, and $\psi$ repeatedly, using the smoothness of $\p\Omega$ and $\sigma \in Z_{a_+, a_-}^k$, we also obtain 
\be \label{E:temp-4.7}\begin{split}
&\big((t, \sigma, \psi) \to \nabla \tilde h \big) \in C_{t\in [0, T]}^l C^{r} \big(Z_{a_+, a_-}^k \times \sG, \cL(H^{k-1}(\Omega, \R^d))\big) \\
\implies & Q_{\phi_0} \in C_{t\in [0, T]}^l C^{r} \big(Z_{a_+, a_-}^k \times \sG, \cL(H^{k-1}(\Omega, \R^d))\big),
\end{split} \ee
with the $C^{r}$ norm of $\cQ_{\phi_0}$ satisfying the desired estimate \eqref{E:Q-C^{r-1}}. 

To obtain the smoothness of $Q (\sigma, \phi) \in \cL(H^{k}(\Omega, \R^d))$ with respect to $\psi$, let $Y\in C^{k+r} (\Omega, \R^d)$ such that $Y\cdot \N =0$ along $\p \Omega$. Applying $\nabla_Y= Y \cdot \nabla$ to \eqref{E:temp-3} and using the commutator $[\Ve_i, Y]= Y_{x_i}$, we obtain the equations satisfied by $Y \cdot (\gamma \circ \phi^{-1})= \nabla_Y h$ in $\Omega$, 
\be \label{E:temp-5} \begin{split}
-\nabla \cdot \left(\frac 1{\sigma \circ \phi^{-1}} \nabla (\nabla_{Y} h)\right) = & \nabla_Y \nabla \cdot (X \circ \phi^{-1}) - \sum_{i = 1}^n 
    \nabla_{Y_{x_i}}\big(\frac 1{\sigma \circ \phi^{-1}} h_{x_i}\big) \\
&+ \nabla \cdot\Big( \big( \nabla_Y \big(\frac 1{\sigma \circ \phi^{-1}} \big)\big) \nabla h - \frac 1{\sigma \circ \phi^{-1}} (DY)^T \nabla h \Big),
\end{split} \ee
and along $\p \Omega$, 
\be \label{E:temp-6} \begin{split}
\frac 1{\sigma \circ \phi^{-1}} \nabla (\nabla_{Y} h)  \cdot \N = & - \nabla_Y ( (X \circ \phi^{-1}) \cdot \N) + \big(\nabla_Y \big(\frac 1{\sigma \circ \phi^{-1}} \big)\big) \nabla h \cdot \N \\
& - \frac 1{\sigma \circ \phi^{-1}} \nabla h \cdot \big(DY(\N) - \Pi(Y) \big).  
\end{split} \ee
Here $\Pi= \Pi(x) \in \cL (\R^d, T_x \p \Omega)$ 
is an extension of the second fundamental form of $\partial \Omega$ if $\p \Omega \ne \emptyset$:
\be \label{E:2ndFF}
\Pi(x)W =\nabla_{W^\top} \N, \quad \forall x \in \p \Omega, \; W \in \R^d, 
\ee
which is smooth and used sometimes in  its corresponding symmetric bilinear form. As in the previous step, the above equations are transformed by $x= \phi(y)$ using 
\[
(\nabla_W f) \circ \phi = \nabla_{(D\phi)^{-1} W} (f \circ \phi), \quad \forall W \in \R^d. 
\]
Denote $\tilde Y= (D\phi)^{-1} (Y\circ \phi)$. From \eqref{E:temp-5}, we have in $\Omega$
\be \label{E:temp-7} \begin{split}
-\nabla \cdot \left(\frac 1\sigma A(\phi) \nabla \big((Y\circ \phi) \cdot \gamma\big)\right) = & \nabla_{\tilde Y} \tr\big(DX (D\phi)^{-1}\big) 
 - \sum_{i = 1}^n         \nabla_{(D \phi)^{-1} (Y_{x_i} \circ \phi)}\big(\frac 1\sigma \gamma^i\big) \\
&+ \nabla \cdot\Big( (D\phi)^{-1} \Big( \big( \nabla_{\tilde Y} \big(\frac 1{\sigma} \big)\big) \gamma - \frac 1{\sigma} \big((DY)^T \circ \phi \big)\gamma \Big)\Big),
\end{split} \ee
and along $\p \Omega$ using \eqref{E:temp-4.5}
\be \label{E:temp-8} \begin{split}
\frac 1{\sigma} A(\phi) \nabla \big((Y\circ \phi) \cdot \gamma\big)  \cdot \N = & - | ((D\phi)^{-1})^T \N| \Big( \nabla_{\tilde Y} ( X\cdot (\N \circ \phi)) - \frac 1{\sigma} \gamma \cdot (\Pi(Y) \circ \phi) \Big) \\ 
&+ \big(\nabla_{\tilde Y} \big(\frac 1{\sigma} \big)\big) \N \cdot (D\phi)^{-1} \gamma - \frac 1{\sigma} \gamma \cdot ((D Y) \circ \phi) ((D\phi)^{-1})^T \N.
\end{split} \ee
Again, from \eqref{E:composition-1} and the standard energy estimates, we obtain 
\[
\big((t, \sigma, \psi) \to (Y\circ \phi) \cdot \gamma \big) \in C_{t\in [0, T]}^lC^r \big(Z_{a_+, a_-}^k \times \sG, \cL(H^{k}(\Omega, \R))\big),
\]
with the $C^r$ norms depending on $r, d, k, \rho_\pm, a_\pm, R$, and $|\rho_*|_{Z^k}$. 

If $\Omega = \BBT^{d_1} \times \R^{d_2}$, simply take $Y=\Ve_1, \ldots, \Ve_d$, and the above analysis gives us the desired results. 

If $\Omega =\BBT^{d_1} \times \Omega_2$ with smooth compact $\Omega_1 \subset \R^{d_1}$, we follow the standard procedure in elliptic PDEs to obtain: a.) interior estimates; b.) boundary estimates in the tangential directions of $\p \Omega_1 \times \BBT^{d_2}$; and c.) boundary estimates in the normal directions using boundary condition combined with the elliptic equation. Roughly, due to the compactness and smoothness  of $\Omega$, there exist $\ep>0$, $m \in \mathbb N$, $P_i \in \p \Omega$, and $C^{k+r}$ vector fields $Y_{i, j}$, for $1\le i\le m$ and $1\le j \le d$, such that the following hold. 
\begin{enumerate} 
\item 
$(Y_{i, j}(x))_{j=1}^d$ form a basis of $\R^d$ at each $x\in B_{3\ep}(P_i)$; and at each $x \in B_{3\ep}(P_i)\cap \p \Omega$, $Y_{i, j}(x) \subset T_x \p\Omega$ for $1\le j \le d-1$ and $Y_{i, d}(x)= \N(x)$. 
\item $\Omega \subset \big(\cup_{i=1}^m B_{2\ep} (P_i)\big) \cup \Omega_{0}$ where $\Omega_0 = \{ x\in \Omega \mid dist(x, \p \Omega) >  \ep\}$, and  the signed distance function $dist(x, \p \Omega)$, with $dist(x,\p \Omega)>0$ for $x \in \Omega$, is $C^{k+r}$ in $\cup_{i=1}^m B_{3\ep} (P_i)$. 
\end{enumerate} 
Let $\chi(s)$ be a smooth cut-off function such that $\chi (s)=1$ for $s<1$ and $\chi(s)=0$ for $s> 2$. 
Define 
\[
\tilde Y_{i, j} (x)= \chi \big( dist(x, P_i)/\ep -1\big) Y_{i,j}(x), \quad \tilde Y_{0, j} = \chi\big(2- dist(x, \p \Omega)/\ep \big) \Ve_j.
\]
It is clear that $\gamma \in H^k(\Omega, \R^d)$ depends on $\phi$ smoothly iff $\nabla_{\tilde Y_{i, j} \circ \phi} (\tilde Y_{i, j'}  \circ \phi \cdot \gamma) \in H^{k-1} (\Omega, \R)$ depends on $\phi$ smoothly for all $i, j, j'$. For $j'\ne  d$ or $i=0$, this follows directly from the above estimates applied to $\tilde Y_{i, j'}$. For $i \ne 0$ and $j\ne d$, from 
\[
\nabla_{\tilde Y_{i, j} \circ \phi} (\tilde Y_{i, d}  \circ \phi \cdot \gamma) = \nabla_{\tilde Y_{i, d} \circ \phi} (\tilde Y_{i, j}  \circ \phi \cdot \gamma) + \big((\nabla_{\tilde Y_{i, j}} \tilde Y_{i, d} - \nabla_{\tilde Y_{i, d}} \tilde Y_{i, j}) \circ \phi \big)\cdot \gamma, 
\]
we obtain the smooth dependence of $\nabla_{\tilde Y_{i, j} \circ \phi} (\tilde Y_{i, d}  \circ \phi \cdot \gamma) \in H^{k-1} (\Omega, \R)$ on $\phi$ by the above estimates applied to $\tilde Y_{i, j}$ and the smoothness \eqref{E:temp-4.7} of $\gamma \in H^{k-1}(\Omega, \R^d)$ in $\phi$ applied to the lower order commutator term on the right side. Finally from the 
elliptic equation of $h$, the last term $\nabla_{\tilde Y_{i, d} \circ \phi} (\tilde Y_{i, d}  \circ \phi \cdot \gamma)$ can be represented by other second derivative terms (see, for example, \cite{GT01}) and its smoothness in $\phi$ follows. 
\end{proof}

\subsection{Rewriting Euler equation in Lagrangian coordinates} \label{SS:Euler-L}

Suppose $\big(v(t), \rho(t), p(t)\big)$, $t \in [0, T]$, is a solution to the Euler equation \eqref{stratified-euler} such that 
\[
|v - v_*|_{C^0 ([0, T], H_\Euler^k)} \ll 1, \quad \nabla (p(t) - p_*) \in H^k(\Omega, \R^d). 
\]
Let $u(t)$ be an associated Lagrangian map defined in \eqref{E:L-map-1}. In view of \eqref{E:L-map-1} and \eqref{E:temp-2}, one may compute 
\[
v\circ u = u_t = v_* \circ u_* \circ \psi + ((Du_*)\circ \psi) \psi_t,  
\]
\begin{align*}
&(v_t + v\cdot \nabla v) \circ u = u_{tt} \\
=&  (v_* \cdot \nabla v_*) \circ u + 2 ((Dv_*) \circ u)  ((Du_*)\circ \psi) \psi_t + ((D^2 u_*) \circ \psi) (\psi_t, \psi_t) + ((Du_*)\circ \psi) \psi_{tt}. 
\end{align*}
From \eqref{stratified-euler} and \eqref{E:p*}, we obtain 
\be \label{E:Euler-L-2} \begin{split}
((Du_*)\circ \psi) \psi_{tt} + 2 ((Dv_*) \circ u) & ((Du_*)\circ \psi) \psi_t + ((D^2 u_*) \circ \psi) (\psi_t, \psi_t)\\
& +\big( w_* + \frac 1{\rho} \nabla (p- p_*) + \big(\frac 1\rho -  \frac 1{\rho_*} \big) \nabla p_*
\big) \circ u =0.  
\end{split} \ee

Recall from \eqref{density-constant} 
\be \label{E:sigma-1}
\sigma \triangleq \rho(t) \circ u(t) \implies \sigma = \rho(0)\circ u(0), 
\ee
which is independent of $t$. 
Applying 
$I- Q(\sigma, u)$ to the above equation and using 
\[
Q(\sigma, u) \big(\big(\frac 1{\rho} \nabla (p- p_*)\big)\circ u\big) =0, \quad Q(\sigma, u) (w_* \circ u) =w_* \circ u, 
\]
the latter of which is due to assumption \eqref{E:p*}, we obtain 
\begin{align*}
- \big(\frac 1{\rho} \nabla (p- p_*)\big)\circ u =  \big( I-Q(\sigma, u) \big) \Big( &((Du_*)\circ \psi) \psi_{tt} + 2 ((Dv_*) \circ u) ((Du_*)\circ \psi) \psi_t \\
&+ ((D^2 u_*) \circ \psi) (\psi_t, \psi_t) +\big( \big(\frac 1\rho -  \frac 1{\rho_*} \big) \nabla p_* \big) \circ u \Big). 
\end{align*}
For $\psi \in \sG \cap \big( \id_\Omega + B_{{\delta}}(H_\Euler^k) \oplus B_{{\delta}}\big((H^k_\Euler)^{\perp_{\rho_*}}\big)\big)$, according to Proposition \ref{local-coordinate} we can write $\psi= \Psi_{\rho_*} (w)$ and thus 
\be \label{E:temp-9}
\psi_t = \sD \Psi_{\rho_*} (w) w_t, \quad \psi_{tt} = \sD \Psi_{\rho_*} (w) w_{tt} + \sD^2 \Psi_{\rho_*} (w)(w_t, w_t). 
\ee
According to Proposition \ref{local-coordinate}, we have $\sD \Psi_{\rho_*} (w) w_{tt} \in T_\psi \sG$. 
Since $u_* \circ \phi \in \tilde \sG$ for any $\phi \in \tilde \sG$, differentiating in $\phi$ we obtain that 
\[
\big( ((Du_*)\circ \psi) \tilde \psi \big) \circ u^{-1} \in H_\Euler^k, \; \forall \tilde \psi \in T_\psi \sG \implies (I-Q(\sigma, u)) ((Du_*)\circ \psi) \sD \Psi_{\rho_*} (w) w_{tt} =0. 
\] 
Hence we obtain the equation of the pressure from the above analysis 
\be \label{E:pressure-1} \begin{split} 
\big(\frac 1{\rho}  \nabla  (p &- p_*) \big)\circ u = 
- \big(I- Q_{u_*} (\sigma, \psi) \big) \Big( ((Du_*)\circ \psi)  \sD^2 \Psi_{\rho_*} (w)(w_t, w_t) \\
&+ 2 ((Dv_*) \circ u) ((Du_*)\circ \psi) \psi_t + ((D^2 u_*) \circ \psi) (\psi_t, \psi_t) +\big(\frac 1{\sigma} -  \frac 1{\rho_* \circ u} \big) \big( ( \nabla p_* ) \circ u\big) \Big), 
\end{split} \ee
where we recall $Q_{u_*}$ as defined in \eqref{E:cQ} and 
\[
u_* = u_*(t), \quad \psi=\Psi_{\rho_*}(w), \quad \psi_t = \sD \Psi_{\rho_*} (w) w_t, \quad u= u_*(t) \circ \psi.
\] 

Inserting \eqref{E:temp-9} into \eqref{E:Euler-L-2}, we have 
\begin{align*}
\sD \Psi_{\rho_*} & (w) w_{tt} +  \sD^2 \Psi_{\rho_*} (w)(w_t, w_t) +  ((Du_*)^{-1}\circ \psi) \Big( 2 ((Dv_*) \circ u)  ((Du_*)\circ \psi) \psi_t \\
& + ((D^2 u_*) \circ \psi) \big(\psi_t, \psi_t) 
+ w_* \circ u  + \big(\frac 1{\sigma} -  \frac 1{\rho_* \circ u} \big) \big( ( \nabla p_* ) \circ u \big) 
+ \big(\frac 1{\rho}  \nabla  (p- p_*) \big)\circ u 
\Big) =0. 
 \end{align*}
According to Proposition \ref{local-coordinate}, $\cP (\rho_*) \sD \Psi_{\rho_*} (w) =I$ at any $w \in B_\delta (H_\Euler^k)$. Hence applying $\cP (\rho_*)$ to the above equation yields
\be \label{E:Euler-L-3} 
w_{tt} + \cF(t, \sigma, w, w_t)=0, 
\ee
where 
\be \label{E:cF-1} \begin{split}
\cF(t, \sigma, w,  w_t  ) = & \cP (\rho_*) \sD^2 \Psi_{\rho_*} (w)(w_t, w_t) \\
&+ \cP (\rho_*) ((Du_*)^{-1}\circ \psi) \Big( 2 ((Dv_*) \circ u)  ((Du_*)\circ \psi) \psi_t  + ((D^2 u_*) \circ \psi) \big(\psi_t, \psi_t) \\
&+ w_* \circ u  + \big(\frac 1{\sigma} -  \frac 1{\rho_* \circ u} \big) \big( ( \nabla p_* ) \circ u \big) + \big(\frac 1{\rho}  \nabla  (p- p_*) \big)\circ u
\Big). 
\end{split}\ee
This is the form of the Euler equation which we shall work on. The correspondence between the Eulerian formulation \eqref{stratified-euler} and the Lagrangian formulation \eqref{E:Euler-L-3} is given by 
\be \label{E:transform} 
\begin{pmatrix} \rho \\ v \end{pmatrix} = \Lambda(t, \sigma, w, w_t) = \begin{pmatrix} 
\sigma \circ \Psi_{\rho_*}(w)^{-1} \circ u_*^{-1} \\ 
v_* + \big( \big( (Du_*) \circ \Psi_{\rho_*} (w) \big) \sD \Psi_{\rho_*} (w) w_t \big) \circ \Psi_{\rho_*}(w)^{-1} \circ u_*^{-1}
\end{pmatrix}. 
\ee
Clearly $\cF$ contains terms of $w_t \in H_\Euler^k$ of degree $0,1$, and $2$, so we may write it as
\[
\cF(t, \sigma, w, w_t) = \cF_0(t, \sigma, w) + \cF_1(t, \sigma, w) w_t + \cF_2(t, \sigma, w) (w_t, w_t), \quad \cF_j = \sD_{w_t}^j \cF|_{w_t=0},
\]
which are computed explicitly using \eqref{E:pressure-1}, 
\be \label{E:cF0} \begin{split}
\cF_0(t, \sigma, w) = & \cP (\rho_*) ((Du_*)^{-1}\circ \psi) \Big( w_* \circ u  +Q_{u_*} (\sigma, \psi)  \big(\frac 1{\sigma} -  \frac 1{\rho_* \circ u} \big) \big( ( \nabla p_* ) \circ u \big)\Big) .
\end{split}\ee
\be \label{E:cF1} \begin{split}
\cF_1(t, \sigma, w)w_t = & 2 \cP (\rho_*) ((Du_*)^{-1}\circ \psi) Q_{u_*} (\sigma, \psi) \big(((Dv_*) \circ u)  ((Du_*)\circ \psi) \sD \Psi_{\rho_*}(w)w_t\big).
\end{split}\ee
\be \label{E:cF2} \begin{split}
\cF_2(t, \sigma, w)(w_t, w_t) = & 
\cP (\rho_*) ((Du_*)^{-1}\circ \psi) Q_{u_*} (\sigma, \psi) \Big( ((Du_*)\circ \psi)  \sD^2 \Psi_{\rho_*} (w)(w_t, w_t) \\
&+ ((D^2 u_*) \circ \psi) \big(\psi_t, \psi_t)\Big).
\end{split}\ee

\begin{lemma} \label{L:Euler-L-1}
Assume $(\rho_*, v_*, p_*)$ satisfy assumptions \eqref{E:BG-density}, \eqref{E:BG-velocity}, and \eqref{E:p*}, then there exist $\delta >0$ determined by $\Omega, k+r, \rho_\pm, |\rho_*|_{Z^k}$, $K\in \mathbb{N}$ determined by $k+r$, and $C>0$ determined by $k+r, \rho_\pm, \Omega$, $|\rho_*|_{Z^{k+r}}$, $|D v_*|_{Z^{k+r-1}}$, $|\nabla \rho_* \cdot v_*|_{H^{k+r-1}}$, $|w_*|_{H^{k+r-1}}$, and $|\nabla p_*|_{Z^{k+r-1}}$ such that 
\[
\cF_0 \in C^{r-1} (\R \times (\rho_* + B_\delta (H^k(\Omega, \R))) \times B_\delta (H_\Euler^k),  H_\Euler^k), 
\]
\[
\cF_1 \in C^{r-1} (\R \times (\rho_* + B_\delta (H^k(\Omega, \R))) \times B_\delta (H_\Euler^k),  \cL(H_\Euler^k)), 
\]
\[
\cF_2 \in C^{r-2} \big(\R \times (\rho_* + B_\delta (H^k(\Omega, \R))) \times B_\delta (H_\Euler^k),  \cL (H_\Euler^k \otimes H_\Euler^k, H_\Euler^k)\big), 
\] 
and for any $t \in \R$, and $r_1, r_2\ge 0$, 
\[
\|\p_t^{r_1} \cF_{1} (t, \cdot, \cdot) \|_{C^{r_2}}  \le C |Du_*(t)|_{Z^{k+r-1}}^K, \quad r_1+r_2 \le r-1; 
\]
\[ 
\|\p_t^{r_1} \cF_2 (t, \cdot, \cdot) \|_{C^{r_2}} \le C |Du_*(t)|_{Z^{k+r-1}}^K, \quad r_1+r_2 \le r-2;
\]
\[
\|\cF_{0} (t, \cdot, \cdot) \|_{C^{r-1}} \le C\big( |Du_*(t)|_{Z^{k+r-1}}^K + \big| \int_0^t  |D u_*(\tau)|_{Z^{k+r-1}} d\tau\big| \big). 
\]
\end{lemma}

For fixed $t$ the above $C^{r_2}$ norms are taken on $(\rho_* + B_\delta (H^k(\Omega, \R))) \times B_\delta (H_\Euler^k)$. One notices that the term $D^2u_*$ in the definition causes the $C^{r-2}$ smoothness of $\cF$. This term only appears in $\cF_2$, so $\cF_0$ and $\cF_1$ are $C^{r-1}$. 

\begin{proof} 
We choose $\delta>0$ such that Proposition \ref{local-coordinate} and Lemma \ref{L:Hodge-2} hold and $\rho_* + B_\delta (H^k(\Omega, \R)) \subset Z_{\frac {\rho_-}2, 2\rho_+}^k$. 
The Lagrangian map $u_*$ associated with the background velocity $v_*$ is defined as in \eqref{lagrangian-coordinate-map}. From the assumption $D v_*\in Z^{k+r-1}$, differentiating in $t$ and using \eqref{E:composition-1} repeatedly, we have 
\[
D u_* \in C^{r_1} (\R, Z^{k+r-r_1}), \quad |\p_t^{r_1} Du_*(t) |_{Z^{k+r-r_1}} \le C |Du_*(t)|_{Z^{k+r-1}}^K, \quad 1\le r_1 \le r. 
\]
Using Lemma \ref{L:Hodge-2} (with its $(l, r)$ replaced by $(r_1, r_2)$), Lemma \ref{hodge}, and Proposition \ref{local-coordinate}, it is straightforward to obtain the smoothness and estimates of $\cF_1$ and $\cF_2$. 

It remains to analyze $\cF_0$, particularly the last term involving $\frac 1{\sigma} -  \frac 1{\rho_* \circ u}$. One the one hand, from the assumptions $\rho_* \in Z_{\rho_, \rho_+}^{k+r}$ and $\nabla \rho_* \cdot v_* \in H^{k+r-1}$, we have  
\[
\p_t (\rho_* \circ u) = (\nabla \rho_* \cdot v_*) \circ u_* \circ \Psi(w) \implies \|\p_t^{r_1} (\rho_* \circ u)\|_{C^{r-r_1}(\R \times B_\delta (H_\Euler^k), H_\Euler^k)} \le  C |Du_*(t)|_{Z^{k+r-1}}^K,
\]
for $1\le r_1 \le r$. on the other hand,   
\begin{align*}
& \sigma  - \rho_* \circ u=  \sigma -\rho_*+ \rho_* - \rho_*\circ u_*(t) + \rho_*\circ u_*(t) - \rho_*\circ u_* (t) \circ \psi \\
=& \sigma -\rho_* - \int_0^t (\nabla \rho_* \cdot v_*) \circ u_*(\tau) d\tau - \int_0^1 \Big( \big((Du_*(t))^T (\nabla \rho_*) \circ u_* (t)\big) \circ \Psi_{\rho_*} (\tau w) 
\sD \Psi_{\rho_*} (\tau w) w d\tau. 
\end{align*}
Therefore, along with the above estimate on $\p_t^{r_1} (\rho_* \circ u)$ we obtain 
\[
\big\| \frac 1{\sigma} -  \frac 1{\rho_* \circ u}\big\|_{C^{r-1} (\R \times (\rho_* + B_\delta (H^k(\Omega, \R))) \times B_\delta (H_\Euler^k), H_\Euler^k)} \le C\big( |Du_*(t)|_{Z^{k+r-1}}^K + \big| \int_0^t  |D u_*(\tau)|_{Z^{k+r-1}} d\tau\big| \big). 
\]
Using these inequalities and \eqref{E:composition-1}, it is straightforward to derive the desired estimates on $\cF_0$. 
\end{proof}

A direct corollary of the above lemma is the local well-posedness of the stratified  Euler equation with a background flow. 

\begin{proof}[Proof of Theorem \ref{T:LWP}] 
According to \eqref{lagrangian-coordinate-map}, for any solution with $v \in C^0( [-T, T], v_* + H_\Euler^k)$, the corresponding Lagrangian map $u(t, \cdot) \in \tilde \sG$ is well-defined and $\rho(t) = \rho(0) \circ u(t)^{-1}$. Therefore solutions to the stratified Euler equation \eqref{stratified-euler} correspond to those of \eqref{E:Euler-L-3} on $w \in C^1 ([-T, T], B_\delta (H_\Euler^k))$ with $w(0) =0$, $w_t(0)= v(0) - v_*$, and $\sigma = \rho(0)$. Due to Lemma \ref{L:Euler-L-1}, \eqref{E:Euler-L-3} is an infinite dimensional ODE on $
B_\delta (H_\Euler^k) \times H_\Euler^k$ with $C^{r-2}$ nonlinearity depending on the density parameter $\sigma$. Therefore the existence, uniqueness, and smooth dependence on the initial data (as well as on the parameter $\sigma = \rho(0)$) of solutions follow immediately.  The pressure is given \eqref{E:pressure-1} and $\nabla (p - p_*) \in H^k$. The conclusion $\rho=\sigma \circ u^{-1}  \in C^0([-T, T], \rho_* + H^k(\Omega, \R))$ follows from the above analysis on $\sigma - \rho_* \circ u$. 
\end{proof} 

\begin{remark} \label{R:non-smooth}
While the Euler equation in the Lagrangian formulation \eqref{E:Euler-L-3} is an ODE on the infinite dimensional space $
B_\delta (H_\Euler^k
)  \times H_\Euler^k$, its Eulerian formulation \eqref{stratified-euler} does not represent an evolution governed by a smooth infinite dimensional vector field in a neighborhood of $(\rho_*, v_*)$ in $H^k(\Omega, \R)   \times H_\Euler^k$. This discrepancy is due to the lack of smoothness of the transformation $\Lambda(t, \sigma, \cdot)$ as defined in \eqref{E:transform}. Firstly, the variation of $\Psi_{\rho_*} (w)^{-1}$ is given by 
\be \label{E:non-smooth-1}
\sD \big( \Psi_{\rho_*} (w)^{-1} \big) \tilde w = - \big( \big( D \Psi_{\rho_*} (w) \big)^{-1}\big( \sD \Psi_{\rho_*} (w) \tilde w \big) \big) \circ \Psi_{\rho_*} (w)^{-1}.  
\ee
Secondly, since ``$\circ \Psi_{\rho_*} (w)^{-1}$'' appears in \eqref{E:transform}, let us consider the composition $\cC(f, \psi) = f\circ \psi$. The variation with respect to $\psi$ is 
\be \label{E:non-smooth-2}
\sD_\psi \cC(f, \psi) \tilde \psi = \big( (D f) \circ \psi \big) \tilde \psi.   
\ee
In both formulas, the spatial derivative $D$ causes the loss of one order of regularity. Following these calculations and the argument in the proof of Lemma \ref{L:Euler-L-1}, one may prove 
for $0\le r_1 \le k$,  
\be \label{E:Lambda-Cr}
\Lambda  \in C^{r_1} \big( \R \times (\rho_* + B_\delta (H^k(\Omega, \R))) \times B_\delta (H_\Euler^{k}) \times H_\Euler^k, (\rho_*, v_*) + H^{k-r_1}(\Omega, \R) \times  H_\Euler^{k-r_1})  \big). 
\ee
\end{remark}

\section{Local unstable manifolds: proof of Theorem \ref{main-theorem}} \label{S:InMa}

Throughout this section, we assume that $(v_0, \rho_0, p_0)$ is an equilibrium of the  Euler equation \eqref{stratified-euler} satisfying \eqref{E:steady}, \eqref{u0-L-infinity}, \eqref{E:Lyap-E}, and the exponential dichotomy (ED). We shall construct its local  unstable manifold in this section.  

Throughout this section, let $\Psi_{\rho_0}$ be the mapping given by Proposition \ref{local-coordinate} , $u_0(t)$ be the Lagrangian map associated with $v_0$ as defined in \eqref{lagrangian-coordinate-map}, and $\cF(t, \sigma, w, w_t)$ be the mapping defined as in \eqref{E:cF-1} with $(v_*, \rho_*, p_*)$ substituted by $(v_0, \rho_0, p_0)$. Here $(w, w_t) \in B_\delta (H_\Euler^k) \times H_\Euler^k$ and $\delta>0$ is given in Lemma \ref{L:Euler-L-1}. 
By standard ODE estimates there exists $C\ge 1$ depending on  $\|D^2 v_0\|_{H^{k+r-2}}, d, k+r$, and $C_*, \mu$,  $\mu_0$, given in \eqref{u0-L-infinity} and \eqref{E:Lyap-E} such that, 
\begin{equation}   \label{u0-hl}
    \left\|D^2u_0(t,\cdot)\right\|_{H^{m}} + \| D\left( Du_0(t,\cdot)\right)^{-1}\|_{H^{m}} \leq C e^{(m+2) \mu |t|}, \quad \forall t\in \R, \; m \in [0,k+r-2].
\end{equation}

\begin{remark} 
Assumption \eqref{u0-L-infinity} and inequalities \eqref{E:composition-1} and \eqref{u0-hl} yield the fact that the composition with $u_0(t)$ or $u_0(t)^{-1}$ introduces only an additional $C e^{K \mu |t|}$ in various estimates. This will be used repeatedly in the rest of the section.  
\end{remark}

Under $H^k$ perturbations to $(v_0, \rho_0)$, the Euler equation is equivalent to \eqref{E:Euler-L-3}. In the Lagrangian formulation \eqref{E:Euler-L-3}, $\rho(t) \circ u(t) = \sigma$, where $u(t)$ is any Lagrangian map associated with a solution $(v(t), \rho(t))$ (without assuming $u(0)=\id_\Omega$), is $t$-independent and thus can be treated as an external parameter. 
Since $(v_0, \rho_0, p_0)$ is a steady state, we have 
\[
w_*\triangleq v_0 \cdot \nabla v_0 + g \Ve_d + \frac 1{\rho_0} \nabla p_0=0, \quad \rho_0 \circ u_0(t) = \rho_0, \;\; \forall t\in \R, 
\]
and thus \eqref{E:cF-1}--\eqref{E:cF2} yield 
\be \label{E:relabeling-2}
\cF(t, \rho_0 \circ \Psi_{\rho_0} (w), w, 0)=0, \; \forall w\in B_\delta (H_\Euler^k).   
\ee
Hence $w(t) =w_0$ is a steady solution with the density parameter $\sigma = \rho_0 \circ \Psi_{\rho_0} (w_0)$, all of which correspond to relabelings of $(v_0, \rho_0, p_0)$. 


We are concerned with solutions $(v(t), \rho(t))$ to the stratified Euler equation \eqref{stratified-euler} on the unstable manifold of $(v_0, \rho_0, p_0)$ characterized by the exponential decay \eqref{E:UnstableS-1} as $ t\to -\infty$. In the Lagrangian formulation, we have the following lemma where $\mu$ and $\lambda$ are given in \eqref{u0-L-infinity}, \eqref{E:Lyap-E}, and \eqref{E:UnstableS-1}, respectively.

\begin{lemma} \label{L:UnstableS-L}
There exists $K \in \mathbb{N}$ determined by $k$ such that, if $\lambda > K\mu$, then there exists $C>0$ depending on  $k, \Omega, \rho_\pm, C_*, \mu$,  $\mu_0$, $|\rho_0|_{Z^k}, |D v_0|_{Z^{k}}$, and $\lambda-K \mu$ such that for any solution $(v(t), \rho(t))$, $t\le 0$, to \eqref{stratified-euler} satisfying \eqref{E:UnstableS-1} with $C \ep < \delta$,  
there exists a unique solution $w(t)$, $t \le 0$, to \eqref{E:Euler-L-3} with $\sigma=\rho_0$ such that 
\be \label{E:U-solu-1} 
\Lambda\big(t, \rho_0, w(t), w_t(t)\big) = \big(\rho(t), v(t) \big)^T, \;\; \|w(t)\|_{H^k}+ \|w_t(t)\|_{H^k}  \le C \ep  e^{(\lambda - K\mu)t}, \;\; \forall t\le 0. 
\ee
\end{lemma}

\begin{proof} 
Let $\tilde u(t) = u_0(t) \circ \Psi_{\rho_0}( \tilde w(t))$, $t\le 0$, be a Lagrangian map associated with $v(t)$ defined as in \eqref{E:L-map-1}. From \eqref{E:transform}, \eqref{u0-L-infinity}, \eqref{E:composition-1}, \eqref{u0-hl}, and Proposition \ref{local-coordinate} we have   
\be \label{E:temp-10}
\|\tilde w_t\|_{H^k(\Omega, \R^d)} \le C e^{- K\mu t} \|v(t)-v_0\|_{H^k(\Omega, \R^d)}, \quad \forall t\le 0, 
\ee
for some $K$ and $C$ as described in the lemma. Hence we obtain $\lim_{t\to -\infty} \tilde w(t)$ exists and belongs to $B_{C\ep} (H_\Euler^k)$ which implies $\phi = \lim_{t\to -\infty} \Psi_{\rho_0}( \tilde w(t))$ also exists. From \eqref{E:temp-10} and \eqref{E:composition-1} it holds,  
\[
\|\Psi_{\rho_0}( \tilde w(t)) \circ \phi^{-1} -id_\Omega\|_{H_\Euler^k} =\big\| \big(\Psi_{\rho_0}( \tilde w(t)) - \phi\big) \circ \phi^{-1} \|_{H_\Euler^k} \le C \ep. 
\] 
According to Proposition \ref{local-coordinate}, there exists $w(t) \in B_{C\ep} (H_\Euler^k)$ such that 
\[
\Psi_{\rho_0}( w(t)) = \Psi_{\rho_0}( \tilde w(t)) \circ \phi^{-1} \implies w_t = \sD \Psi_{\rho_0}( w)^{-1} \big ( (\sD \Psi_{\rho_0}( \tilde w) \tilde w_t) \circ \phi^{-1} \big), \; w(-\infty) =0. 
\] 
The desired exponential decay estimates on $w(t)$ and $w_t(t)$ follow directly from \eqref{E:temp-10}. 
Let 
\[
u(t) = \tilde u(t) \circ \phi^{-1} = u_0(t) \circ \Psi_{\rho_0}( \tilde w(t)) \circ \phi^{-1} = u_0(t) \circ \Psi_{\rho_0}(w(t)). 
\]
Clearly $u(t)$ is also a Lagrangian map associated with $v$ and thus $w(t)$ solves \eqref{E:Euler-L-3} with 
\[
\sigma = \rho(t) \circ u(t) =\lim_{t\to -\infty} \rho(t) \circ u(t)= \lim_{t\to -\infty} \big( (\rho(t)- \rho_0) \circ u(t) + \rho_0 \circ \Psi_{\rho_0}(w(t)) \big), 
\]
where $\rho_0 \circ u_0(t) = \rho_0$ was also used. From \eqref{E:UnstableS-1} and the decay of $w(t)$ we obtain $\sigma =\rho_0$.  

To prove the uniqueness, assume that $\bar w(t)$ is another solution to \eqref{E:Euler-L-3} with $\sigma=\rho_0$ satisfying \eqref{E:U-solu-1}. Let 
\[
\psi = \Psi_{\rho_0} (w(0))^{-1} \circ \Psi_{\rho_0} (\bar w(0)). 
\]
Since $u_0(t) \circ \Psi_{\rho_0} (\bar w(t))$ is also a Lagrangian map of $v(t)$, \eqref{E:L-map-1} and the definition of $\psi$  yield 
\[
u_0(t) \circ \Psi_{\rho_0} (\bar w(t)) = u_0(t) \circ \Psi_{\rho_0} ( w(t)) \circ \psi \implies \Psi_{\rho_0} (\bar w(t)) = \Psi_{\rho_0} ( w(t)) \circ \psi, \quad \forall t\le 0. 
\]
Letting $t \to -\infty$, \eqref{E:U-solu-1} implies $\psi=\id_\Omega$, which finishes the proof. 
\end{proof}

From this lemma, in order to obtain exponentially decaying solutions $(v(t), \rho(t))$ to \eqref{stratified-euler},  it suffices to work on solutions $w(t)$ to \eqref{E:Euler-L-3} 
satisfying, for some $\lambda \in (\lambda_{cs}, \lambda_u)$, 
\be \label{E:U-solu-2} 
\sup_{t\le 0} e^{-\lambda t} \big(\|w(t)\|_{H^k} + \|w_t(t)\|_{H^k} \big) < \infty, \; \text{ with parameter } \sigma =\rho_0.
\ee

In Subsection \ref{SS:ED-L} we first obtain the linear exponential dichotomy in the Lagrangian formulation. 
The local invariant unstable manifolds will be constructed in Lagrangian coordinates in Subsection \ref{SS:InMa-L} by the Lyapunov-Perron method, which will be transformed back into the Eulerian formulation in Subsection \ref{SS:InMa-E}.

\subsection{Linearization in Lagrangian formulation} \label{SS:ED-L}

We rewrite (\ref{E:Euler-L-3}) as a first order system and expand $\cF(t, \rho_0, w, w_t)$ at $w=w_t=0$:
\begin{equation}
    \label{stratified-euler-taylor}
         z_t = A(t) z + \sF(t, z), \quad z = (z_1,z_2)^T \in B_{\delta_0}(H^k_\Euler) \times H^k_\Euler,
\end{equation}
where 
$$
A(t) = \begin{pmatrix}
    0 & I \\
    -\sD_{z_1}\cF_0(t,\rho_0, 0) & -\cF_1(t, \rho_0, 0)
\end{pmatrix}, \; \text{ and }\; \mathscr F (t, z) = \begin{pmatrix}
    0 \\
    -\tilde{\cF}(t,z_1,z_2)
\end{pmatrix}, 
$$
$$
\tilde{\cF} (t, z)=\cF_0(t,\rho_0, z_1)-\sD_w\cF_0(t, \rho_0, 0)z_1+ \left(\cF_1(t,\rho_0, z_1) - \cF_1(t, \rho_0, 0)\right) z_2+ \cF_2(t, \rho_0, z_1) (z_2, z_2).
$$
The linearization of (\ref{E:Euler-L-3}) at $(0,0)$ takes the  form:
\begin{equation}
    \label{z-linearized-eqn}
         z_t = A(t) z
        \iff
        w_{tt} + \sD_w \cF_0(t, \rho_0, 0)w + \cF_1(t, \rho_0, 0)w_t= 0.
\end{equation}
Let 
\[
T(t,t_0) = \big( T_1(t,t_0) , T_2(t,t_0) \big)^T \in \cL \big( (H^k_\Euler)^2 \big)
\]
denote the solution map of (\ref{z-linearized-eqn}) with initial time $t_0$ and terminal time $t$. Namely $T(t, t_0) z(t_0)$ is the solution to (\ref{z-linearized-eqn}). Recall the linearized operator of the Euler equation at $(v_0, \rho_0, p_0)$ is denoted by $L$ as given in (\ref{linearization}). Let $\pi_v$ and $\pi_\rho$ denote the projection operators to the $v$ and $\rho$ components of solutions to \eqref{stratified-euler}, respectively. We also recall the splitting $H^k_\Euler\times H^k(\Omega,\mathbb{R}) = X_u \oplus X_{cs}$ in assumption (ED) with $L_{u, cs} \triangleq L|_{X_{u,cs}}$ and let $\bP_{u, cs}$ denote the associated projections. 

To establish the relationship between the linearized solutions (denoted by $\tilde v$, $\tilde w$, {\it etc.}) in Eulerian and Lagrangian formulations, we start with linearizing \eqref{E:transform} at $w(t)\equiv 0$ to derive 
\begin{equation}    \label{E:Lambda0}
\begin{pmatrix} \tilde \rho \\ \tilde v \end{pmatrix} = \Lambda_0(t)\begin{pmatrix} \tilde w \\ \tilde w_t \end{pmatrix}, \; \text{ where } \; \Lambda_0(t)z \triangleq  \sD \Lambda(t, \rho_0, 0, 0)z  = \begin{pmatrix}  - \nabla \rho_0 \cdot \big ( Du_0(t) z_1 \big) \circ u_0(t)^{-1} \\ \big(D u_0(t) z_2 \big) \circ u_0(t)^{-1} \end{pmatrix}.
    \end{equation}
Clearly $\Lambda_0(t) \in \cL\big( (H_\Euler^k)^2,  H^k(\Omega,\mathbb{R}) \times H^k_\Euler \big)$ and it is strongly continuous in $t$.     
By the equivalence of the Euler equation in the Eulerian and Lagrangian formulations,  we have 
\be \label{E:conjugacy}
\Lambda_0(t) T(t, t_0)  z = e^{(t-t_0)L} \Lambda_0 (t_0) z. 
\ee
While $\Lambda_0(t)$ is not invertible, one can still recover $\tilde w_t$ from $\tilde v$. 
Consequently we obtain a formula of $T(t, t_0)$ as 
\be \label{E:T}
T_2(t, t_0) z  = (D u_0(t))^{-1} \big(\pi_v e^{(t-t_0)L} \Lambda_0(t_0) z 
\big) \circ u_0(t), \quad 
T_1(t, t_0) z  = z_1 + \int_{t_0}^t T_2 (\tau, t_0)z d\tau. 
\ee
Define the $t$-dependent operators 
\be \label{E:U-1} 
U(t) \begin{pmatrix} \bar \rho \\ \bar v  \end{pmatrix}  
=  \begin{pmatrix}
       \int_{-\infty}^t (Du_0(s))^{-1}\big(\pi_v e^{sL_u} (\bar \rho, \bar v)^T \big) \circ u_0(s) ds 
       \\
        (Du_0(t))^{-1} (\pi_v e^{tL_u} (\bar \rho, \bar v)^T ) \circ u_0(t) 
    \end{pmatrix}, \quad \begin{pmatrix} \bar \rho \\ \bar v  \end{pmatrix} \in X_u, 
\ee
\be \label{E:bP_u}
\bP_u(t) = U(t) e^{-tL_u} \bP_u \Lambda_0 (t). 
\ee
To prove that $\bP_u(t)$ gives the {\it linear exponential dichotomy in Lagrangian formulation}, we start with the following auxiliary lemma. 

\begin{lemma} \label{L:X_u-rho} 
There exists $K \in \mathbb{N}$ determined by $k$ such that, if $\lambda_u > K\mu$, then for any $(\bar \rho, \bar v) \in X_u$, let $(\tilde \rho(t), \tilde v(t))^T = e^{tL_u}(\bar \rho, \bar v)^T$, then   it holds
\[
\tilde \rho (t)  = - \nabla \rho_0 \cdot \int_{-\infty}^t (Du_0(s-t))^{-1}\big( \tilde v(s) \circ u_0(s-t) \big) ds. 
\]
\end{lemma}

\begin{proof} 
As a solution in $X_u$ to the Euler equation linearized at $(\rho_0, v_0, p_0)$, $(\tilde \rho(t), \tilde v(t))$ satisfies 
\be \label{E:temp-11} \begin{split} 
&\tilde \rho_t + v_0 \cdot \nabla \tilde \rho + \tilde v \cdot \nabla \rho_0=0 \Longleftrightarrow \p_t (\tilde \rho \circ u_0) = -  (\tilde v \cdot \nabla \rho_0) \circ u_0, \\
& \|\tilde \rho(t)\|_{H^k} + \|\tilde v(t) \|_{H^k} \le C e^{\lambda_u t}, \; t\le 0, 
\end{split}  \ee
which implies 
\[
\tilde \rho(t)  \circ u_0(t) = - \int_{-\infty}^t \big(\tilde v(s) \cdot \nabla \rho_0 \big) \circ u_0(s) ds,
\]
if $\lambda_u > k \mu$. Since $(\rho_0, v_0, p_0)$ is steady, we have  
\[
\rho_0 \circ u_0 (s-t)  =\rho_0 \implies (\nabla \rho_0  \circ u_0 (s-t) ) \cdot D u_0 (s-t ) X = X \cdot \nabla \rho_0, \; \forall X \in \R^d.
\]
The lemma follows simply from applying this identity to the above integrand where $X$ is substituted by $(Du_0(s-t))^{-1}\big( \tilde v(s) \circ u_0(s-t) \big)$. 
\end{proof}  

Next we give some basic properties of $U(t)$. 

\begin{lemma} \label{L:U}
There exists $K \in \mathbb{N}$ determined by $k$ such that, if $\lambda_u > K\mu$, then there exists 
$C>0$ depending on $k, \Omega, \rho_\pm, C_*, \mu$,  $\mu_0$, $|\rho_0|_{Z^k}, |D v_0|_{Z^{k}}$,  $\lambda_u  -K \mu$, such that
\[
\|U(t)\|_{\cL( X_u, (H^k_\Euler)^2)} \le C e^{(\lambda_u -K\mu)t}, \quad  \Lambda_0(t) U(t) = e^{tL_u}, \quad \forall t\le 0. 
\]
\end{lemma} 

\begin{proof} 
It is straightforward to verify the exponential estimate of $U(t)$ using \eqref{E:composition-1}, \eqref{u0-L-infinity}, \eqref{u0-hl}, and assumption (ED). To prove the equality, let us use the same notation $(\tilde \rho(t), \tilde v(t))^T = e^{tL_u}(\bar \rho, \bar v)^T$ as above. Direct computation based on the definitions of $\Lambda_0(t)$ and $U(t)$ shows 
\[
\Lambda_0 (t) U(t) \begin{pmatrix} \bar \rho \\ \bar v  \end{pmatrix} = \begin{pmatrix}
       - \nabla \rho_0 \cdot \int_{-\infty}^t \big( Du_0(t) (Du_0(s))^{-1} \big(\tilde v(s) \circ u_0(s) \big) \big)\circ u_0(t)^{-1} ds 
              \\    \tilde v(t) 
    \end{pmatrix}. 
\]
Since $u_0(s)$ is the flow map of the autonomous vector field $v_0$ on $\Omega$, it satisfies 
\[
u_0(s) \circ u_0(t)^{-1} = u_0(s-t), \quad 
Du_0(t) (Du_0(s))^{-1} = (D u_0(t-s)) \circ u_0(s). 
\]
Hence we obtain from Lemma \ref{L:X_u-rho}
\begin{align*}
\pi_\rho \Lambda_0(t)  U(t) ( \bar \rho, \bar v)^T = & -\nabla \rho_0 \cdot \int_{-\infty}^t  \big( Du_0(t-s) \tilde v(s) \big) \circ u_0(s-t) ds 
= \tilde \rho(t),
\end{align*}
which yields the desired identity. 
\end{proof}

We are ready to prove some basic properties of $\bP_u(t)$. 

\begin{lemma} \label{L:bP_u} 
There exists $K \in \mathbb{N}$ determined by $k$ such that, if $\lambda_u > K\mu$, then there exists $C>0$ depending on $k, \Omega, \rho_\pm, C_*, \mu$,  $\mu_0$, $|\rho_0|_{Z^k}, |D v_0|_{Z^{k}}$, and $\lambda_u  -K \mu$ such that for all $t \le 0$, 
\[
\Lambda_0(t) \bP_u(t) = \bP_u \Lambda_0(t), \quad \bP_u(t)^2 = \bP_u (t), \quad \bP_u(t)U(t)=U(t), \quad \left\|\bP_{u}(t)\right\|_{\cL\left(\left(H^k_\Euler\right)^2 \right)} \leq C e^{-K\mu t}.
\]
\end{lemma}

The first two equalities in the lemma imply that $\bP_u(t)$ is a projection operator corresponding to $\bP_u$ through $\Lambda_0 (t)$.

\begin{proof}
The boundedness and the estimates follow directly from \eqref{E:Lambda0}, \eqref{E:composition-1}, \eqref{u0-L-infinity}, \eqref{u0-hl}, and assumption (ED). It remains to show the equalities. From Lemma \ref{L:U}, the definition of $\bP_u(t)$ and the invariance of $X_u$ under $e^{tL}$ assumed in (ED), it is direct to compute
\[
\Lambda_0(t) \bP_u(t) = \Lambda_0(t) U(t) e^{-tL_u} \bP_u \Lambda_0 (t)  =e^{tL_u} e^{-tL_u} \bP_u \Lambda_0 (t) = \bP_u \Lambda_0 (t);
\]
\begin{align*}
\bP_u(t)^2 =& U(t) e^{-tL_u} \bP_u \Lambda_0 (t) U(t) e^{-tL_u} \bP_u \Lambda_0 (t) \\
= & U(t) e^{-tL_u} \bP_u  e^{tL_u} e^{-tL_u} \bP_u \Lambda_0 (t) = U(t) e^{-tL_u}\bP_u \Lambda_0 (t) = \bP_u(t);
\end{align*}
\[
\bP_u(t)U(t)= U(t) e^{-tL_u} \bP_u \Lambda_0 (t) U(t) = U(t) e^{-tL_u} \bP_u e^{tL_u} = U(t).
\]
The proof of the lemma is complete. 
\end{proof}

The $t$-dependent center-stable and unstable subspaces are defined as 
\be \label{E:Y(t)} 
Y_u(t) = \bP_u(t) (H_\Euler^k)^2, \quad Y_{cs}(t) = \ker \bP_u(t), \quad t\le 0. 
\ee

\begin{proposition} \label{time-dependent-dichotomy}
There exists $K \in \mathbb{N}$ determined by $k$ such that, if $\lambda_u >  K\mu$, then there exists $C>0$ depending on $k, \Omega, \rho_\pm, C_*, \mu, \mu_0, \lambda_{cs}$, $|\rho_0|_{Z^k}, |D v_0|_{Z^{k}}$, and $\lambda_u  -K \mu$ such that the following hold for all $t, t_0 \le 0$. 
\begin{enumerate} 
\item $Y_u(t), Y_{cs}(t) \subset (H_\Euler^k)^2$ are closed subspaces and $(H_\Euler^k)^2 = Y_u(t) \oplus Y_{cs}(t)$. 
\item $Y_u(t) = U(t) X_u$ and $U(t): X_u \to Y_u(t)$ is an isomorphism. 
\item For any $t_0, t \le 0$, it holds 
\be \label{E:TbP_u} 
T(t, t_0) \bP_u(t_0) = \bP_u(t) T(t, t_0), \quad T(t, t_0) Y_{u, cs} (t_0) \subset Y_{u, cs}(t). 
\ee
\item Let $T_{u, cs} (t, t_0) = T(t, t_0)|_{Y_{u, cs} (t_0)}$, then we also have 
$$\|T_{u}(t,t_0)\|_{\cL \left(\left(H^k_\Euler\right)^2 \right)} \leq C e^{(\lambda_u - K\mu)(t-t_0)-K\mu t_0} \ \  \forall t \leq t_0 \leq 0,$$
$$\|T_{cs}(t,t_0)\|_{\cL \left(\left(H^k_\Euler\right)^2 \right)}\leq C e^{\lambda_{cs}(t-t_0)-K\mu t_0}\ \ \forall t_0\leq t \leq 0. $$
\end{enumerate}
\end{proposition}

\begin{proof} 
Statement (1) is a direct corollary of $\bP_u(t)^2= \bP_u(t) \in \cL( (H_\Euler^k)^2)$ proved in Lemma \ref{L:bP_u}. 

On the one hand, from the definition \eqref{E:bP_u} we have $Y_u(t) \subset U(t) X_u$. On the other hand, $\bP_u(t) U(t) = U(t)$ proved in Lemma \ref{L:bP_u} implies $U(t) X_u \subset Y_u(t)$. Hence $Y_u(t) = U(t) X_u$. From Lemma \ref{L:U} and the invertibility of $e^{tL_u}$, $U(t)$ is injective and bounded. Open mapping theorem implies that $U(t): X_u \to Y_u(t)$ is isomorphic and thus statement (2) follows.

To prove statement (3), we start with the formula \eqref{E:T} of $T(t, t_0)$ and use the commutativity $\bP_u e^{\tau L} = e^{\tau L} \bP_u$, the conjugacy \eqref{E:conjugacy}, and Lemma \ref{L:bP_u} to compute 
\be \label{E:T2bP_u} \begin{split}
T_2 (t, t_0) \bP_u(t_0) z= & (D u_0(t))^{-1} \big(\pi_v e^{(t-t_0)L} \Lambda_0(t_0) \bP_u (t_0) z \big) \circ u_0(t) \\
= & (D u_0(t))^{-1} \big(\pi_v e^{(t-t_0)L} \bP_u \Lambda_0(t_0) z \big) \circ u_0(t) \\
= & (D u_0(t))^{-1} \big(\pi_v \bP_u   \Lambda_0(t) T(t, t_0) z \big) \circ u_0(t) = \pi_2 \bP_u(t) T(t, t_0)z, 
\end{split} \ee
where $\pi_j$, $j=1,2$, denotes the projection to the first and second components of $z$, respectively.  
Again using \eqref{E:T}, \eqref{E:conjugacy}, the definition \eqref{E:bP_u} of $\bP_u$, Lemma \ref{L:bP_u}, and the above equalities of $T_2 (t, t_0) \bP_u(t_0) z$, we have 
\be \label{E:T1bP_u} \begin{split}
T_1 (t, t_0) \bP_u(t_0) z= &  \int_{-\infty}^{t_0} (Du_0(s))^{-1}\big(\pi_v e^{(s-t_0)L_u} \bP_u \Lambda_0 (t_0) z \big) \circ u_0(s) ds+ \int_{t_0}^{t} T_2 (s, t_0) \bP_u(t_0) z ds \\
= &  \int_{-\infty}^{t} (Du_0(s))^{-1}\big(\pi_v e^{(s-t)L_u} e^{(t-t_0)L_u} \bP_u \Lambda_0 (t_0) z \big) \circ u_0(s) ds \\
= &  \int_{-\infty}^{t} (Du_0(s))^{-1}\big(\pi_v e^{(s-t)L_u}\bP_u\Lambda_0 (t) T(t, t_0)   z \big) \circ u_0(s) ds = \pi_1 \bP_u(t) T(t, t_0) z. 
\end{split} \ee
Summarizing the above equalities we obtain the dynamic commutativity between $\bP_u(t)$ and $T(t, t_0)$ in \eqref{E:TbP_u}. The invariance of $Y_{u, cs}(t)$ under $T(t, t_0)$ follows as a corollary immediately.

For $t\le t_0 \le 0$ and $z\in Y_u(t_0) = \bP_u(t_0) (H_\Euler^k)^2$, the exponential decay estimate of $T_2(t, t_0) z$ and $T_1(t, t_0) z$ follow from (ED), \eqref{u0-L-infinity}, \eqref{u0-hl}, \eqref{E:composition-1}, as well as the second equalities in \eqref{E:T2bP_u} and \eqref{E:T1bP_u}, respectively. 

Similarly, for $t_0 \le t\le 0$ and $z\in Y_{cs} (t_0) =\ker  \bP_u(t_0)$, using $\bP_{cs} = I -\bP_u$, \eqref{E:T}, and the second equality in \eqref{E:T2bP_u} we obtain  
\begin{align*}
T_2 (t, t_0) z =& T_2(t, t_0) (I-\bP_u(t_0)) z 
= (D u_0(t))^{-1} \big(\pi_v e^{(t-t_0)L_{cs}}\bP_{cs} \Lambda_0(t_0) z \big) \circ u_0(t), 
\end{align*}
which along with (ED), \eqref{u0-L-infinity}, \eqref{u0-hl},  and \eqref{E:composition-1} yields the desired exponential decay estimate of $T_2 (t, t_0) z$. The estimate on $T_1 (t, t_0) z$ follows directly from that of $T_2 (t, t_0) z$ and \eqref{E:T}. 
\end{proof}

\subsection{Unstable Integral Manifolds in Lagrangian formulation} \label{SS:InMa-L}

Motivated by Lemma \ref{L:UnstableS-L}, we follow the Lyapunov-Perron integral equation method to construct the integral unstable manifolds of (\ref{stratified-euler-taylor}) as in \cite{lin-zeng-manifold}. 
One observes that the projection maps $\bP_{u,cs}(t)$ may grow exponentially as $t \to -\infty$, see Proposition \ref{time-dependent-dichotomy}. 
Intuitively this is due to the fact that the angle between unstable subspace $Y_u(t)$ and center-stable subspace $Y_{cs}(t)$ may shrink exponentially  as $t \to -\infty$. 
To overcome this extra exponential growth, we use the quadratic nature of $\sF_{u,cs} = \mathcal O (\|z\|^2)$ (defined in \eqref{stratified-euler-taylor}) combined with the large spectral gap \eqref{E:lambda-1} assumed in  Theorem \ref{main-theorem}. For $\lambda \in (\lambda_{cs},\lambda_u) $, 
define 
$$S_\lambda \triangleq \left\{z=(z_u,z_{cs}) \in C^0 \left((-\infty,0],
(H^k_\Euler)^2\right)
:\|z\|_{\lambda} \leq \infty\right\},$$
where
\[
z_{u,cs}(t) = \bP_{u, cs} (t) z(t) \in Y_{u,cs}(t), \quad \|z\|_\lambda = \sup_{t \leq 0} e^{-\lambda t} \|z(t)\|_{H^k}. 
\]
Clearly, $ S_\lambda$ is  the Banach space of continuous functions with the exponentially decaying rate  $\lambda$ backward in time, equipped with the exponentially weighted norm.
Fixing $\delta_0, \delta_1>0$ to be determined later, we define the Lyapunov-Perron integral operator $\sL$ on $\overline {B_{\delta_1} ( S_\lambda)} \times Y_u(0)$ as 
\be \label{E:sL}
\sL(z,z_{u0})\triangleq\tilde{z} = \left(\tilde{z}_{u},\tilde{z}_{cs}\right),
\ee
where, for each 
$t \leq 0$,
\begin{equation}
\label{lyapunov-perron}
\begin{cases}
    \tilde{z}_{u}(t) = T_u(t,0)z_{u0} + \int_0^t T_u(t,s) {\sF}_{u}(s,z(s)) ds, \\
    \tilde{z}_{cs}(t) = \int_{-\infty}^t T_{cs} (t,s) \sF_{cs}(s,z(s))ds. 
\end{cases}
\end{equation}
Here $\sF_{u,cs}(t,z) = \bP_{u,cs}(t) \sF(t, z)$ where $\sF$ was defined in (\ref{stratified-euler-taylor}). 
Note that for any  $z(t) \in  B_{\delta} (S_\lambda)$ where $\delta$ is given in Lemma \ref{L:Euler-L-1}, $z(t)$ is a solution to (\ref{stratified-euler-taylor}) with its initial data satisfying $z_u(0) = z_{u0}$ iff $z$ is a fixed point of $\sL(\cdot, z_{u0})$ by the standard invariant manifold theory (see, e.g. \cite{CHOW1988285,CHOW1991266}).

\begin{proposition}    \label{uniform-contraction} 
There exists $K \in \mathbb N$ determined by $k+r$ such that for any $\lambda$ satisfying
    \begin{equation}
        \label{lambda}
        \lambda \in ( \lambda_{cs},\lambda_u - K\mu\big), \quad \lambda > 2K \mu, 
    \end{equation}
    there exists $C>0$ determined by $k+r, \rho_\pm, \Omega$, $C_*, \mu$,  $\mu_0$, $\lambda_u  -K \mu -\lambda$, $\lambda - \lambda_{cs}$, $|\rho_0|_{Z^{k+r}}$, $|D v_0|_{Z^{k+r-1}}$, and $|\nabla p_0|_{Z^{k+r-1}}$ such that for any $\delta_1 \in (0, \delta]$ where $\delta$ is given in Lemma \ref{L:Euler-L-1},
\[
\| \sD_z \sL\|_{C^{0} (\overline{B_{\delta_1} (S_\lambda)}
, \cL(S_\lambda))} \le C \delta_1,
\] 
\end{proposition}

\begin{proof}
By the definition of $\sF_{u, cs}$, \eqref{u0-L-infinity}, \eqref{u0-hl}, \eqref{E:composition-1} and Lemmas \ref{L:Euler-L-1} and \ref{L:bP_u}, there exists $K \in \mathbb N$ determined by $k+r$ such that if $\lambda_u > K\mu$, then for some $C>0$
\be \label{E:sF} \begin{split} 
&\sF_{u,cs}(t,0) =0, \quad \sD_{z} \sF_{u,cs}(t,0) = 0, \\ 
&\left\|\sD^2_{z} \sF_{u,cs}(t,\cdot)\right\|_{C^{r-4}\left((B_{\delta}(H^k_\Euler))^2,\, \cL ((H^k_\Euler)^2 \otimes (H^k_\Euler)^2, (H_\Euler^k)^2)\right)} \leq Ce^{K \mu |t|}.
\end{split} \ee
These properties also imply 
\be \label{E:sF-1}
\left\|\sD_{z} \sF_{u,cs}(t, z)\right\|_{\cL ((H^k_\Euler)^2 \otimes (H^k_\Euler)^2, (H_\Euler^k)^2)} \leq Ce^{K \mu |t|} \|z\|_{H^k}, \quad \forall z \in B_{\delta}(H^k_\Euler))^2.
\ee

To show the smoothness of $\sL$ and estimate its derivatives, we focus on the dependence of $\sL$ on $z\in \overline{B_{\delta_1} (S_\lambda)}$. 
    Formally, for any $1\le m \le r-2$ and $\bar z\in S_\lambda$,  its derivative takes the form  
    \[
    \sD_z^m \sL(z,z_{u0})(\bar z, \ldots, \bar z)\triangleq\hat{z} =\left(\hat{z}_{u},\hat{z}_{cs}\right), 
    \]
    where      
    \begin{equation}
    \label{DzL}
    \begin{cases}
        \hat{z}_u(t) = \int_0^t T_u(t,s) \sD_z^m \sF_u(s,z(s))\big( \bar z(s), \ldots, \bar z(s) \big) ds\\
        \hat{z}_{cs}(t) =\int_{-\infty}^t T_{cs}(t,s) \sD_z^m \sF_{cs}(s,z(s))\big( \bar z(s), \ldots, \bar z(s) \big)ds.
    \end{cases}
    \end{equation}
Using \eqref{E:sF-1} and  Proposition \ref{time-dependent-dichotomy}, it is straightforward to  estimate $\sD_z^m$ for $m=1$ and $t\le 0$,     
\begin{align*}
e^{-\lambda t} \| \hat z_u(t) \|_{H^k} \le & C \|z\|_\lambda \|\bar z\|_{\lambda} \int_t^0 e^{ -\lambda t + (\lambda_u - K \mu)(t-s)-K\mu s - K \mu s + 2 \lambda s} ds \\
\le & C \|z\|_\lambda \|\bar z\|_{\lambda} \int_t^0 e^{ (\lambda_u - K \mu -\lambda)(t-s) + (\lambda - 2K\mu) s} ds \le \frac {C \|z\|_\lambda \|\bar z\|_{\lambda}}{\lambda_u - K \mu -\lambda},
\end{align*}
\begin{align*}
e^{-\lambda t} \| \hat z_{cs}(t) \|_{H^k} \le & C \|z\|_\lambda \|\bar z\|_{\lambda} \int_{-\infty}^t e^{ -\lambda t + \lambda_{cs} (t-s)-K\mu s - K \mu s + 2 \lambda s} ds \le \frac {C \|z\|_\lambda \|\bar z\|_{\lambda}}{\lambda-\lambda_{cs}}. 
\end{align*}
where $\lambda \ge 2K\mu$ was also used. Therefore 
\be \label{E:DsF-1} 
\|\sD_z \sL (z, z_{u0}) \bar z \|_\lambda \le C \|z\|_\lambda \|\bar z\|_{\lambda}. 
\ee    
For $2\le m \le r-2$, using \eqref{E:sF} instead, much as the above we have 
\be \label{E:DsF-m} 
\|\sD_z^m \sL (z, z_{u0}) (\bar z, \ldots, \bar z) \|_\lambda \le C \|\bar z\|_{\lambda}^m. 
\ee    

To prove that these operators $\sD^m \sL$ are truly the Fr\'echet derivatives of $\sL$, we only need to prove $\sD_z^{r-2} \sL$ is $C^0$ in $z$. In fact, for any $z, z+\tilde z \in \overline{B_{\delta_1} (S_\lambda)}$ and $\bar z \in S_\lambda$ with $\|\bar z\|_\lambda \le 1$, let 
\[
\left(\hat{z}_{u},\hat{z}_{cs}\right) = \sD_z^{r-2} \sL(z+\tilde z,z_{u0})(\bar z, \ldots, \bar z)-  \sD_z^{r-2} \sL(z,z_{u0})(\bar z, \ldots, \bar z). 
\]
Since $r-2 \ge 2$, 
again by \eqref{E:sF} and Proposition \ref{time-dependent-dichotomy} we have, for $t\le 0$, 
\begin{align*}
&e^{-\lambda t}\left\| \hat z_u (t) \right\|_{H^k} \\
\le & C \|\bar z\|_{\lambda}^{r-2} \int_t^0 e^{(\lambda_u - K\mu -\lambda) (t-s) + ((r-3) \lambda - K\mu)s
} \big\| \sD_z^{r-2} \sF_u \big(s, (z+\tilde z)(s)\big) - \sD_z^{r-2} \sF_u \big(s, z(s)\big) \big\| ds,
    \end{align*}
where the norm of $\sD_z^{r-2} \sF$ is taken in the space of bounded multilinear transformations.     
For any $\alpha>0$, on the one hand, from \eqref{E:sF} there exists $T>0$ independent of $z$, $\tilde z$, and $\bar z$ such that 
\[
C  \int_t^{\max\{t, -T\}} e^{ ((r-3) \lambda - K\mu)s } \big\| \sD_z^{r-2} \sF_u \big(s, (z+\tilde z)(s)\big) - \sD_z^{r-2} \sF_u \big(s, z(s)\big) \big\| ds 
\le C  \int_{-\infty}^{-T} e^{ ( \lambda - 2K\mu)s} d\tau \le \alpha 
\]    
On the other hand, 
\[
\lim_{\|\tilde z\|_\lambda \to 0} \int_{-T}^0 \big\| \sD_z^{r-2} \sF_u \big(s, (z+\tilde z)(s)\big) - \sD_z^{r-2} \sF_u \big(s, z(s)\big) \big\| ds =0. 
\]    
Therefore we obtain $\|\hat z_u \|_\lambda \to 0$ as $\|\tilde z\|_{\lambda} \to 0$ uniformly in $\bar z$ as long as $\|\bar z\|_{\lambda} =1$.  The convergence of $\hat z_{cs}$ can be obtained similarly. Hence $\sD_z^{r-2} \sL(z, z_{u0})$ is $C^0$ in $z \in  \overline{B_{\delta_1} (S_\lambda)}$ and $\sL$ is a $C^{r-2}$ mapping. This completes the proof of the lemma. 
\end{proof}

Clearly $\sL$ is linear in $z_{u0}$ and from Proposition \ref{time-dependent-dichotomy} we have 
\be \label{E:temp-14}
\|\sL(0, z_{u0}) \|_\lambda = \| T_u(\cdot, 0) z_{u0}\|_\lambda = \sup_{t\le 0} C e^{(\lambda_u - K\mu -\lambda) t} \|z_{u0}\|_{H^k} \le C \|z_{u0}\|_{H^k}.
\ee
For $\delta_0, \delta_1>0$ such that 
\be \label{E:delta}
2C\delta_1 < 1, \quad 2C \delta_0\le \delta_1, 
\ee
the above inequality and Proposition \ref{uniform-contraction} imply that, for any $z_{u0} \in B_{\delta_0} (Y_u(0))$,  $\sL(\cdot, z_{u0})$ is a contraction on $\overline{B_{\delta_1} (S_\lambda)}$ with a Lipschitz constant $\frac 12$. By the uniform contraction mapping theorem (see, e.g. Theorem 2.2 in \cite{Chow-Hale}), 
there exists 
\be \label{fixed-point} \begin{split} 
z^\star  \in C^{r-2} \big( B_{\delta_0} & (Y_u(0)),  \overline {B_{\delta_1} (S_\lambda)}\big) \; \text{ such that } \; \\
&\sL (z, z_{u0}) = z, \ (z, z_{u0}) \in  \overline {B_{\delta_1} (S_\lambda)} \times B_{\delta_0}(Y_u(0)) \Longleftrightarrow z = z^\star (z_{u0}).
\end{split} \ee
From the definition and Lipschitz estimate of $\sL$ as well as \eqref{E:temp-14}, we have 
\begin{equation}
\label{fixed-point-initial}
z_u^\star (0, z_{u0}) = z_{u0}, \quad    \left\|z^\star(z_{u0})\right\|_\lambda  \leq C \|z_{u0}\|_{H^k}.
\end{equation}

The following lemma shows that $z^\star$ is independent of $\lambda$.  

\begin{lemma} \label{L:UM-L} 
Let both $\lambda$ and $\tilde \lambda$ satisfy \eqref{lambda}, which determine the fixed point functions 
\[
z^\star \in C^{r-2} \big( B_{\delta_0}  (Y_u(0)), \overline {B_{\delta_1} (S_\lambda)}\big), \quad \tilde z^\star \in C^{r-2} \big( B_{\tilde \delta_0} (Y_u(0)),  \overline {B_{\tilde \delta_1} (S_{\tilde \lambda})}\big), 
\]
respectively. Then there exists $\bar \delta_0>0$ such that $z^\star= \tilde z^\star$ on $B_{\bar \delta_0}  (Y_u(0))$. 
\end{lemma} 

\begin{proof}
Without loss of generality, assume $\lambda < \tilde \lambda$, then it is clear $S_{\tilde \lambda} \subset S_\lambda$. For any $\bar \delta_0 \le \min\{\delta_0, \tilde \delta_0\}$ and $z_{u0} \in B_{\bar \delta_0} (Y_u(0))$, the corresponding fixed point  
\[
\tilde z^\star (z_{u0}) \in \overline {B_{\tilde \delta_1} (S_{\tilde \lambda})} \subset  \overline {B_{\tilde \delta_1} (S_\lambda)}, \quad \|\tilde z^\star (z_{u0}) \|_{\lambda} \le \|\tilde z^\star (z_{u0}) \|_{\tilde \lambda} \le C \|z_{u0}\|_{H^k} \le C \bar \delta_0. 
\]
Take $\bar \delta_0$ such that $C \bar \delta_0 \le \delta_1$. Since the definition of $\sL$ is independent of $\lambda$, $\tilde z^\star (z_{u0}) \in \overline {B_{\delta_1} (S_\lambda)}$ is also the fixed point of $\sL(\cdot, z_{u0})$ on $\overline {B_{\delta_1} (S_\lambda)}$. Due to the uniqueness of the fixed point of $\sL(\cdot, z_{u0})$, we obtain $\tilde z^\star (z_{u0}) =  z^\star (z_{u0})$. 
\end{proof}

Define 
\be \label{E:UM-L}
W_L^u (\delta_0)
= \{z^\star (0, z_{u0})
:z_{u0} \in B_{\delta_0}(Y_u(0)) \},
\ee
which is the slice of the unstable integral manifold of (\ref{stratified-euler-taylor}) (equivalent to (\ref{E:Euler-L-3})) at $t=0$ satisfying the following properties. 

\begin{theorem} \label{T:UM-L}
There exists $K \in \mathbb N$ determined by $k+r$ such that, for any equilibrium $(v_0, \rho_0, p_0)$ of the  Euler equation \eqref{stratified-euler} satisfying \eqref{E:steady}, \eqref{u0-L-infinity}, \eqref{E:Lyap-E}, and the exponential dichotomy (ED), and any $\lambda$ satisfying \eqref{lambda}, 
    there exists $C_0 \ge 1$ determined by $k+r, \rho_\pm, \Omega$, $C_*, \mu$,  $\mu_0$, $\lambda_u  -K \mu -\lambda$, $\lambda - \lambda_{cs}$, $|\rho_0|_{Z^{k+r}}$, $|D v_0|_{Z^{k+r-1}}$, and $|\nabla p_0|_{Z^{k+r-1}}$ such that $W_L^u (\frac 1{C_0})$ defined in \eqref{E:UM-L}  satisfies the following. 
\begin{enumerate}     
\item $W_L^u (\frac 1{C_0})$ is a $C^{r-2}$ submanifold in $(H_\Euler^k)^2$ and the tangent space $T_0 W_L^u (\frac 1{C_0}) = Y_u(0)$. 
\item  For any $z_0 \in W_L^u(\frac 1{C_0})$, the solution $z(t)$ to \eqref{stratified-euler-taylor} with the initial value $z(0)=z_0$ satisfies  
$\|z(t)\|_{H^k} \le C_0 e^{\lambda t} \|z_0\|_{H^k}$ for all $t \le 0$.
\item Suppose $\tilde \lambda$ satisfies \eqref{lambda}, then there exists $\delta_2 >0$ such that if  $z(t)$, $t \le 0$, is a solution to \eqref{stratified-euler-taylor}  satisfying $\|z(\cdot)\|_{\tilde \lambda} < \delta_2$, then $z(0) \in W_L^u (\frac 1{C_0})$. 
\end{enumerate}
\end{theorem}

\begin{proof}
Let $C$ be determined by Proposition \ref{uniform-contraction}. For any $C_0> 4 C^2$, 
according to \eqref{E:delta}, $W_L^u(\frac 1{C_0})$ is well-defined. By the uniqueness of the fixed point of $\sL(\cdot, z_{u0})$, 
$z^\star(\cdot,0) = 0$ and hence $0\in W_L^u (\frac 1{C_0})$. 
Differentiating the fixed point equation we have 
$$
\sD_{z_{u0}} z^\star(\cdot,z_{u0}) = \sD_{z_{u0}}\sL(z^\star(\cdot,z_{u0}),z_{u0}) + \sD_z \sL (z^\star(\cdot,z_{u0}),z_{u0})\sD_{z_{u0}} z^\star(\cdot,z_{u0}).
$$
Due to (\ref{DzL}) and \eqref{E:sF}, it holds 
$\sD_z \sL(0, 0)  = 0$ which implies 
\be \label{E:sDzsL}
\sD_{z_{u0}} z^\star(\cdot,0) \tilde z_u= \sD_{z_{u0}} \sL (0,0) \tilde z_u= T_u(\cdot,0)\tilde z_u \implies   \sD_{z_{u0}} z^\star(0,0) \tilde z_u = \tilde z_u \in Y_u(0).
\ee
Therefore statement (1) follows from the Implicit Function Theorem. 

From the definition of $W_L^u (\frac 1{C_0})$, for any $z_0 \in W_L^u(\frac 1{C_0})$, its solution $z(t)$ to \eqref{stratified-euler-taylor} is given by $z(t) = z^\star (t, \bP_u (0) z_0)$. Hence statement (2) follows directly from \eqref{fixed-point-initial}. 


Finally, statement (3) is an immediate corollary of Lemma \ref{L:UM-L}.
\end{proof}

\subsection{Invariant unstable manifolds in Eulerian formulation} \label{SS:InMa-E}

Finally we come back to construct the local unstable manifold of the steady state $(v_0,\rho_0, p_0)$ in the original Eulerian coordinates of the stratified Euler equations \eqref{stratified-euler}. Throughout this subsection, all solutions are assumed to satisfy $\nabla (p - p_0) \in H^k(\Omega, \R^d)$ and thus their local well-posedness is ensured by Theorem \ref{T:LWP}.  
Letting $t=0$ in the correspondence $\Lambda$ defined in \eqref{E:transform} between \eqref{stratified-euler} and its Lagrangian formulation (\ref{stratified-euler-taylor}) or equivalently \eqref{E:Euler-L-3} with $\sigma =\rho_0$, we define 
\be \label{E:UM-H}
H(\bar \rho, \bar v) = \Lambda\big(0, \rho_0, z^\star \big(0, U(0) (\bar \rho, \bar v)^T\big) \big) = \begin{pmatrix}
\rho_0 \circ (\Psi_{\rho_0}(z_1))^{-1} \\
    v_0 + \left(\sD \Psi_{\rho_0}(z_1)z_2\right)\circ(\Psi_{\rho_0}(z_1))^{-1}
\end{pmatrix}, 
\ee
where 
\[
(z_1, z_2)^T = z^\star \big(0, U(0) (\bar \rho, \bar v)^T\big),
\]
and for some sufficiently small $\delta_2>0$
\be \label{E:UM-E}
W^u (\delta_2) \triangleq \{ H(\bar \rho, \bar v) : (\bar \rho, \bar v)^T \in B_{\delta_2} (X_u) \}  
\subset \Lambda(0, \rho_0, W_L^u (\tfrac 1{C_0})) 
 \subset (\rho_0, v_0) + H^k (\Omega, \R)  \times H_\Euler^k.  
\ee
We shall see that $W^u$ is the local unstable manifold of $(\rho_0, v_0)$. However, as seen in Remark \ref{R:non-smooth}, $\Lambda(t, \rho_0, \cdot)$ is a $C^0$, but not smooth, mapping from $(B_\delta (H_\Euler^k))^2$ to $(\rho_0, v_0) + H^k (\Omega, \R)  \times H_\Euler^k$, so $H$ is not smooth and $W^u$ may not be smooth submanifold in $(\rho_0, v_0) + H^k (\Omega, \R)  \times H_\Euler^k$. 

\begin{proposition}  \label{unstable-manifolds-in-eulerian-coordinates}
In addition to \eqref{E:steady}, \eqref{u0-L-infinity}, \eqref{E:Lyap-E}, assume that, for some  $1 \le r_1 \le k$, 
$X_u \subset H^{k-r_1} (\Omega, \R)  \times H_\Euler^{k-r_1}$ is also a closed subspace. There exists $K \in \mathbb N$ determined by $k+r$ such that, for any $\lambda$ satisfying \eqref{lambda}, there exist $C, \delta_2>0$ determined by $k+r, \rho_\pm, \Omega$, $C_*, \mu$,  $\mu_0$, $\lambda_u  -K \mu -\lambda$, $\lambda - \lambda_{cs}$, $|\rho_0|_{Z^{k+r}}$, $|D v_0|_{Z^{k+r-1}}$, and $|\nabla p_0|_{Z^{k+r-1}}$ such that the following hold. 
\begin{enumerate} 
\item for any $1\le r_0 \le \min\{r-2, r_1\}$, $W^u(\delta_2)$ is a $C^{r_0}$ submanifold in $(\rho_0, v_0) + H^{k-r_0} (\Omega, \R)  \times H_\Euler^{k-r_0}$ and 
$T_{(\rho_0, v_0)} W^u(\delta_2) = X_u$.  
\item $H: B_{\delta_2} (X_u) \to W^u(\delta_2)$ is a homeomorphism where $W^u(\delta_2) \subset (\rho_0, v_0) + H^{k} (\Omega, \R)  \times H_\Euler^{k}$ is equipped with the subset topology. 
\item For any $\left(\rho^\ast, v^\ast\right) \in W^u(\delta_2) \cap \big((\rho_0, v_0) + B_{\frac {\delta_2}C}\left(H^k(\Omega,\mathbb R)\times H^k_\Euler \right)\big)$,  the solution $(\rho(t), v(t))$ of (\ref{stratified-euler}), with the initial value $(\rho(0), v(0)) = \left(\rho^\ast, v^*\right)$, satisfies
\[
(\rho(t), v(t))^T \in W^u(\delta_2), \;\; \left\|(\rho(t), v(t)) - (\rho_0, v_0)\right\|_{H^k} \leq Ce^{(\lambda-K \mu)t}\|(\rho^*, v^*) - (\rho_0, v_0)\|_{H^{k-r_1}},\; \forall \, t \leq 0.
\]    
    \end{enumerate}
\end{proposition}

\begin{proof}
Clearly the embedding $(X_u, \|\cdot\|_{H^k}) \to (X_u, \|\cdot\|_{H^{k-r_1}})$ is a bounded linear bijective operator. The latter is also a Hilbert space when $X_u$ is a closed subspace of $H^{k-r_1}$. So according to the Open Mapping Theorem, $(X_u, \|\cdot\|_{H^k})$ and $(X_u, \|\cdot\|_{H^{k-r_1}})$ are isomorphic Hilbert spaces, which is also true if $k-r_1$ is replaced by any $m\in [k-r_1, k]$. Therefore, viewing $U(0)$ as the composition of $U(0) \in \cL \big((X_u, \|\cdot\|_{H^k}), (H_\Euler^k)^2\big)$ with the inverse of the embedding above, we have $U(0) \in \cL\big((X_u, \|\cdot\|_{H^{k-r_1}}), (H_\Euler^k)^2 \big)$. Consequently,  \eqref{E:Lambda-Cr} 
and \eqref{fixed-point} imply 
\[
H \in C^{r_0} \big(B_{\delta_2} (X_u), (\rho_0, v_0) + H^{k-r_0} (\Omega, \R)  \times H_\Euler^{k-r_0} ) \big), \quad H(0)=(\rho_0, v_0)^T. 
\]
Moreover, from \eqref{E:Lambda0}, \eqref{E:sDzsL}, and Lemma \ref{L:U}, we have $\sD H (0) = I \in \cL((X_u, \|\cdot\|_{H^{k-r_1}}), H^{k-r_1} (\Omega, \R)  \times H_\Euler^{k-r_1}) $, or more precisely the embedding of $X_u$ into $H^{k-r_0} (\Omega, \R)  \times H_\Euler^{k-r_0}$. By the Implicit Function Theorem, for small $\delta_2>0$, the image  $W^u(\delta_2)$ of $H$ is a $C^{r_0}$ submanifold in $H^{k-r_0} (\Omega, \R)  \times H_\Euler^{k-r_0}$ tangent to $X_u$ at $(\rho_0, v_0)$. 

From the above analysis and Remark \ref{R:non-smooth}, $H : B_{\delta_2} (X_u) \to W^u(\delta_2) \subset (\rho_0, v_0) + H^{k} (\Omega, \R)  \times H_\Euler^{k}$ is bijective and continuous. For statement (2) we only need to show that $H^{-1}: W^u (\delta_2) \to X_u$ is also continuous when $W^u(\delta_2)$ is equipped with the subset topology from $H^k$. Let $(\rho_n, v_n) = H(z_{u0}^{(n)}) \to (\rho, v) = H (z_{u0})$ in $H^k$ as $n \to \infty$. It is obvious that the convergence also holds in $H^{k-r_1}$. From statement (1), we have $z_{u0}^{(n)} \to z_{u0}$ in $H^{k-r_1}$ which yields the convergence $z_{u0}^{(n)} \to z_{u0}$ in $H^{k}$ as the $H^k$ and $H^{k-r_1}$ topology are equivalent on $X_u$. Hence the continuity of $H^{-1}: W^u (\delta_2) \to X_u$ follows. 

To prove statement (3), let $(\bar \rho, \bar v)^T = H^{-1} (\rho^*, v^*)\in  B_{\delta_2}(X_u)$ and 
\[
z(t) = (z_1(t), z_2(t))^T  = z^\star \big(t, U(0)( \bar \rho, \bar v )^T \big), \quad (\rho(t), v(t))^T  = \Lambda(t, \rho_0, z(t)),  
\]
for $t\le 0$. Clearly $ (\rho(t), v(t))$ is the solution to \eqref{stratified-euler} with initial value $ (\rho^*, v^*) \in W^u(\delta_2)$ as $z(\cdot ) \in B_{\delta_1} (S_\lambda)$ is the solution to (\ref{stratified-euler-taylor}) with initial value $z(0) = z^\star (0, U(0)( \bar \rho, \bar v)^T) \in W_L^u$.
By the exponential weight in the definition of $S_\lambda$, Lemma \ref{L:U}, the equivalence of the $\|\cdot\|_{H^k}$ and $\|\cdot\|_{H^{k-r_1}}$ on $X_u$, the local smooth estimates on $H^{-1}: W^u(\delta_2) \to X_u$ due to the above analysis based on the Implicit Function Theorem, and (\ref{fixed-point-initial}), we have for $t\le 0$, 
\begin{equation}\label{z-decaying-rate} \begin{split} 
& \|z(t)\|_{H^k} 
\leq e^{\lambda t}\|z\|_\lambda \leq C e^{\lambda t}\| U(0)( \bar \rho, \bar v)^T\|_{H^k} \leq C e^{\lambda t}\| ( \bar \rho, \bar v) \|_{H^k} \\
\leq & C e^{\lambda t}\| ( \bar \rho, \bar v) \|_{H^{k-r_1}} 
\le  C e^{\lambda t} \|(\rho^*, v^*)-(\rho_0, v_0)\|_{H^{k-r_1}}.
\end{split} \end{equation}
While $\Lambda(t, \rho_0, \cdot): H^k \to H^k$ is only $C^0$ (see Remark \ref{R:non-smooth}), we rewrite
\begin{align*}
\rho(t) - \rho_0 
= & \rho_0 \circ \Psi_{\rho_0} (z_1)^{-1} \circ u_0(t)^{-1} - \rho_0 \circ u_0(t)^{-1} \\
=& \Big(  \int_0^1 \big( (\nabla \rho_0) \circ \Psi_{\rho_0} (\tau z_1) \big) \cdot \sD \Psi_{\rho_0} (\tau z_1) z_1 d\tau \Big)
\circ \Psi_{\rho_0} (z_1)^{-1} \circ u_0(t)^{-1},
\end{align*}
then the $H^k$ norm of $( \rho(t), v(t))^T$ can be estimated using the formula \eqref{E:transform} of $\Lambda$ and \eqref{z-decaying-rate} as 
\be \label{E:temp-14.5}
\|(\rho(t), v(t))-(\rho_0, v_0)\|_{H^k} 
\le C e^{-K\mu t} \|z(t)\|_{H^k} \le C^{(\lambda - K\mu)t} \|(\rho^*, v^*)-(\rho_0, v_0)\|_{H^{k-r_1}}, 
\quad t\le 0. 
\ee
This proves the desired exponential decay estimates. 

To show the local invariance of $W^u(\delta_2)$, for $t_0, t \leq 0$, let   
$\tilde{z}(t) = (\tilde{z}_1(t),\partial_t\tilde{z}_1(t))$ where
\be \label{E:temp-15} \begin{split}
\tilde{z}_1(t) = &\Psi_{\rho_0}^{-1}\big(u_0(t_0) \circ \Psi_{\rho_0}(z_1(t+t_0)) \circ (u_0(t_0))^{-1} \big)\\
= & \cP(\rho_0) \big(u_0(t_0) \circ \Psi_{\rho_0}(z_1(t+t_0)) \circ (u_0(t_0))^{-1} -\id_\Omega \big).
\end{split} \ee
Proposition \ref{local-coordinate} was used to derive the above second equality. 
From Proposition \ref{local-coordinate}, \eqref{E:composition-1}, (ED), and \eqref{u0-hl}, we have 
\begin{align*}
\|\tilde z_1(t) & \|_{H^k} \le  C \big\| u_0(t_0) \circ \Psi_{\rho_0}(z_1(t+t_0)) \circ (u_0(t_0))^{-1} - \id_\Omega \big\|_{H^k} \\
\le & C \int_0^1\Big\| \Big( \big( D u_0(t_0) \circ \Psi_{\rho_0} (\tau z_1(t+t_0) ) \big) \cdot \big( \sD \Psi_{\rho_0} (\tau z_1(t+t_0) ) z_1(t+t_0) \big) \Big)    \circ (u_0(t_0))^{-1} \Big\|_{H^k} d\tau  \\
\le & C e^{-K\mu t_0} \| z_1(t+t_0)\|_{H^k}. 
\end{align*}
Using the second equality of \eqref{E:temp-15},  $\p_t \tilde z_1(t)$ can be estimated similarly. Along with  \eqref{z-decaying-rate}, we have 
\be \label{E:temp-16}
\|\tilde z(t) \|_{H^k} \le C e^{-K\mu t_0} \| z (t+t_0)\|_{H^k} \le C e^{\lambda t + (\lambda -K\mu) t_0} \|(\rho^*, v^*)-(\rho_0, v_0)\|_{H^{k-r_1}}. 
\ee
Let $u(t) = u_0(t) \circ \Psi_{\rho_0} (z_1(t))$, which is a Lagrangian map corresponding to $v(t)$, and 
$$
\tilde{u}(t) = u_0(t) \circ \Psi_{\rho_0}(\tilde{z}_1(t)) = u_0(t+t_0) \circ \Psi_{\rho_0}(z_1(t+t_0)) \circ (u_0(t_0))^{-1} = u(t+t_0) \circ (u_0(t_0))^{-1}.
$$
Direct computation yields  
$\tilde{u}_t(t)= v(t+t_0) \circ \tilde{u}(t) $, so $\tilde{u}(t)$ is a Lagrangian map associated with the solution $(\rho(t+t_0), v(t+t_0))$. Since $z(t)$ is a solution to (\ref{stratified-euler-taylor}), we have 
\[
\rho (t) \circ u(t)= \rho_0 \implies \rho(t+t_0) \circ \tilde u(t) = \rho(t+t_0) \circ u(t+t_0) \circ u_0(t_0)^{-1} = \rho_0 \circ u_0(t_0)^{-1} = \rho_0. 
\]
Therefore $\tilde z(t)$ is also a solution to \eqref{E:Euler-L-3} with $\sigma =\rho_0$, or equivalently (\ref{stratified-euler-taylor}), and $ (\rho(t+t_0), v(t+t_0)) = \Lambda (t, \rho_0, \tilde z(t))$. By taking $\delta_2$ such that $C \delta_2 \le \min \{\delta_0, \delta_1\}$, estimate \eqref{E:temp-16} implies $\tilde z(\cdot) = z^\star (\cdot, \bP_u(0) \tilde z(0)) \in B_{\delta_1} (S_\lambda)$ with $\tilde z(0) \in W_L^u (\tfrac 1{C_0})$. As $\|(\rho^*, v^*)-(\rho_0, v_0)\|_{H^{k}} \le \delta_2/C$, by taking $C$ reasonably large and using \eqref{E:temp-14.5} and that $H$ is locally diffeomorphic, we have $(\rho(t_0), v(t_0)) = \Lambda (0, \rho_0, \tilde z(0)) \in W^u(\delta_2)$. Since $t_0\le 0$ is arbitrary, the proof of the local invariance is complete. 
%
\end{proof}



Finally we are ready to prove the main result Theorem \ref{main-theorem}. 

\begin{proof}[Proof of Theorem \ref{main-theorem}] 
Let $K_0 =K+1$ and $\lambda_0 =  \lambda_u - K_0 \mu$, where $K$ is given in Lemma \ref{L:UnstableS-L}, Proposition \ref{unstable-manifolds-in-eulerian-coordinates},and Theorem \ref{T:UM-L}. We shall see that for some sufficiently small $\delta_2>0$ determined by $k+r, \rho_\pm, \Omega$, $C_*, \mu$,  $\mu_0$, $\lambda_0 - \lambda_{cs}$, $|\rho_0|_{Z^{k+r}}$, $|D v_0|_{Z^{k+r-1}}$, and $|\nabla p_0|_{Z^{k+r-1}}$, the set $W^u = \overline{W^u (\frac {\delta_2}2)}$, with $W^u(\frac {\delta_2}2)$ defined above, satisfies all the properties in Theorem  \ref{main-theorem}. 
In fact, Proposition \ref{unstable-manifolds-in-eulerian-coordinates}(1)(2) implies that $W^u$ is a topological submanifold in $(\rho_0, v_0) + H^k (\Omega, \R)  \times H_\Euler^k$ and a $C^{r_0}$ submanifold in $(\rho_0, v_0) + H^{k-r_0} (\Omega, \R)  \times H_\Euler^{k-r_0}$ with $T_{(\rho_0, v_0)} W^u = X_u$. 
Theorem  \ref{main-theorem}(2) is also a direct corollary of Proposition \ref{unstable-manifolds-in-eulerian-coordinates}(3). 
Finally, given any solution $(\rho(t), v(t), p(t))$, $t \le 0$, to \eqref{stratified-euler} satisfying \eqref{E:ExpDecay} for some $\lambda\in \mathcal{I}$, let $t_0\le 0$ and consider $\big(\rho(t+t_0), v(t+t_0), p(t+t_0)\big)$ which is also a solution satisfying \eqref{E:ExpDecay} for the same $\lambda$. For any $\ep>0$, by taking $-t_0\gg 1$, \eqref{E:UnstableS-1} holds and Lemma \ref{L:UnstableS-L} applies to yield that there exists a unique solution $z(t) = (z_1(t), \p_t z_1(t))$ to \eqref{stratified-euler-taylor} such that \eqref{E:U-solu-1} is satisfied. Subsequently,  Theorem  \ref{main-theorem}(3) follows from Theorem \ref{T:UM-L}(3) and \eqref{fixed-point-initial}. 
\end{proof}

\section{Unstable manifolds of steady flows in 2-dim domains} \label{S:2D}

In this section, we consider two types of equilibria of the 2-dim stratified Euler equation \eqref{stratified-euler}: steady flows with small density variation on 2-dim bounded domains and shear flows on $\BBT\times \R$ exhibiting the Rayleigh-Taylor instability with a background flow. As in Section 3 and 4 in \cite{lin-zeng-manifold}, we focus on the linear exponential dichotomy condition (ED), which is the major assumption required in 
Theorem \ref{main-theorem} for them to be applied to yield the nonlinear unstable and stable manifolds.

\subsection{Unstable manifolds of steady flows with small density variation on 2-dim bounded domains} 

In this subsection, we consider the 2-dim domains 
\be \label{E:Omega-2}
\Omega =  \mathbb{T}^{d_1} \times \Omega_2, \ \text{ where } \ \Omega_2 \subset\subset \R^{d_2}, \quad d_1 + d_2 = 2; \quad \text{ and } \Omega \text{ is simply connected if }  d_2=2. 
\ee
Here $\Omega$ is a bounded interval if $d_2=1$ or a smooth bounded domain in $\R^2$ if $d_2=2$. Our first result is that small density stratification and relatively strong unstable eigenvalues yield the exponential dichotomy (ED). 

\begin{proposition} \label{P:ED-1}
Assume \eqref{E:Omega-2} and 
\begin{itemize} 
\item an equilibrium $(\rho_0, v_0) \in H^{k+1} (\Omega, \R) \times H_\Euler^{k+2}$, $k> 2$, of the Euler equation \eqref{stratified-euler} satisfies  $\rho_- \triangleq \inf \rho_0>0$,
\item its Lagrangian map $u_0(t)$ (defined in \eqref{lagrangian-coordinate-map-0}) satisfies \eqref{u0-L-infinity}, \eqref{E:Lyap-E}, with parameters $\mu$,  $\mu_0$, and $C_*$,
\item and the linearization operator $L$ defined in \eqref{linearization} at $(\rho_0, v_0)$  has an eigenvalue $\lambda_0$, Re$\lambda_0>k \mu$, along with an eigenfunction $(\rho, v)\in H^1(\Omega, \R) \times H_\Euler^1$,
\end{itemize}
then there exists $\ep>0$ depending on $\Omega$, $k$, $\|v_0\|_{H^{k+2}}$, $C_*$, $\mu$,  $\mu_0$, and $\rho_-$ such that if 
\be \label{E:temp-16.5} 
\|\nabla \rho_0\|_{H^{k}}^{\frac 12} < \min \{1, \ep (\text{Re} \lambda_0 - k\mu) \},  
\ee
then the linear exponential dichotomy (ED) holds with some $X_u$ and $X_{cs}$ such that $\dim X_u< \infty$ and $\lambda_u > k \mu$. 
\end{proposition}

The proposition states that the existence of 
unstable eigenvalues satisfying \eqref{E:temp-16.5} yields the linear exponential dichotomy (ED) with finite dimensional $X_u \subset H^k(\Omega, \R) \times H_\Euler^k$, which, in turn, implies the closedness of $X_u \subset L^2$. Hence Theorem \ref{main-theorem} applies provided that the first assumption in \eqref{E:lambda-1} is verified. Sometimes eigenfunctions are first identified in rough spaces like $H^1(\Omega, \R) \times H_\Euler^1$ as in the above assumptions, \eqref{E:temp-16.5} ensures that they belong to $H^k(\Omega, \R) \times H_\Euler^k$ automatically. 

\vspace{0.05in}
\noindent $\bullet$ {\bf Vorticity formulation.} 
The proof of the proposition will be carried out in the vorticity formulation of the Euler equation. 
Given a velocity field  $v(x)$ and density $\rho(x)$ 
on $\Omega$, let 
\[
\omega = \p_{x_1} v^2 - \p_{x_2} v^1, \quad \vec{m} = (m_1, \ldots, m_{d_1})^T, \ \text{ where } \ m_j = \int_\Omega \rho v^j dx,
\; j=1, \ldots, d_1, 
\]
denote its vorticity and the momenta in each of the periodic directions. For solutions to the Euler equation \eqref{stratified-euler}, it is standard to verify the conservation of momenta and to compute the evolution equation of the vorticity
\begin{subequations} \label{E:Euler-V}
\be \label{E:Vor-E} 
\vec{m}_{t} = 0,  \quad \omega_t + v \cdot \nabla \omega + \nabla \big(\frac 1\rho\big) \times \nabla p=0, 
\ee
where  \eqref{E:g} was also used. 
The velocity $v$ can be recovered by $\omega$ and $\vec{m}$ as 
\be \label{E:Vor-Vel}
v = J \nabla   \Delta^{-1}  \omega + \sum_{j=1}^{d_1} \frac {m_j - \int_\Omega \rho (J \nabla   \Delta^{-1}  \omega)_j dx}{\int_\Omega \rho dx} \Ve_j \triangleq \cB(\rho, \omega, \vec{m}), \; \text{ where } \ J = \begin{pmatrix}  0 & -1  \\ 1 & 0 \end{pmatrix}, 
\ee 
where $\Delta^{-1} \omega$ a.) has zero boundary value if $\p \Omega \ne \emptyset$ or b.) zero mean value if $\Omega = \BBT^2$. Unlike the Euler equation with homogeneous density, here the pressure also appears in the vorticity evolution \eqref{E:Vor-E} given by 
\be \label{E:pressure-2}
\frac 1\rho \nabla p = - \big(I-\cP(\rho)\big) (v \cdot \nabla v + g\Ve_d),
\ee
\end{subequations}
where $\cP(\rho)$ was given in Lemma \ref{hodge}. 
Therefore \eqref{E:Vor-E}, along with \eqref{E:Vor-Vel}, \eqref{E:pressure-2}, and the density transport equation, gives an equivalent system in terms of the unknown $(\rho, \vec{m}, \omega)$ to the Euler equation \eqref{stratified-euler}. 

\begin{remark} \label{R:Vor-Vel}
If $d_2=2$ and $\Omega$ is  not simply connected, for each connected component $\cC_j$, $j=1, \ldots$ of the inner boundary of $\Omega$, let $c_j=\int_{\cC_j} v \cdot d\vec{l}$ be the circulation along $\cC_j$ and then $v$ can be recovered by $\omega$ and $c_j$. See, for example, \cite{LWZ19} for some related computations. In particular, $c_j$ is a constant of motion of the Euler equation \eqref{stratified-euler} if $\rho$ is a constant along $\cC_j$, the latter of which is also a property preserved by the Euler equation. If $\rho$ is not a constant along any $\cC_j$, then in order to recover the velocity, the evolution of either the circulations, or the boundary values of the stream function, {\it etc.} have to be coupled to the evolution of the vorticity. 
\end{remark}

\begin{proof}[Proof of Proposition \ref{P:ED-1}]
To prove the proposition, we may assume $\| \nabla \rho_0\|_{H^k} \le \frac 12$. 
Let $p_0$ be the pressure associated with the equilibrium $(\rho_0, v_0)$ given by \eqref{E:pressure-2} and  $(\omega_0, \vec{m}_0)$ the associated equilibrium vorticity and momenta which satisfies 
\[
\| \omega_0\|_{H^{k+1}}, \, |\vec{m}_0|, \, \|\nabla p_0\|_{H^{k+1}} \le C, 
\]
for some $C>0$ depending on $\Omega$, $k$, $\|v_0\|_{H^{k+2}}$, and $\rho_-$. 
It is clear that the linearization \eqref{linearization} of \eqref{stratified-euler} is conjugate to that of \eqref{E:Euler-V} and thus we shall focus on the latter. The linearization in the vorticity formulation is given by 
\be \label{E:LEuler-V} \begin{split}
\begin{pmatrix} \rho \\ \omega \\ \vec{m} \end{pmatrix}_t = & \begin{pmatrix} - v_0 \cdot \nabla \rho \\ - v_0 \cdot \nabla \omega   \\ 0 \end{pmatrix} 
+ \begin{pmatrix} 0 \\ - v \cdot \nabla \omega_0   \\ 0 \end{pmatrix}
+ \begin{pmatrix} - \nabla \rho_0 \cdot v  \\ - \frac {\nabla p_0}{\rho_0^2} \times \nabla \rho + \frac {1}{\rho_0^2} \nabla \rho_0\times \nabla p - \frac {2 \rho}{\rho_0^3} \nabla \rho_0 \times \nabla p_0 \\ 0 \end{pmatrix} \\ 
\triangleq & \tilde L_1 (\rho, \omega, \vec{m})^T + \tilde L_2 (\rho, \omega, \vec{m})^T + \tilde L_3 (\rho, \omega, \vec{m})^T    \triangleq \tilde L (\rho, \omega, \vec{m})^T, 
\end{split} \ee
where $v$ is  determined by the linearization of \eqref{E:Vor-Vel}
\[
v = \cB(\rho_0, \omega_0, \vec{m}) + \cB(\rho_0, \omega, \vec{m}_0) + \sD_\rho \cB(\rho_0, \omega_0, \vec{m}_0) \rho
\] 
and the linear transformation $(\rho, \omega, \vec{m}) \to \frac 1{\rho_0}\nabla p$ is given by linearizing \eqref{E:pressure-2} or equivalently \eqref{E:LPressure}.

To analyze \eqref{E:LEuler-V},  let 
\[
Y_l = H^{l} (\Omega, \R) \times H^{l-1} (\Omega, \R) \times \R^{d_1}, \;  l \in \mathbb{N}, 
\]
and equip $Y_l$ with the equivalent norm
\[
|(\rho, \omega, \vec{m})^T|_{a, l} \triangleq \big( a^2\|\rho\|_{H^l}^2 + \|\omega\|_{H^{l-1}}^2 + |\vec{m}|^2\big)^{\frac 12}, \quad a = (\eta+ \|\nabla \rho_0\|_{H^{k}})^{-\frac 12}\ge 1, 
\]
where the precise value of $\eta \in (0, \frac 12)$ will be determined later. Here $\eta$ is included solely for the possible case of $\nabla \rho_0=0$. 
From the assumptions, the standard elliptic theory, \eqref{E:Vor-Vel}, \eqref{E:pressure-2}, and \eqref{E:LPressure}, there exists $C>0$ decided by $\Omega$, $k$, $\|v_0\|_{H^{k+2}}$, and $\rho_-$, such that, for $k+1 \ge l \in \mathbb{N}$, we have 
\[
\|v\|_{H^l} \le C( \|\rho\|_{L^1} + \|\omega\|_{H^{l-1} } + | \vec{m}|),  \quad \| \nabla p\|_{H^{l}} \le C( \|\rho\|_{H^l} +  \|\omega\|_{H^{l-1} } + | \vec{m}|).
\]
Subsequently, it is straightforward to obtain for $k \ge l \in \mathbb{N}$, 
\be \label{E:temp-17}
\| \tilde L_2\|_{\cL (Y_{l}, Y_{l+1} )} \le C, \quad \| \tilde L_3 \|_{\cL (Y_{l} )} \le C ( a^{-1} + a \|\nabla \rho_0\|_{H^{k}})\le C a^{-1}.
\ee
Using $u_0(t)$, one may verify 
\be \label{E:temp-17.5}
e^{t\tilde L_1} (\rho, \omega, \vec{m})^T = (\rho, \omega, \vec{m})^T \circ u_0(t)^{-1} \implies \| e^{t\tilde L_1} \|_{\cL(Y_l)} \le C e^{l\mu|t|},
\ee
where \eqref{u0-L-infinity}, \eqref{u0-hl}, and \eqref{E:composition-1} are used and $C$ depends on $\Omega$, $k$, $\|v_0\|_{H^{k+2}}$, $C_*$, $\mu$,  $\mu_0$, and $\rho_-$. Adding the bounded perturbation term $\tilde L_3$, we have 
\[
\|e^{t(\tilde L_1 + \tilde L_3)} \|_{\cL(Y_l)} \le C e^{(l\mu + Ca^{-1} 
) |t|}.
\]
Since $\tilde L_2$ is a compact operator on $Y_l$, by a straightforward argument using the Duhamel's formula, $e^{t\tilde L}$ is a compact perturbation of $e^{t(\tilde L_1 + \tilde L_3)}$. For any 
\be \label{E:temp-18}
\alpha \in (l \mu + C a^{-1} 
, \text{Re}\lambda_0 ), \quad { if } \ \text{Re}\lambda_0 >l \mu + C a^{-1},  
\ee
the spectra set 
\[
\sigma_u \triangleq \{ \lambda \in \sigma (\tilde L) \mid \text{Re}\lambda\ge \alpha\} 
\]
is finite. 
Hence there exist $\lambda_u, \lambda_{cs} \ge 0$ such that   
\[
\sup \{ \text{Re}\lambda \mid \lambda \in \sigma (\tilde L) \setminus \sigma_u\} < \lambda_{cs} <\lambda_u < \inf \{ \text{Re}\lambda \mid \lambda \in \sigma_u\}. 
\]
Let $\Gamma \subset \sigma(\tilde L)^c$ be a smooth contour  such that the intersection of its interior with $\sigma(\tilde L)$ is exactly $\sigma_u$. Define 
\[
\Pi_u = \frac 1{2\pi i} \oint_\Gamma (\lambda I - \tilde L)^{-1} d\lambda, \quad \tilde X_{u, l} = \Pi_u Y_l, \quad \tilde X_{cs, l} = \ker \Pi_u \subset Y_l. 
\]  
From the standard spectral theory and the compactness of $\tilde L_2$ (see for example \cite{ko95}), the following hold for $k \ge l \in \mathbb{N}$.
\begin{itemize} 
\item $\Pi_u \in \cL(Y_l)$, $\Pi_u^2=\Pi_u$, and $\Pi_u \tilde L = \tilde L \Pi_u$, which imply that $\tilde X_{u, l}$ and $\tilde X_{cs, l}$ are closed invariant subspaces of $Y_l$. 
\item $\dim \tilde X_{u, l}< \infty$, and $\tilde X_{u, l} $ consists of eigenfunctions and generalized eigenfunctions of eigenvalues in $\sigma_u$. 
\item There exists $M>0$ such that  
\[
\|e^{t\tilde L} |_{ \tilde X_{u, l}} \|_{\cL( \tilde X_{u, l})} \le M e^{\lambda_u t}, \; \forall t\le 0; \quad \|e^{t\tilde L} |_{ \tilde X_{cs, l}} \|_{\cL( \tilde X_{cs, l})} \le M e^{\lambda_{cs} t}, \; \forall t\ge 0.  
\]
\end{itemize}

Next we show that the eigenfunction $(\rho, v) \in H^1$  of the eigenvalue $\lambda_0$ satisfying  \eqref{E:temp-16.5} actually belongs to $H^{k} (\Omega, \R) \times H_\Euler^{k}$.   
From \eqref{E:LPressure}, we have that the associated linear pressure $p \in H^2(\Omega, \R)$. Clearly the corresponding $(\rho, \omega, \vec{m}) \in Y_1$ and satisfy 
\[
\lambda_0 \begin{pmatrix} \rho \\ \omega \\ \vec{m} \end{pmatrix} = \tilde L \begin{pmatrix} \rho \\ \omega \\ \vec{m} \end{pmatrix} = \begin{pmatrix} - v_0 \cdot \nabla \rho - \nabla \rho_0 \cdot v \\ - v_0 \cdot \nabla \omega  - v \cdot \nabla \omega_0- \frac {\nabla p_0}{\rho_0^2} \times \nabla \rho + \frac {1}{\rho_0^2} \nabla \rho_0\times \nabla p - \frac {2 \rho}{\rho_0^3} \nabla \rho_0 \times \nabla p_0 \\ 0 \end{pmatrix}.
\]
From Re$\lambda_0 > 0$ and \eqref{u0-L-infinity}, it holds  
\[
\vec{m}=0, \quad \lim_{t \to -\infty} e^{\lambda_0 t} \| \rho \circ u_0(t) \|_{L^2}= \lim_{t \to -\infty} e^{\lambda_0 t} \| \omega \circ u_0(t) \|_{L^2} =0.
\] 
Differentiating $e^{\lambda_0 t} \rho \circ u_0(t)$ and $e^{\lambda_0 t} \omega \circ u_0(t)$ (in the distribution sense) and using \eqref{E:LEuler-V}, we obtain  
\be \label{E:temp-18.1}
\begin{pmatrix} \rho \\ \omega \\ \vec{m} \end{pmatrix} - \int_{-\infty}^0 e^{\lambda_0 \tau}  \left(\tilde L_3 \begin{pmatrix} \rho \\ \omega \\ \vec{m} \end{pmatrix} \right) 
\circ u_0(\tau) d\tau = \int_{-\infty}^0 e^{\lambda_0 \tau} \left(\tilde L_2 \begin{pmatrix} \rho \\ \omega \\ \vec{m} \end{pmatrix} \right) 
\circ u_0(\tau) d\tau.
\ee
Since Re$\lambda_0 > 2\mu$, from the above equality, \eqref{u0-L-infinity}, \eqref{u0-hl}, \eqref{E:composition-1}, and \eqref{E:temp-17} we have a priori estimate 
\[
\| (\rho, \omega, \vec{m}) \|_{Y_2} \le \frac C{ \text{Re}\lambda_0 - 2\mu}  \big( \| (\rho, \omega, \vec{m}) \|_{Y_1} + C a^{-1} 
(\rho, \omega, \vec{m}) \|_{Y_2} \big).
\]
If   
\[
C a^{-1}= C \big( \eta+  \|\nabla \rho_0\|_{H^{k}})^{\frac 12} < \text{Re}\lambda_0 - 2\mu,
\]
which is possible when $\ep < \frac 1{2C}$ and \eqref{E:temp-16.5} is satisfied, 
then \eqref{E:temp-18.1} can be viewed as a linear equation of $(\rho, \omega, \vec{m})$ in the form of $(I+ S) (\rho, \omega, \vec{m})= f$ with $\|S\|_{\cL (Y_2)} <1$ and $f \in Y_2$ for $(\rho, \omega, \vec{m}) \in Y_1$. Hence we obtain $(\rho, \omega, \vec{m}) \in Y_2$ which in turn implies $(\rho, \omega, \vec{m}) \in \tilde X_{u, 2}$. 
Repeating this argument we obtain $(\rho, \omega, \vec{m}) \in \tilde X_{u, k}$ if \eqref{E:temp-16.5} holds for some small $\ep>0$. 

The existence of an eigenvalue $\lambda_0$ satisfying \eqref{E:temp-16.5}  for  $\ep = \frac 1{2C}$ implies that there exists $\alpha$ such that \eqref{E:temp-18} holds for $l=k$ and $\sigma_u \ne \emptyset$ as $\lambda_0 \in \sigma_u$. As the linearized Euler equation at $(\rho_0, v_0)$ in the velocity formulation \eqref{linearization} and in the vorticity formulation \eqref{E:LEuler-V} are conjugate, $\tilde X_{u, k}$ and $\tilde X_{cs, k}$  yield the invariant splitting $H^k (\Omega, \R) \times H_\Euler^k =X_u \oplus X_{cs}$ desired in (ED). 
This completes the proof of statement (1). 
\end{proof}

According to the above proposition, to verify the linear exponential dichotomy (ED) of steady states with small density stratification, it suffices to identify unstable eigenvalues whose real parts are large  relative to $\mu$ as given in \eqref{u0-L-infinity} and \eqref{E:Lyap-E}. A typical such situation is when the equilibrium $(\rho_0, v_0)$ is a perturbation to an equilibrium $\bar v$  of the Euler equation with uniform density $\bar \rho=const$ 
\be \label{E:Euler-0}
\bar \rho (v_t + v \cdot \nabla v ) + \nabla p + g \bar \rho \Ve_2=0 \ \text{ in } \ \Omega;  \quad v \cdot \N =0 \ \text{ on }\ \p\Omega. 
\ee  
Its vorticity formulation becomes 
\be \label{E:Euler-0-V}
\omega_t + v \cdot \nabla \omega=0, \quad \vec{m}_t =0, 
\ee
where $v$ is still recovered from the vorticity $\omega$ and the momenta $\vec{m}$ by \eqref{E:Vor-Vel}. 
At an equilibrium $\bar v(x)$ with the associated $\bar \omega$ and $\vec{ \bar m}$, the linearized equation in the vorticity formulation is given by 
\be \label{E:LEuler-0-V} 
\omega_t = \bar L \omega \triangleq  - \bar v \cdot \nabla \omega - v \cdot \nabla \bar \omega. 
\ee
The next proposition states that the existence of relatively strong unstable eigenvalues of equilibria of the homogeneous Euler equation \eqref{E:Euler-0} yields the exponential dichotomy (ED) of nearby steady flows with small density stratification. 

\begin{proposition} \label{P:spectra-1}
Suppose \eqref{E:Omega-2} and the following hold.  
\begin{itemize} 
\item $ \bar v\in H_\Euler^{k+2}$, $k> 2$, is an equilibrium of the Euler equation \eqref{E:Euler-0} with the constant density $\bar \rho>0$ such that its Lagrangian map $\bar u(t)$ (defined in \eqref{lagrangian-coordinate-map-0}) satisfies \eqref{u0-L-infinity} and \eqref{E:Lyap-E} with 
the parameters $\bar \mu$,  $\bar \mu_0$, and $\bar C_*$.
\item The linearized operator $\bar L$ defined in \eqref{E:LEuler-0-V} has an eigenvalue $\lambda_*$ with an eigenfunction $\bar \omega_*  \in L^2(\Omega, \R)$ such that Re$\lambda_*> k \bar \mu$. 
\end{itemize}
Then there exists  $\delta>0$ such that, for any equilibrium $(\rho_0, v_0)\in H^{k+1} (\Omega, \R) \times H_\Euler^{k+2}$ of the stratified Euler equation \eqref{stratified-euler} satisfying 
\be \label{E:temp-18.3}
\|v_0 - \bar v\|_{H^{k+1}}< \delta \min\{1, \text{Re}\lambda_* - k \bar \mu\}, \quad \|\rho_0 - \bar \rho\|_{H^{k+1}}^{\frac 12}  \le \min\{1, \delta (\text{Re}\lambda_* - k \bar \mu)\}, 
\ee
the exponential dichotomy (ED) holds for the linearized operator $L$ of \eqref{stratified-euler}  at $(\rho_0, v_0)$
with $\dim X_u < \infty$ and $\lambda_u > k \bar \mu$. 
\end{proposition} 

\begin{remark} 
According to the above proposition, if the first assumption in \eqref{E:lambda-1} is also satisfied by $(\rho_0, v_0)$, then Theorem \ref{main-theorem} applies.
Notice that the integer $K_0$ in Theorem \ref{main-theorem} is determined only by $k+r$. In concrete applications, one may try to verify \eqref{E:lambda-1} for $(\rho_0, v_0)$ using a perturbation argument assuming that it holds for $\bar v$. 
\end{remark}

In order to prove this proposition, we need the following lemma on convergent sequences of sub-additive functions. 

\begin{lemma} \label{L:sub-add}
Let $f_n(t)$, $t \ge 0$ and $n \in \mathbb{N}$, be a sequence of functions satisfying 
\[
\lim_{n\to \infty} f_n(t) = f_0(t), \quad |f_n(t)| \le C_0 t, \quad f_n(t+t') \le f_n(t)  + f_n (t'), \quad \forall t, t' \ge 0, \; n \in \mathbb{N}, 
\]
for some $C_0 \ge 0$, then for any $\ep>0$,  there exist $N\in \mathbb{N}$ and $\bar C \in \R$ such that 
\[
f_n(t) \le \bar C + (\mu_0 +\ep) t, \; \forall n \ge N, \ t\ge 0, \; \text{ where } \ \mu_0 = \inf_{t>0} \frac 1t f_0(t) \in [-C_0, C_0]. 
\]
\end{lemma} 

\begin{proof} 
Due to the sub-additivity of $f_0(t)$ and Fekete’s lemma, it holds $\lim_{t\to \infty} \frac 1t f_0(t) =\mu_0$. Hence there exists $t_0\ge 1$ such that $\frac 1{t_0} f_0(t_0) < \mu_0 +\frac \ep4$. Since $\lim_{n\to \infty} f_n(t_0) = f_0(t_0)$, there exists $N \in \mathbb{N}$ such that $|f_n(t_0) - f_0(t_0)| \le \frac \ep4 t_0$ for all $n \ge N$. Let $M\in \mathbb{N}$ satisfy $M > \frac {4C_0}\ep$. Therefore, for any $n \ge N$ and $t= mt_0 + t_1$ with $M \le m \in \mathbb{N}$ and $t_1\in [0, t_0)$, we have 
\begin{align*}
f_n(t) \le & m f_n(t_0) + f_n (t_1) \le m (f_n (t_0) - f_0(t_0)) + m f_0(t_0)  + C_0 t_1 \\ 
\le & (m t_0 \ep)/4 + mt_0 (\mu_0+ \ep/4) + (C_0/m) mt_0 \le (\mu_0 +\ep) t. 
\end{align*}
Finally, let $\bar C= C_0 M t_0$ and the lemma follows immediately. 
\end{proof} 

\begin{corollary} \label{C:sub-add}
Let $ \bar v\in H_\Euler^{k}$, $k> 2$, such that its Lagrangian map $\bar u(t)$  satisfies \eqref{u0-L-infinity} and \eqref{E:Lyap-E} with parameters $\bar \mu$, $\bar \mu_0$, and $\bar C_*$, then there exists $\delta, C_* >0$ such that, for any $v \in H_\Euler^{k}$ with $\|v - \bar v\|_{H^k} < \delta$, its Lagrangian map $u(t)$ satisfies \eqref{u0-L-infinity}  and \eqref{E:Lyap-E} with $\mu = \bar \mu > \bar \mu_0$ and $C_*$. Here all Lagrangian maps are defined as in \eqref{lagrangian-coordinate-map-0}. 
\end{corollary}

\begin{proof}
We argue by contradiction. Assume the corollary is false, then there exist sequences $v_n \in H_\Euler^{k}$ and $t_n \in \R$  such that 
\be \label{E:temp-18.5}
\|v_n - \bar v\|_{H^k} + \|D v_n - D \bar v\|_{L^\infty} < 1/n, \quad \|D u_n(t_n)\|_{L^\infty} \ge n e^{\mu t_n}, 
\ee
where $u_n(t)$ are the Lagrangian map of $v_n$. Without loss of generality, assume $t_n \ge 0$. Let $f_n(t) = \log \|D u_n(t)\|_{L^\infty}$ and $f_0(t) = \log \|D \bar u(t)\|_{L^\infty}$. Clearly $f_n$ and $f_0$ satisfy the assumptions of Lemma \ref{L:sub-add} with  $C_0 = \|D \bar v\|_{L^\infty} +1$ and $\mu_0<\bar \mu$ due to \eqref{E:Lyap-E}. Then \eqref{E:temp-18.5} contradicts with Lemma \ref{L:sub-add} with $\ep = \mu - \mu_0$. Hence the corollary holds. 
\end{proof}

With the above preparation, we are ready to prove Proposition \ref{P:spectra-1}.

\begin{proof}[Proof of Proposition \ref{P:spectra-1}]  
By Corollary \ref{C:sub-add}, there exist $\delta, C_*>0$ such that \eqref{u0-L-infinity} and \eqref{E:Lyap-E} hold with $\bar \mu$ and $C_*$ for the Lagrangian map $u_0(t)$ (as defined in \eqref{lagrangian-coordinate-map-0}) of any $v_0$ satisfying \eqref{E:temp-18.3}. 


Let $(\omega_0, \vec{m}_0)$ and $(\bar \omega, \vec{\bar m})$ be the vorticity and momenta of the equilibrium $(\rho_0, v_0)$ and $(\bar \rho, \bar v)$, respectively, and $\tilde L = \tilde L_1 + \tilde L_2 + \tilde L_3$ be the linear operator of the stratified Euler equation in the vorticity formulation at $(\rho_0, \omega_0, \vec{m}_0)$, where the splitting of $\tilde L$ is same as in \eqref{E:LEuler-V}. Clearly $\tilde L$ is conjugate to $L$ in the Eulerian formulation. 
To analyze the resolvent of $\tilde L$, for $\lambda \in \mathbb{C}$ and $(\tilde \rho, \tilde \omega, \vec{\tilde m})^T$, 
\be \label{E:temp-19}
(\lambda - \tilde L) \begin{pmatrix} \rho \\ \omega \\ \vec{m} \end{pmatrix} = \begin{pmatrix} \tilde \rho \\ \tilde \omega \\  \vec{\tilde m} \end{pmatrix} \Longleftrightarrow 
\begin{pmatrix} \lambda \rho + v_0 \cdot \nabla \rho 
\\ \lambda \omega + v_0 \cdot \nabla \omega 
\\ \lambda \vec{ m} \end{pmatrix} = \big(\tilde L_2 + \tilde L_3 \big) \begin{pmatrix} \rho \\ \omega \\ \vec{m} \end{pmatrix} +\begin{pmatrix} \tilde \rho \\ \tilde \omega \\  \vec{\tilde m} \end{pmatrix}.
\ee
We proceed as in the proof of Proposition \ref{P:ED-1} using the same spaces $Y_l$, equipped with the norm $|\cdot|_{a, l}$ where $a =(\eta+ \|\nabla \rho_0\|_{H^k})^{-\frac 12} \ge 1$. If Re$\lambda > k\bar \mu$, then \eqref{E:temp-19} is equivalent to 
\be \label{E:temp-20}
( \rho, \omega, \vec{m})^T = \int_{-\infty}^0 e^{\lambda \tau} \Big( ( \tilde L_2+ \tilde L_3)( \rho, \omega, \vec{m})^T + (\tilde \rho, \tilde \omega, \vec{\tilde m})^T \Big) \circ u_0 (\tau) d\tau, 
\ee
where $\vec{m} =\frac 1\lambda \vec{\tilde m}$ in particular. We denote \eqref{E:temp-20} symbolically as 
\be \label{E:temp-20.5}
(I- A(\lambda)) ( \rho, \omega, \vec{m})^T =  B(\lambda) (\tilde \rho, \tilde \omega, \vec{\tilde m})^T. 
\ee
If Re$\lambda > k \mu$ and $k \ge l \in \mathbb{N}$, then $A, B \in \cL(Y_l)$.  To single out the principal part of $A$, let 
\[
\bar A (\lambda) \omega  = - \int_{-\infty}^0 e^{\lambda \tau} (v\cdot \nabla \bar \omega) \circ \bar u (\tau) d\tau, \quad A_0 (\lambda) \omega  = - \int_{-\infty}^0 e^{\lambda \tau} (v\cdot \nabla \bar \omega) \circ u_0 (\tau) d\tau, 
\]
which, due to \eqref{u0-L-infinity}, \eqref{u0-hl}, \eqref{E:composition-1}, and \eqref{E:temp-17}, satisfy 
\be \label{E:temp-21}
\|\bar A(\lambda)\|_{\cL(H^{l-1} (\Omega, \R), H^{l} (\Omega, \R))}, \, \|A_0(\lambda)\|_{\cL(H^{l-1} (\Omega, \R), H^{l} (\Omega, \R))} \le C/(\text{Re}\lambda - l\bar \mu), \quad 1 \le l \le k,
\ee
\be \label{E:temp-22}
\| A (\lambda)- diag (0, A_0(\lambda), 0)\|_{\cL (Y_l)} \le C (\|v_0 - \bar v\|_{H^{k+1}} + a^{-1} 
) / (\text{Re}\lambda - l\bar \mu ), \quad 1\le l \le k. 
\ee
To estimate $A_0(\lambda) - \bar A(\lambda)$, let $u(s, t, \cdot)$ be the Lagrangian map as defined in \eqref{lagrangian-coordinate-map-0} associated with $v_*(s, \cdot) = s v_0 + (1-s) \bar v$. Linearizing in $s$, one may deduce  
\begin{align*}
\p_{st} u = ((D v_*) \circ u )\p_s u + (v_0 - \bar v) \circ u  \implies & \p_s u (s, t) = \int_0^t \big( (D u (s, t-\tau)) ( v_0 - \bar v) \big) \circ u(s, \tau)  d\tau\\
\implies & \|\p_s u\|_{L^\infty} \le C e^{- \bar \mu t} \|v_0 - \bar v\|_{H^{k}},  
\end{align*}
where Corollary \ref{C:sub-add} was also applied to $v(s, \cdot)$. Consequently, 
\begin{align}
& (A_0 (\lambda)- \bar A(\lambda)) \omega = \int_0^1 \int_{-\infty}^0 e^{\lambda \tau} \p_s u (s, \tau) \cdot \big( (\nabla (v\cdot \nabla \bar \omega)) \circ u (s, \tau) \big) d\tau ds  \notag \\
\implies & \|(A_0(\lambda) - \bar A(\lambda)) \|_{\cL(L^2(\Omega, \R))}  \le \frac C{\text{Re}\lambda - \bar \mu} \|v_0 - \bar v\|_{H^{k}}  \notag \\
\implies  &  \| A (\lambda)- diag (0, \bar A(\lambda), 0)\|_{\cL (Y_1)} \le \frac C{\text{Re}\lambda - \bar \mu}  \big(
\|v_0 - \bar v\|_{H^{k}} + (\eta+ \|\nabla \rho_0\|_{H^k})^{\frac 12} \big) \le C \delta \label{E:temp-23},
\end{align}
where $\eta$ is chosen to be sufficiently small and \eqref{E:temp-18.3} is used.

From the definition of $\bar A(\lambda)$, we notice that $(I- \bar A(\lambda))^{-1} \in \cL(L^2 (\Omega, \R))$ iff $\lambda \notin \sigma (\bar L)$, where we recall that $\bar L$ defined in \eqref{E:LEuler-0-V} is the linearized operator of the Euler equation \eqref{E:Euler-0-V} at $\bar \omega$. Apparently $\bar L$ is a compact perturbation of $\omega \to -\bar v \cdot \nabla \omega$ which generates a strongly $C^0$ group of unitary operators on $L^2 (\Omega, \R)$. 
Therefore, for any $\ep>0$, the set $\{ \lambda \in \sigma(\bar L): \text{Re}\lambda \ge \ep\}$ is finite and their eigenspaces have a finite dimension in total. Moreover $I- \bar A(\lambda) \in \cL(L^2 (\Omega, \R))$ is Fredholm with index zero due to the compactness of $\bar A(\lambda)$ on $L^2 (\Omega, \R)$. From its definition, $\bar A(\lambda) \in \cL(L^2 (\Omega, \R), H^1 (\Omega, \R))$ depends on $\lambda \in \mathbb{C}$ analytically if Re$\lambda > \bar \mu$ due to the volume preservation of $\bar u(t)$. Using \eqref{E:temp-23}, the assumption $\lambda_* \in \sigma(\bar L)$ with Re$\lambda_*>k \bar \mu$, Fredholm property of $I- \bar A(\lambda)$, Lyapunov-Schmidt reduction, and a continuation argument of zeros of complex analytic functions  (see, for example, \cite{ko95}), we have that, for $\delta\ll1$, there exists $\lambda_0$ near $\lambda_*$ such that $\ker (I-A(\lambda_0))$ is non-trivial. Therefore $\lambda_0 \in \sigma(\tilde L)$ and Re$\lambda_0 > k \mu$ and thus the assumptions of Proposition \ref{P:ED-1} are satisfied which yields the desired exponential dichotomy of $L$ at $(\rho_0, v_0)$.    
\end{proof} 

We often encounter two classes of equilibria of the stratified Euler equation \eqref{stratified-euler} which are close to those of the homogeneous Euler equation \eqref{E:Euler-0}. 

\vspace{0.05in}
\noindent $\bullet$ {\it Shear flow:} 
\be \label{E:shear-1}
\big(\rho_0=\rho_0(x_2), \ v_0= U(x_2) \Ve_1 \big) \; \text{ on } \; p_0 = p_0(x_2) = - g \int \rho_0 dx_2, \; \Omega = \BBT \times (-b, b),
\ \text{ or } \ \Omega = \BBT^2. 
\ee
for any $\rho_0>0$ and $U$. Here one recalls that $g=0$ is assumed in \eqref{E:g} if $\Omega=\BBT^2$. In particular the first assumption $\lambda_u > \lambda_{cs} + 3K_0 \mu$ in \eqref{E:lambda-1} is ensured  automatically by $\lambda_u > \lambda_{cs}$ since $\mu$ can be chosen as any positive quantity for shear flows. Hence Theorem \ref{main-theorem} applies to any  unstable shear flow satisfying \eqref{E:steady} with small density stratification $\nabla \rho_0$.    

\vspace{0.05in}
\noindent $\bullet$ {\it From nonlinear elliptic equation.}  Here we assume $g=0$ and, for simplicity, $\Omega$ is bounded and simply connected. Suppose $f, \gamma \in C^{k} (\R, \R)$ and $\psi: \Omega \to \R$ such that $f>0$ and satisfy 
\be \label{E:steady-1}
\Delta \psi + \frac {f'(\psi)}{2f(\psi)} |\nabla \psi|^2 = \gamma (\psi), \ \text{ in } \ \Omega; \quad \psi|_{\p \Omega}=0. 
\ee
Let 
\be \label{E:steady-3}
v_0 = J \nabla \psi, \quad \rho_0 = f(\psi), \; \text{ where } \; J = \begin{pmatrix} 0 & -1 \\ 1 & 0 \end{pmatrix}, 
\ee
then the corresponding vorticity $\omega_0 = \Delta \psi$. One may deduce from \eqref{E:steady-1} 
\begin{align*}
& J \nabla \psi \cdot \nabla \big( f(\psi) \Delta \psi + \tfrac 12 f'(\psi) |\nabla \psi|^2 \big)=0 \\ 
\implies & f(\psi)J \nabla \psi \cdot \nabla \Delta \psi + J f'(\psi) \nabla \psi \cdot ((J \nabla \psi \cdot \nabla) J \nabla \psi) =0 \\
\implies & \rho_0 v_0 \cdot \nabla \omega_0 + J \nabla \rho_0 \cdot (v_0 \cdot \nabla v_0) =0 , 
\end{align*}
which is the steady vorticity equation of the stratified Euler equation \eqref{stratified-euler}. The steady density equation is obviously satisfied by $\rho_0$ and $v_0$ and thus $(\rho_0, v_0)$ is a steady state of \eqref{stratified-euler}. Suppose $\psi_*$ is a non-degenerate solution to \eqref{E:steady-1} with $f= a= const$. Here the non-degeneracy means that the linearized operator $\Delta - \gamma'(\psi_*)$ has a bounded inverse. By an Implicit Function Theorem argument, for any $f(s) = a+ f_1(s)$ with $\|f_1\|_{C^{k}} \ll 1$, \eqref{E:steady-1} has a unique solution $\psi$ and thus yields a steady solution $(\rho_0, v_0)$ satisfying $\| \rho_0-a\|_{H^k} \ll 1$. 

In both cases, the exponential dichotomy of steady solutions to the homogeneous Euler equation \eqref{E:Euler-0} implies that of the stratified Euler equation \eqref{stratified-euler}. There is a rich literature on studies of the stability of steady solution to \eqref{E:Euler-0}, see, e.~g.~\cite{To35, LinC55, DR04, MP94, FH98, lin03, Lin05}. 

\begin{remark} 
The analysis in this section takes advantage of the vorticity formulation available in two dimensions. In particular the compactness of  the term $v \cdot \nabla \omega_0$ as an operator on $\omega$ in the linearization \eqref{E:LEuler-V} is very helpful, while the term $v \cdot \nabla v_0$ in \eqref{linearization} is only a bounded operator on $v$. 
\end{remark}

\subsection{Unstable Manifolds of a class of 2-dim Stratified Shear Flows on $\Omega= \BBT \times \R$}

In this subsection, we consider the unstable manifolds and nonlinear instability of a shear flow 
\be \label{E:steady-2}
\big(\rho_0(x_2), v_0= \ep U(x_2) \Ve_1 \big), \; \text{ along with } \ p_0 = p_0(x_2) = - g \int \rho_0 dx_2, \; g>0,  \; \text{ on } \ \Omega = \mathbb T \times \mathbb R,
\ee
satisfying \eqref{E:steady}. The Lagrangian maps of shear flows do not exhibit any exponential growth, so \eqref{u0-L-infinity} and \eqref{E:Lyap-E} are satisfied with any $\mu >0$. In order to obtain the local unstable manifold, one only needs to verify the exponential dichotomy (ED) of the linearized system.
We shall discuss the classical Rayleigh-Taylor instability problem ($\rho_0' > 0$) with a small background shear flow ($\ep \ll 1$) based on the linear analysis of the instability and essential spectra by the geometric optics method {\it etc.} (see, e.g., \cite{Vishik1996SpectrumOS,Shvydkoy-Latushkin,FV91,essential-spectrum,roman}).

\begin{remark}[Characterization of Spectra]
For any densely defined closed linear operator $A$ on Banach spaces, its spectrum is often decomposed into two parts: $\sigma(A) = \sigma_{d}(A) \cup \sigma_{ess}(A)$. The first part represents the discrete spectrum, which consists of isolated eigenvalues with finite algebraic multiplicity,
and the second part the essential spectrum is the complement. 
Another common definition of the essential spectrum is characterized by Fredholm operators, that is, $\sigma^F_{ess}(A)=\{\lambda \in \mathbb C: A - \lambda \text{ is not Fredholm}\}$. By Nussbaum's formula in \cite{nussbaum}, the essential spectral radii associated with these two definitions coincide for bounded linear operators. 
\end{remark}

\begin{proposition} \label{rayleigh-taylor-infinite-cylinder}
Consider shear flows given in \eqref{E:steady-2} satisfying \eqref{E:steady}. 
Assume 
\[ 
U',U'',\rho_0',\rho_0'' \to 0 \text{ as } |x_2| \to \infty,\; \lim_{x_2 \to \infty} \rho_0 (x_2) = \rho_\pm>0, \ \text{ and } \ U'(x_2) \neq 0,  \; \rho_0'(x_2) > 0, \; \forall x_2 \in \mathbb R,
\]
then there exists $\varepsilon_0 > 0$ such that for any $0<\varepsilon<\varepsilon_0$, the shear flow $(\rho_0(x_2), \ep U(x_2))$  satisfies the exponential dichotomy (ED) with $0< \dim X_u < \infty$ and has a local unstable manifold as provided in Theorem \ref{main-theorem}. 
\end{proposition}

\begin{proof}
By Theorem 3.7 in \cite{roman}, fix a wave number $m \in \mathbb{N}$ of $x_1$, there exists $\varepsilon_0 > 0$ such that for any $0<\varepsilon<\varepsilon_0$, the linearized  Euler operator $L$ defined in \eqref{linearization} corresponding to the steady profile $(\rho_0, \varepsilon U)$ has at least one unstable eigenvalue $\lambda_0 = -imc$, Im$c > 0$, associated with an eigenfunction $(\rho, v)$ of the linearized Euler equation \eqref{linearization} in the form of 
\[
\rho =  e^{im x_1}\frac{\rho_0' \psi }{\ep U-c}, \quad  v= e^{im x_1}\big(  \psi' (x_2), -im \psi(x_2) \big),
\]
where $\psi =\psi(x_2) \in H^{k+r}$ is the solution to the Rayleigh-type system. 

Following the notations in \cite{roman}, for any $l \ge 0$, let $G_t \triangleq e^{tL}$ be the strongly continuous group on $H^l(\Omega,\mathbb R) \times H^l_\Euler$ generated by the linearization $L$. 
Let $r_{ess}(G_t,H^l)$ denote the essential spectral radius under $H^l$ norm. According to Corollary 3.6 and (66) in \cite{roman},
\begin{equation}
\label{essential-radius}
r_{ess}(G_t,H^l) = e^{t\nu}, \text{ where } \nu = \max\left(0,\;\sup_{\{U'(x_2) = 0,\; \rho_0'(x_2) > 0\}} \sqrt{\frac{g \rho_0'(x_2)}{\rho_0(x_2)}}\right), \quad \forall t\ge 0.
\end{equation}
In particular, since  $\{x_2 \in \mathbb R: U'(x_2) =0, \; \rho_0'(x_2) > 0 \}$ is empty by our assumption, we have $\nu = 0$ and $r_{ess}(G_t, H^k) = 1$.  Therefore,  only finitely many isolated eigenvalues of $G_t$ with finite total algebraic multiplicity exist outside the closed disk $\overline{B_R (\mathbb C)}$ for any $R>1$. Recalling $|e^{\lambda_0}| >1$, we can choose $\lambda_u> \log R > \lambda_{cs} > 0$ such that   
\[
R \in (1, |e^{\lambda_0}|) \setminus \{ |\alpha| : \alpha \in \sigma(G_1) \}, \quad e^{\lambda_u} \in \big(R, \min_{\alpha \in \sigma_u}  |\alpha| \big), \quad e^{\lambda_{cs}} \in \big(\max_{ \alpha \in \sigma_{cs}} |\alpha|, R \big),
\]
where 
\[
\sigma_u = \{ \alpha \in \sigma(G_1) : |\alpha| > R\}, \quad \sigma_{cs} = \sigma(G_1) \setminus \sigma_u. 
\] 
Define the spectral projection and the eigenspaces 
\[
\Pi_{cs} \triangleq \frac{1}{2 \pi i }\oint_{\p B_R(\mathbb{C})} (\lambda I - G_1)^{-1}d\lambda, 
\]
\[ 
X_{cs} = \Pi_{cs} (H^k (\Omega, \R) \times H_\Euler^k), \quad X_{u} = (I-\Pi_{cs}) (H^k (\Omega, \R) \times H_\Euler^k).  
\]
In particular $0< \dim X_u < \infty$ since $e^{\lambda_0} \in \sigma_u$. The finite dimensionality and the smoothness of generalized eigenfunctions imply that $X_u$ is closed in $H^l (\Omega, \R) \times H_\Euler^l$ for all $0 \le l \le k$. 
From the standard spectral theory, $X_u$ and $X_{cs}$ are invariant subspaces of $G_1$. 
In fact, the commutativity between $G_1$ and $G_t$ for any $t \in \R$ yields the commutativity $G_t \Pi_{cs}=\Pi_{cs} G_t$, which implies that $X_{cs}$ and $X_u$ are also invariant under $G_t$ for any $t$. 
On the one hand, $G_1|_{X_u}$ is conjugated to a finite dimensional matrix with the spectra $\sigma_u$. Hence it is straightforward to verify  
\[
\|G_n^{-1} \|_{\cL(X_u)} \le C e^{- \lambda_u n}, \quad \forall n \in \mathbb{N}. 
\]
On the other hand, 
\[
G_n|_{X_{cs}} =  \frac{1}{2 \pi i }\oint_{\p B_R(\mathbb{C})} \lambda^n (\lambda I - G_1)^{-1}d\lambda \implies 
\|G_n\|_{\cL(X_{cs})} \le C  e^{\lambda_{cs} n}, \quad \forall n \in \mathbb{N}. 
\]
Since $\|G_t\|_{\cL(H^k)}$ is uniformly bounded for $t \in [-1, 1]$, we obtain the dichotomy estimates
\be \label{E:temp-24}
\|e^{-tL} \|_{\cL(X_u)}  \le C e^{- \lambda_u t}, \quad \|e^{tL} \|_{\cL(X_{cs})}  \le C e^{\lambda_{cs} t}, \quad \forall t\ge 0.
\ee
Hence the exponential dichotomy (ED) is satisfied by the above $X_{u, cs}$ and $\lambda_{u, cs}$. 


Since $v_0$ is a shear flow, the exponential growth rate assumptions \eqref{u0-L-infinity} and \eqref{E:Lyap-E} of the Lagrangian coordinate map $u_0(t)$ are satisfied with any $\mu >0$. Therefore the proposition follows directly from Theorem \ref{main-theorem}. 
\end{proof}

A concrete example of a stratified shear flow $(U_0(x_2),\rho_0(x_2))$ satisfying the assumptions in Proposition \ref{rayleigh-taylor-infinite-cylinder}  is, for any $\rho_+ > \rho_->0$,  
\[
U(x_2) = \tanh(x_2), \quad \rho_0(x_2) = \frac 12(\rho_{+} - \rho_{-}) \tanh(x_2) + \frac12 (\rho_{+} + \rho_{-}).
\]

\section*{Acknowledgements}

ZL is supported in part by NSF of China under Grant No. 12494544. CZ is supported in part by US NSF grant  DMS-2350115. ZL and CZ would also like to thank Roman Shvydkoy for many fruitful conversations.

\bibliographystyle{abbrv}
\bibliography{Bibliography}
\end{document}